\documentclass[runningheads,a4paper,12pt,twoside,oribibl,envcountsame]{article} 
\usepackage[margin=1.05in]{geometry}
\usepackage[mathletters]{ucs} 
\DeclareUnicodeCharacter{120016}{\mathcal{A}}
\DeclareUnicodeCharacter{120017}{\mathcal{B}}
\DeclareUnicodeCharacter{120018}{\mathcal{C}}
\DeclareUnicodeCharacter{120019}{\mathcal{D}}
\DeclareUnicodeCharacter{120020}{\mathcal{E}}
\DeclareUnicodeCharacter{120021}{\mathcal{F}}
\DeclareUnicodeCharacter{120022}{\mathcal{G}}
\DeclareUnicodeCharacter{120023}{\mathcal{H}}
\DeclareUnicodeCharacter{120024}{\mathcal{I}}
\DeclareUnicodeCharacter{120025}{\mathcal{J}}
\DeclareUnicodeCharacter{120026}{\mathcal{K}}
\DeclareUnicodeCharacter{120027}{\mathcal{L}}
\DeclareUnicodeCharacter{120028}{\mathcal{M}}
\DeclareUnicodeCharacter{120029}{\mathcal{N}}
\DeclareUnicodeCharacter{120030}{\mathcal{O}}
\DeclareUnicodeCharacter{120031}{\mathcal{P}}
\DeclareUnicodeCharacter{120032}{\mathcal{Q}}
\DeclareUnicodeCharacter{120033}{\mathcal{R}}
\DeclareUnicodeCharacter{120034}{\mathcal{S}}
\DeclareUnicodeCharacter{120035}{\mathcal{T}}
\DeclareUnicodeCharacter{120036}{\mathcal{U}}
\DeclareUnicodeCharacter{120037}{\mathcal{V}}
\DeclareUnicodeCharacter{120038}{\mathcal{W}}
\DeclareUnicodeCharacter{120039}{\mathcal{X}}
\DeclareUnicodeCharacter{120040}{\mathcal{Y}}
\DeclareUnicodeCharacter{120041}{\mathcal{Z}}
\usepackage[utf8x]{inputenc} 
\usepackage{latexsym,amsmath,amssymb,bm}
\usepackage{hyperref} 
\makeatletter\def\HyPsd@expand@utfvii{}\makeatother 
\usepackage{graphicx} 
\usepackage{amsthm} 
\usepackage{cleveref}
\usepackage[dvipsnames]{xcolor} 
\usepackage{stmaryrd}

\def\hrefurl#1{\href{#1}{\rule{0ex}{1.7ex}\color{blue}\underline{\smash{#1}}}} 
\def\fillpagebreak#1{\par\vspace{0pt plus #1\textheight}\pagebreak[3]\vspace{0pt plus -#1\textheight}} 
\def\nq{\hspace{-1em}}          
\def\fr#1#2{{\textstyle\frac{#1}{#2}}} 
\def\frs#1#2{{^{#1}\!/\!_{#2}}} 
\def\trp{{\!\top\!}}            
\def\nq{\hspace{-1em}}          
\def\eoe{\hspace*{\fill} $\diamondsuit\quad$\ntc{\\}} 

\def\qedeq{\rule{1.4ex}{1.4ex}} 
\def\qed{\hspace*{\fill}\rule{1.4ex}{1.4ex}$\quad$\\} 
\def\citep{\cite}
\def\citet{\cite}
\def\bfm{\bfseries\boldmath} 
\newenvironment{keywords}{\centerline{\bf\small Keywords}\begin{quote}\small}{\par\end{quote}\vskip 1ex}
\newtheorem{theorem}{Theorem}
\newtheorem{definition}[theorem]{Definition}
\newtheorem{proposition}[theorem]{Proposition}

\newtheorem{lemma}[theorem]{Lemma}
\newtheorem{claim}[theorem]{Claim}
\newtheorem{observation}[theorem]{Observation}
\newtheorem{example}[theorem]{Example}
\renewenvironment{proof}{{\noindent\bf Proof.}}{\vskip 1ex}
\def\tc#1{}\def\ntc#1{#1} 
\def\,{\mskip 3mu} \def\>{\mskip 4mu plus 2mu minus 4mu} \def\;{\mskip 5mu plus 5mu} \def\!{\mskip-3mu}
\def\v{\boldsymbol}
\def\vt{{\boldsymbol\theta}}
\def\vl{{\boldsymbol\lambda}}
\def\Var{{\mathbb{V}}}
\def\Cov{{\rm Cov}}
\def\W{{\mathbb W}}
\def\Hiid{H_{\rm iid}}
\def\Halt{H_{\overline{\rm iid}}}
\def\smax{10'000}
\def\simequd{\mathrel{\rotatebox[origin=c]{180}{$\simeq$}}}

\begin{document}

\title{\vspace{-4ex}
\normalsize\sc \vskip 2mm\bf\Large\hrule height5pt \vskip 4mm 
Testing Independence of \\ Exchangeable Random Variables
\vskip 4mm \hrule height2pt
}
\author{{\bf Marcus Hutter}\\[3mm]
  \normalsize DeepMind\\
  \small Latest version \& more @ \\ \small \hrefurl{http://www.hutter1.net/official/bib.htm\#exiid}}
\date{22 October 2022}
\maketitle

\begin{abstract}%
Given well-shuffled data, can we determine whether the data items are statistically (in)dependent?
Formally, we consider the problem of testing whether a set of exchangeable random variables are independent.
We will show that this is possible and develop tests that can confidently reject the null hypothesis that data is 
independent and identically distributed and have high power for (some) exchangeable distributions.
We will make no structural assumptions on the underlying sample space.
One potential application is in Deep Learning, 
where data is often scraped from the whole internet, with duplications abound, 
which can render data non-iid and test-set evaluation prone to give wrong answers.

\vspace{3ex}\def\contentsname{\centering\normalsize Contents}\setcounter{tocdepth}{1}
{\parskip=-2.9ex\tableofcontents}
\end{abstract}


\newpage
\section{Introduction}\label{sec:intro}

We consider the problem of testing whether a set of exchangeable random variables $X_1,...,X_n$
are independent, solely from observations $x_{1:n}:=x_1 x_2 ... x_n$ sampled from $Q$.
A distribution $Q(x_1,...,x_n)$ is called (finitely) exchangeable if it is invariant under all (finite) permutations of its argument.
We make no structural assumptions on the underlying probability space ($𝓧^n,Σ,Q$) beyond $Q$ being exchangeable,
and of course that $𝓧⊇\{x_1,...,x_n\}$.
Less formally, assume we have observed $x_{1:n}$, which we believe to be well-shuffled,
and want to know whether they originated from some iid distribution $P_\vt$.
The shuffling implies that the $X_t$ are identically distributed, 
but it does not make them independent.%

A priori one may think this is a hopeless problem.
For instance, if we remove the `identically distributed' condition,
every $x_{1:n}$ is independent w.r.t.\ \emph{some} non-iid distribution.
One can always take $P[X_t=x_t]=1~∀t$,
i.e.\ no valid test can reject the hypothesis that $x_{1:n}$ are independent,
unless one makes some further assumptions on $P$.

The exchangeability assumption implies that the only useful information in $x_{1:n}$ is the counts $n_x:=|\{x_t:x_t=x\}|$ of each $x∈𝓧$.
They form a minimal sufficient statistic.
Due to the assumed lack of structure in $𝓧$,
the specific label $x$ also bears no information,
so we may as well injectively map each $x_t$ to a label from $\{1,...,d''\}$,
where $d''∈ℕ$ is the number of \emph{different} $x_t$ in $x_{1:n}$,
and can even sort them e.g.\ w.r.t.\ decreasing $n_x$.
That means, the only useful information in $x_{1:n}$ is actually the second-order counts $m_k:=|\{x:n_x=k\}|$.
All this will be made clear later.
We are primarily interested in the case of low duplicity, 
i.e.\ most $n_x$ are small, though our results are general.

\paragraph{Contents.}
We provide some motivating examples in \Cref{sec:app},
which also informally show that the second-order counts $m_k$ 
can indeed sometimes reveal that $x_{1:n}$ did not come from an iid process.
A practical example is data duplication in machine learning and the test set contamination problem it results in.
In \Cref{sec:Form} we introduce first-order counts $n_x$ and second-order counts $m_k$, 
exchangeable distributions, 
and the nature of statistical tests in this context for finite and countable $𝓧$.
In particular we reduce iid distributions to multinomial distributions,
and then for our tests to a mixture of Poisson distributions. 
We then show in \Cref{sec:XtoN} that we can reduce every $𝓧$ whatsoever
to $𝓧=ℕ$ or $𝓧=\{1:d\}$ with discrete $σ$-algebra $Σ=2^𝓧$.
After this lengthy preparation, we are finally able to develop our statistical tests in \Cref{sec:iidtests}.
The tests we consider are based on the observation that a mixture of Poisson distributions is ``smooth'',
so if $m_k$ as a function of $k$ is not sufficiently smooth,
this can be used as evidence for dependence. 
The tests are summarized in \Cref{thm:iidtests}.
We experimentally verify our tests in \Cref{sec:exp} on artificially generated data.
In \Cref{sec:moretests} we give an outlook on alternative ways of deriving iid tests for exchangeable data.
\Cref{sec:disc} concludes.

In \Cref{sec:Lemmas}, we state/derive a number of technical lemmas we require to derive our tests.
For improving the power of our tests, in \Cref{secm:iidtests} we derive upper bounds on our test statistics analogous to \Cref{sec:iidtests} 
but without the Poisson approximation, i.e.\ directly for iid $𝓧$ or the multinomial distribution.
Details on the multinomial and product of Poisson distributions and their relation can be found in \Cref{secm:Lemmas}.
They are used to derive upper bounds on the variance of our tests in the multinomial model.
A list of notation can be found in \Cref{sec:Notation}.

\paragraph{Unrelated work.}
Independence tests in the literature most often refer to testing 
whether a pair of random variables $(X,Y)$ is independent,
given a number of iid(!)\ sample pairs $\{(x_t,y_t)\}$
(mutual information and chi-square tests are popular).
Our setup is totally different and much harder.

Another setup is stochastic processes.
Dependence can be tested via estimating auto-correlation coefficients, 
but this requires ordered data and $𝓧=ℝ$.
One could use some other independence test on the pairs $\{(x_{t-1},x_t)\}$ 
without the $𝓧=ℝ$ assumption to test a Markov vs iid hypothesis,
and/or adapt auto-correlation tests to unordered data.
We briefly remark on this in \Cref{sec:moretests}.

\section{Examples \& Potential Applications}\label{sec:app}

In this section we will provide some motivating examples and potential applications.
This will also provide some intuition why rejecting the hypothesis $\Hiid$ that data is iid is possible at all,
but also the difficulty from not having any more structure available.
We consider biased coin flips (binomial process), Black Jack, and data duplication.
We also discuss the relevance to machine learning, 
whose dominant training paradigm still operates under the iid assumption.

\paragraph{Binomial.} 
Consider a binary sequence $x_{1:n}=x_1 x_2 ... x_n∈\{0,1\}$ of length $n=1000$, say.
If $x_{1:n/2}=1^{n/2}$ and $x_{n/2+1:n}=0^{n/2}$ 
we confidently reject the hypothesis $\Hiid$ that $x_{1:n}$ was sampled i.i.d.
But we are unlikely to observe such a sequence
if $x_{1:n}$ is well-shuffled (sampled from an exchangeable process).
If we shuffle 500 ones and 500 zeros,
a typical $x_{1:n}=0111100101...0100100100$ looks random,
or does it? There are exactly $n_0=n_1=n/2=500$ ones and zeros.
While for a fair coin, we expect \emph{about} $n/2$ ones,
would or should you believe anyone telling you that this is a sequence of fair coin flips?
The probability of observing \emph{exactly} $n/2$ ones in a sequence of $n$ fair coin flips 
is around $\sqrt{2/πn}=2.5\%$, so a test for $N_1\smash{\stackrel?=}n/2$ 
would confidently reject the hypothesis that the $x_{1:n}$ above arose from a fair coin.

What about $n=1'000'000$ and $n_1=314'159$.
Obviously this is not from a fair coin,
but our aim is to test for iid, not fairness.
Could such a sequence have been the result of a biased coin?
Since we assume the bits to be perfectly shuffled,
$n_1$ is a sufficient statistic,
so any test plausibly should only depend on $n_1$,
and not the sequence $x_{1:n}$ itself.
The probability that $N_1=n_1$ is $≤0.00086$ for a coin of any bias, 
i.e.\ test $\smash{N_1\stackrel?=n_1}$ would reject $\Hiid$.
Of course, tests have to be designed before observing the data,
and a-priori $n_1=314'159$ is unlikely,
so such a test is unlikely to have any power.
We can of course combine tests and apply a union bound,
but not too many, otherwise the tests become too weak.
$n/2$ seems special, so maybe we should put such a test in the mix,
but what about $314'159$? It's the first 6 digits of $\pi$.
Maybe this is too much numerology, 
but what about testing for prime $n_1$?
The density of primes $p$ is around $1/\ln p$,
so a-priori we should expect an $n_1$ around $3⋅10^5$ to be composite (with confidence $1-1/\ln(n_1)\dot=93\%$).
Maybe this is just not enough to reject $\Hiid$,
but we could always up the numbers.

Imagine $n$ so large that the binary representation of $n_1$ contains
some encrypted message or a long segment of Chaitin's number of wisdom.
In general, finding every pattern in $n_1$ is an AI-complete problem.
There are universal tests which in principle could test for all such eventualities,
and in some situations practical approximations thereof can be very powerful.
We discuss them briefly in \Cref{sec:moretests},
but we were not able to make them work as well as the specific tests develop in this paper,
so will not consider universal tests any further
(nor will we delve into numerology any further).

\paragraph{Black Jack.}
A standard deck of cards without Jokers consists of 52 cards of 13 ranks, each in four suits, two red and two black.
If we shuffle together infinitely many such decks and then draw $n$ cards,
this equivalently to drawing cards uniformly iid from the $|𝓧|=52$ different card faces.
If we have only one deck and draw all 52 cards from it,
we obviously observe every card exactly once.
Such an outcome would be extremely unlikely had we drawn 52 cards 
from an infinite set of decks (see \eqref{eq:once}).

Assume now an unknown number of decks have been shuffled together.
Assume $n$ cards have been drawn so far from this pile and we remembered their face (called card counting strategy).
An interesting question is to infer the number of decks the cards have been drawn from.
Or consider the weaker question: Are $x_{1:n}$ consistent with $\Hiid$?
If yes, we cannot infer the next card face better than by chance (1/52),
so should not waste our time trying to do so, 
and wait with raising the stakes for when we have seen more cards.
The answer to this question is relevant even if we know the number of decks $c$.
For Black Jack, 1-8 decks are used, in casinos often 6.
If $n$ is small, we will not be able to reject $\Hiid$,
but if we have seen all cards, then each card will appear exactly $c$ times,
again ruling out $\Hiid$. If we are close to the end of the pile,
most faces will have appeared $c$ times, none more, and only a few significantly less.
Even mid-way through the pile, each face has appeared at most $c$ times,
which is evidence against $\Hiid$ for small $c$.
For instance, the chance of seeing no face twice when drawing 26 cards iid from 52 faces is less than 0.2\% 
(cf.\ the birthday paradox).
That is, latest half-way through a single deck, this fact is revealed.
If cards are drawn from 2 decks, 
more than 52 cards are needed to reveal that they are not iid
(Figure~\ref{fig:MEMPV2} bottom right).

\paragraph{Data duplication.}
In modern Machine Learning, esp.\ Deep Learning, data $x_{1:n}$ is abundant (large $n$) and 
observation spaces are huge (large $𝓧$). 
For instance, ImageNet consists of over 14 million images, 
usually resized or cropped to e.g.\ $224×224$ pixels of $256×256×256$ colors,
i.e.\ $𝓧={256}^{3×224×224}$. We can as well assume that $𝓧$ is infinite.
Assume $x_{1:n}$ contains no duplicate images and is well shuffled. 
As we will show later, no valid test can reject $\Hiid$ in this case.
But if $x_{1:n}$ contains duplicates one may be able to reject $\Hiid$,
similarly to the Black Jack example above, 
even without knowing anything about the observation space $𝓧$.
For instance, assume every observation is duplicated,
i.e.\ every $x$ that appears in $x_{1:n}$ appears exactly twice.
For uncountable $𝓧$, if $x$ is sampled from a probability \emph{density},
the probability of sampling the same $x$ twice is $0$.
So duplications can only happen if $P_\vt$ contains point masses, i.e.\ is not purely continuous,
i.e.\ $P_\vt(x)>0$ for some $x$. For finite or countable $𝓧$, this is necessarily true.
But if $P_\vt(x)>0$, then the frequency of seeing $x$ is binomially distributed.
While seeing some $x$ twice is plausible, 
seeing \emph{all} $x$ exactly twice is very unlikely, so we can reject $\Hiid$:
If data is iid and some items are duplicate,
we should also see triples and quadruples, etc.

\paragraph{Relevance for machine learning.}
The predominant training and evaluation protocol in Machine Learning in general and Deep Learning in particular is still to assume the data is iid, 
train on most of the data and evaluate on the rest.
Interestingly, this is true even for models dealing with definitely non-iid text.
For instance, for Transformers, text is crudely chopped into chunks of equal length and then shuffled.
The empirical test loss is an unbiased estimator of the true loss,
so is a proper way of comparing the performance of different models.
Data sizes in modern machine learning are huge,
so that even 10\% held-out data is so much that test noise is often of little concern.
That's at least the general story.

But in Deep Learning, data these days is often scraped from the whole internet,
and duplications abound.
For instance, assume the whole data set contains 3 copies of each data item.
If we randomly split off 10\% as the test set,
then the train set contains nearly all (99\%) of the test set items. 
With heldout-validation a pure memorizer without any generalization capacity will perform nearly 
perfectly on the test set \citep{Hutter:22preqmdl},
but will fail in practice on future data.
Indeed shuffling the data makes this problem the worst \citep{Sogaard:21,Gorman:19}. 
The problem is known as test set contamination, and well known. 

The standard solution is to decontaminate or clean the data,
e.g.\ removing duplicates, but this does not suffice.
One has to remove \emph{approximate} duplicates too.
But at what threshold for example should a document that cites a training-set document verbatim be removed from the test-set?
When are two images scraped from the internet rescaled or cropped or jpeg compressed versions of the same image, and even if so, should they be regarded duplicates?
While approaches exist that meliorate the problem, 
in theory this problem is ill-defined \cite[FAQ]{Hutter:06hprize},
and in practice a huge, actually AI-complete, problem \citep[App.C]{Brown:20}.

Test set evaluation is empirically sound for iid data,
therefore the failure of this paradigm must be attributed to the non-iid nature of the data.
The strength and weakness of the iid tests developed in this paper are that they are completely model- and data-agnostic. 
This makes them universally applicable and valid, but also very weak.
On the other hand, as discussed above, finding good/perfect model- and data-type-sensitive tests 
is itself a difficult/impossible research question beyond the scope of this article. 

\section{Problem Formalization and Preliminaries}\label{sec:Form}

We now introduce notation and concepts used throughout the paper:
general notation, the multinomial and Binomial distributions, 
first-order and second-order counts, 
exchangeable distributions, 
and statistical tests.
The reader familiar with these concepts could skim this section 
to just pick up the notational convention used in this article.

\paragraph{Notation.} 
We use calligraphic upper letters such as $𝓧$ for sets and $|𝓧|$ or $\#𝓧$ for the size of $𝓧$.
Probability spaces are denoted by $(𝓧^n,Σ,P)$ with $d=|𝓧|∈ℕ∪\{∞\}$.
$P_\vt$ denotes iid distributions.
$Q$ denotes exchangeable distributions.
Capital letters $X,N,M,T,E,O,D,C,U,V,Z,Y$ denote random variables,
and corresponding lower case letters samples corresponding to them.
We will use the shorthand $P(x):=P[X=x]$ and similarly for other random variables.
The variance of $Z$ is $\Var[Z]:=𝔼[Z^2]-𝔼[Z]^2$ and the covariance of $Y$ and $Z$ is $\Cov[Y,Z]:=𝔼[YZ]-𝔼[Y]𝔼[Z]$.
In addition to the classical $O()$ notation, 
$O_P(f(n))$ denotes all random functions $F(n)$ for which $∀δ>0~∃c>0~∀n:P[|F(n)|≥c⋅|f(n)|]≤δ$.
We use $f(n)\lesssim g(n)$ to denote $f(n)≤g(n)⋅[1±O_P(1/\sqrt{n})]$,
and similarly $\gtrsim$ and $\simequd$. 
We use $\lesssim$ even when stronger asymptotic or non-stochastic bounds would be possible,
since we mostly care about the leading-order approximation in $n$ and not the approximation error
as long as it tends to zero almost surely for $n→∞$.
See Appendix~\ref{sec:Notation} for more standard notation and beyond.

\paragraph{Multinomial distribution.} 
Let $X_{1:n}:=(X_1,...,X_n)$ be $𝓧$-valued random variables.
Let $x_{1:n}$ be sampled from some probability distribution $P$, 
where $x_t∈𝓧$ for $t∈\{1:n\}:=\{1,...,n\}$.
Though we are interested in general measurable (finite, infinite, uncountable) $𝓧$,
we will show in \Cref{sec:XtoN} that without loss of generality 
we can and hence will assume that $𝓧=\{1:d\}$ or $𝓧=ℕ$ ($d=∞$), and $Σ=2^𝓧$.
We will use the shorthand $P(x_{1:n}):=P[X_{1:n}=x_{1:n}]$ and $P(x_t)=P[X_t=x_t]$ 
and similarly for other random variables.
The null hypothesis $\Hiid=\{P_\vt\}$ is that $X_{1:n}$ are i.i.d.
In this case $P=P_\vt$ for some $\vt∈[0;1]^d$ with $∑_{x=1}^d θ_x=1$ and $P_\vt(x):=θ_x$.
Let $n_x:=\#\{t:x_t=x\}∈\{0:n\}$ be the number of times $x$ appears in $x_{1:n}$, called (first-order) \emph{item count},
and $N_x:=\#\{t:X_t=x\}$ be the corresponding random variables.
Then $P_\vt$ can be written as
\begin{align*}
  P_\vt(x_{1:n}) ~=~ P_\vt(x_1)⋅...⋅P_\vt(x_n) ~=~ θ_{x_1}⋅...⋅θ_{x_n} ~=~ ∏_{x=1}^d θ_x^{n_x}
\end{align*}
This expression is independent of the order of $x_{1:n}$, which leads to the multinomial distribution
\begin{align}\label{eq:multinom}
  P_\vt[N_{1:d}=n_{1:d}] ~=~ \Big({n\atop n_1,...,n_d}\Big)\smash{ ∏_{x=1}^d θ_x^{n_x}}
\end{align}
In particular the probability of event $N_x=k$ has a binomial distribution
(see \Cref{secm:Lemmas} for further details):
\begin{align}
  P_\vt[N_x=k] ~&=:~ P_{θ_x}(k) ~=~ f_k^n(θ_x) \nonumber \\
  f_k^n(θ) ~&:=~ P_θ(k) ~=~ ({\textstyle{n\atop k}})θ^k(1-θ)^{n-k}  \label{eq:binom}
\end{align}

\paragraph{Poisson distribution.}
\begin{lemma}[\bfm Poisson distribution]\label{lem:poisson}
The Poisson($λ$) distribution $P_λ(k):=λ^k e^{-λ}/k!=:g_k(λ)$ for $λ≥0$ and $k∈ℕ_0$ has the following properties:
$𝔼[k]=\Var[k]=λ$. For fixed $k$ it is unimodal in $λ$ with maximum at $λ^*=k$ and 
$\max_λ P_λ(k)=P_k(k)=(1-\dot{ε}_k)/\sqrt{2πk}$ and $0≤\dot{ε}_k→0$ \eqref{eq:stirling}.
\end{lemma}

With a slight overload in notation, 
let $P_\vl(\v n):=∏_{x=1}^d λ_x^{n_x}e^{-λ_x}/n_x!$ for $\v n∈ℕ_0^d$ be a product of $d$ independent 
but not identical Poissons, where $\v n≡n_{1:d}$. 
It is well-known that $P_\vt(n_{1:d})=P_\vl[N_{1:d}=n_{1:d}|N=n]$ for $\vl=n\vt$ and $N:=N_+=N_1+...+N_d$.
The mean and variance of Poisson($λ$) are both $λ$,
hence $𝔼_\vl[N]=∑_x 𝔼_\vl[N_x]=∑_x λ_x=n$ and 
similarly $\Var_\vl[N]=n$ (using independence).
This means that $N=n±O_P(\sqrt{n})$ is close to $n$ for large $n$,
and therefore under certain conditions,
$P_\vt[N_x=k]≈P_\vl[N_x=k]$ even without conditioning on $N=n$.
For instance, $𝔼_\vt[N_x]=nθ_x=λ_x=𝔼_\vl[N_x]$ and 
$\Var_\vt[N_x]=nθ_x(1-θ_x)≤nθ_x=λ_x=\Var_\vl[N_x]$.
Unfortunately, for the events $R$ we care about,
it is extremely cumbersome to quantify the relation $P_\vt(R)≈P_\vl(R)$.
We will do so in \Cref{secm:Lemmas}.
In the main \Cref{sec:iidtests} we adopt a simpler approach:
Noting that $N$ is itself Poisson($n$) distributed,
\begin{align}\label{eq:cn}
  P_\vt(\v n) ~&=~ P_\vl(\v n)/P_\vl[N=n] ~=~ c_n⋅P_\vl(\v n) \\ \nonumber
  P_\vl[N=n] ~&=~ P_n(n) ~=~ n^n e^{-n}/n! ~=:~ 1/c_n ~\simequd~ 1/\sqrt{2πn} 
\end{align}
Hence, for any event $R$, we can upper $P_\vt[R]≤c_n⋅P_\vl[R]$.
For our tests, $P[R]$ is typically exponentially small in $n$,
so the blow-up by $c_n=O(\sqrt{n})$ is insignificant in theory:
Increasing sample size $n$ to $n+O(\log n)\simequd n$ cancels $c_n$.
We therefore can and will treat the item counts $N_1,...,N_d$ 
as independent Poisson distributed $N_x\sim P_{λ_x}$, 
i.e.\ $\v N≡N_{1:d}\sim P_\vl$,
which greatly facilitates the developments of our tests.
In \Cref{secm:iidtests,secm:Lemmas} we show that $c_n$ can directly be replaced by $1$ in many cases of practical interest.

\paragraph{Second-order count multiplicity.} 
Let $m_k:=\#\{x:n_x=k\}$ be the number of $x$ that appear $k$ times in $x_{1:n}$,
called (second-order) \emph{count multiplicities}, and $M_k:=\#\{x:N_x=k\}$ be the corresponding random variables.
Note that $M_k=0$ for $k>N$ but $M_k=0$ also for many $k≤N$
due to $∑_{k=0}^∞ k⋅M_k=N$ and $∑_{k=0}^∞ M_k=d$.
Let $M_+:=∑_{k=1}^∞ M_k=\#\{x:N_x>0\}=\#\{X_1,...,X_n\}$ be the number of different $X_t$, not counting multiplicities.
We are mostly interested in $d=∞$, in which case $M_0=∞$ is not a useful statistic.
We therefore exclude $M_0$ in $\v M:=M_{1:n}$.

\paragraph{Exchangeable distributions.}
Non-iid distributions will be denoted by $Q$.
A distribution $Q$ is exchangeable if it is invariant under permutations of $x_{1:n}$,
i.e.\ $Q(x_{1:n})=Q(x_{π(1:n)})$, where $π∈S_n$ is any permutation of $1:n$.
As in the iid-case, $Q$ only depends on the counts $\v n$.
Let $𝓠$ be the class of all exchangeable distributions $Q$.

For instance, for $d=2$, Laplace's rule $Q(x_{t+1}|x_{1:t})=(\#\{τ≤t:x_τ=x_{t+1}\}+1)/(t+2)$
has exchangeable but non-iid distribution $Q(x_{1:n})=n_1!n_2!/(n+1)!$.
Similarly for the Good-Turing and Ristad distributions \citep{Hutter:18off2onx}.

Exchangeable distributions occur naturally as follows:
Assume $X'_{1:n}$ are drawn from an arbitrary distribution $Q'$,
and then perfectly shuffled such as to destroy any order information.
Formally, $X_{1:n}=X'_{Π(1:n)}$, where $Π$ is drawn uniformly from all permutations $S_n$.
It is easy to see that $X_{1:n}$ are exchangeable random variables.
In particular exchangeable $X_1,...,X_n$ are identically distributed, i.e.\ $Q[X_t=x]=Q[X_{t'}=x]$.

\paragraph{Invariant statistical tests.} 
A (valid) statistical test of significance $0<α<1$ is a reject region $R⊂𝓧^n$ such that $P_\vt[R]≤α~∀\vt$.
We can reject the hypothesis $\Hiid$ that $\v x≡x_{1:n}$ is iid with confidence $1-α$ \emph{iff} $\v x∈R$,
that is, $\Hiid$ is falsely rejected (Type~I error) with probability at most $α$.
Reject regions are most often defined via a test statistic $T:𝓧^n→ℝ$ 
and $R=\{\v x: T(\v x)>c\}$ for some critical value $c∈ℝ$.
$T$ at critical level $c_α:=\inf\{c:\sup_\vt P_\vt[T(\v X)>c]≤α\}$ has significance $α$.
The $p$-value of a test $T$ for data $\v x$ is $p:=\sup_\vt P_\vt[T(\v X)>T(\v x)]$
is the smallest level $α$ at which we can reject $\Hiid$:
$T$ can reject $\Hiid$ with confidence $1-p$.

Since we assume $X_{1:n}$ are exchangeable (shuffled), 
it is natural to ask for a test to reject $\Hiid$ independently of the order in which $X_1,...,X_n$ are presented.
That is, $T$ should be a function of the item counts $N_{1:d}$ only.

Furthermore, we do not want to make any structural assumptions on $𝓧$.
While each $Q∈𝓠$ is \emph{not} necessarily invariant under permutations of elements of $𝓧$, 
the class $𝓠$ itself is. Since we want to test against all $𝓠$,
it is natural to consider tests that are not affected by permuting $𝓧$,
that is, $T(X_{1:n})=T(π(X_1),...,π(X_n))$, where this $π$ is any permutation of elements in $𝓧$.
Combining both invariances, we must have $T(X_{1:n})=T(N_{1:d})=T(N_{π(1:d)})$.
$T$ is invariant under reordering of $N_{1:d}$ \emph{iff} it only depends on $M_0,...,M_n$.
It may be possible to make an argument for order-independent tests
that among the most powerful tests w.r.t.\ to some invariant sub-class of $𝓠$ 
there is always an invariant test, i.e.\ they include all minimax optimal tests.

\begin{definition}[\bfm Invariant tests $T$]\label{prop:invariant}
We call tests $T:𝓧^n→ℝ$ that are invariant under permutations of the argument $x_1,...,x_n$ 
as well as invariant under permutations of the elements in $𝓧$,
invariant tests. Invariant tests are functions of $M_0,...,M_n$ only.
\end{definition}

\paragraph{The power of tests.}
Neyman-Pearson use alternative hypotheses to determine the 
power $β=Q[T>c]$ of a test $T$ for $Q$ ($1-β=$Type~II error).
In our case, the alternative hypothesis is the set of exchangeable distributions without the iid distributions $\Halt:=𝓠\setminus\Hiid$.
There are no uniformly most powerful (UMP) tests for $\Halt$, not even close; $\Halt$ is too broad.
Each test will have high power for some subset of $𝓠$ and low power for other $Q∈𝓠$.
We do not formally define ``interesting'' subsets of $𝓠$ 
and derive the power of tests for them or find UMPs for these subsets. 
We focus on developing tests which have known small (upper bound on the) Type~I error $α$
= small size $α$ = significance level $α$ = probability of falsely rejecting $\Hiid$ when it is actually true.
We therefore rarely mention $𝓠$, so unless explicitly mentioned to the contrary, 
distributions and sampling refers to iid or multinomial or Binomial.
Our work is closer in spirit to Fisher hypothesis testing without alternative hypothesis,
but we do demonstrate the power of the tests empirically in \Cref{sec:exp} on some hand-selected $Q$.

\section{Reducing general $𝓧$ to $ℕ$}\label{sec:XtoN}

In this section we discuss general probability spaces $(𝓧^n,Σ,P)$, 
only to discover that we can without loss of generality restrict our analysis 
to $𝓧=ℕ$ and $𝓧=\{1:d\}$ with discrete $σ$-algebra $Σ=2^𝓧$.
The only assumption we have to make on $Σ$ is that it contains all singletons,
$\{\v x\}∈Σ~∀\v x∈𝓧^n$, in order for the events $\v X=\v x$ to be measurable.
For example, every T1 or Hausdorff space provided with the Borel sets satisfies this,
in particular $ℝ^n$.

\paragraph{Infinite $𝓧$.}
So far we have considered finite and countable $𝓧$.
Consider now $𝓧^n=ℝ^n$ with joint Gaussian density $ρ=\text{Gauss}(\v 0,\v{Ξ})$.
The probability that $x_{1:n}$ contains repetitions is zero,
hence $x_1,...,x_n$ are all different, so $M_0=∞$, $M_1=n$, $M_k=0~∀k≥2$.
Since an invariant test $T$ is a function of $M_{0:n}$ only,
there is a constant $c'$ such that almost surely $T(M_{0:n})=c'$,
hence $Q_ρ[T>c]$ is identically 0 or identically 1, i.e.\ the same for all $\v{Ξ}$. 
It cannot be 1, since $T$ must satisfy $Q_ρ[T>c]≤α<1$ for iid $Q_ρ$ ($Ξ_{tt'}\propto⟦t=t'⟧$),
but then $T$ never rejects $\Hiid$, even if $ρ$ is non-iid and maximally correlated ($Ξ_{tt'}=1~∀tt'$).
In general, any uncountable $𝓧$ can be equipped with a $σ$-algebra and non-atomic measure,
leading to the same conclusion.
Now consider $𝓧=εℤ$ and discretize $ρ$. 
For $ε→0$ the conclusion still holds.
Formally, for every $δ>0$ there exists an $ε>0$,
such that all $x_1,...,x_n$ are different with probability at least $1-δ$.
Hence the conclusion also holds for countably infinite $𝓧$.

\begin{proposition}[\bfm All tests are powerless against densities]\label{prop:dense}
If $𝓧$ is infinite and all $x_1,...,x_n$ are different,
no valid invariant test can reject $\Hiid$. In particular,
$X_{1:n}$ are almost surely all different if sampled from a non-atomic measure,
e.g.\ if the measure has a density w.r.t.\ to the Lebesgue measure on $𝓧=ℝ^d$.
\end{proposition}

\paragraph{\boldmath Reduction of $𝓧$ to $ℝ$.} 
Let $𝓧_{pp}:=\{x∈𝓧: P[\{x\}]>0\}$, which is countable, $β:=P[𝓧_{pp}]$,
then $P_{pp}[A]:=P[A|𝓧_{pp}]$ is a pure point measure and 
$P_{\overline{pp}}[A]:=P[A|𝓧\setminus𝓧_{pp}]$ is non-atomic:
\begin{align*}
  P[A] ~=~ β⋅P_{pp}[A] ~+~ (1-β)⋅P_{\overline{pp}}[A]
\end{align*}
i.e.\ every measure $P$ can be decomposed into a pure point measure 
and a non-atomic rest.

Consider now a point measure $\tilde P_{pp}$ on $ℝ$ 
with $\tilde P_{pp}[\{2i\}]:=P_{pp}[\{y_i\}]$ and zero elsewhere, 
where $𝓧_{pp}=\{y_1,y_2,...\}$ is some enumeration of elements in $𝓧_{pp}$. Now define
\begin{align}\label{eq:tildeP}
  \tilde P[A] ~=~ β⋅\tilde P_{pp}[A] ~+~ (1-β)⋅\text{Gauss}_{0,1}[A]
\end{align}
As far as the second-order counts $\v M:=(M_1,...,M_n)$ and $M_0=∞$ are concerned,
$\tilde P[\v M=\v m]=P[\v M=\v m]$, since for the discrete part we bijected $𝓧_{pp}$ to (a subset of) $2ℕ$
with same probability mass, and $\v M$ is invariant under such bijection.
As for the non-atomic part, in both cases, we almost surely each time sample a novel $x$ 
not seen before, i.e.\ only $M_1$ is affected and increases by 1 with probability $1-β$.
That is, we can restrict ourselves to measures on $ℝ$ of the form \eqref{eq:tildeP}:

\begin{proposition}[\bfm $𝓧=ℝ$ suffices]\label{prop:xrsuff}
For every invariant test $T$, \\
$P[T>c]≤α$ for iid $P$ on $𝓧$ $\Longleftrightarrow$
$\tilde P[T>c]≤α$ for iid $\tilde P$ on $ℝ$ of the form \eqref{eq:tildeP}.
\end{proposition}
Note that $T$ only depends on the counts $\v M$ and $M_0$, 
so the same $T$ is defined across every $𝓧$,
and gives the same result independent from which infinite space $\{x_1,...,x_n\}⊆𝓧$ came from.

\paragraph{Reduction of $ℝ$ to $ℕ$.} 
Consider $\tilde P_\ell[\{2i-1\}]:=1/\ell$ for $i=1,...,\ell$ and 0 on $ℝ$ else.
Assume we draw $n$ iid samples from $\tilde P_\ell$. 
The probability of sampling some $x$ twice is 
\begin{align}\label{eq:once}
  \tilde P_\ell[∃t≠t':X_t=X_{t'}] ~≤~ ∑_{t≠t'}\tilde P_\ell[X_t=X_{t'}] ~=~ ∑_{t≠t'}\frac1\ell ~≤~ \frac{n²}\ell
\end{align}
That is, for fixed $n$, the probability for all $x$ being unique ($P_\ell[\v M=(n,0,0,...)]$) tends to 1 for $\ell→∞$.
Combining this with the point measure above implies for all $\v m:=m_{1:n}$,
\begin{align}\label{eq:ell}
  β⋅\tilde P_{pp}(\v m) + (1-β)⋅\tilde P_\ell(\v m) ~~~\longrightarrow~~~ \tilde P(\v m) ~~~\text{for}~~~ \ell→∞
\end{align}

\begin{proposition}[\bfm $𝓧=ℕ$ suffices]\label{prop:xnsuff}
For every invariant test $T$ and infinite $𝓧$, \\ 
$P[T>c]≤α$ for all iid $P$ on $𝓧$ ~~~$\Longleftrightarrow$~~~
$\tilde P[T>c]≤α$ for all iid $\tilde P$ on $ℕ$.
\end{proposition}
This justifies our restriction to finite $𝓧=\{1:d\}$ and countable $𝓧=ℕ$ ($d=∞$).

\paragraph{Finite $𝓧$.}
\emph{Embedding finite $𝓧$ into infinite $𝓧$:}
Observing (only) $x_{1:n}$, we we do not want to make any assumption from which space $𝓧$
they have been sampled from. Obviously $|𝓧|≥M_+$ is needed.
But any $P$ on a finite domain, say $\{1:d'\}$, can be extended to infinite $𝓧$
by setting $P[𝓧\setminus \{1:d'\}]=0$ without affecting $\v M$, 
i.e.\ infinite $𝓧$ also contain all finitely supported $P$.
So if a test has confidence $1-α$ for $|𝓧|=∞$, then it also has confidence at least $1-α$ for $|𝓧|<∞$.
The converse however is not true:

\emph{Knowing $𝓧$ is finite and its size $d$:}
While $d=∞$ includes all measures that have finite support, i.e.\ $θ_x=0$ for all $x>d'$,
\emph{knowing} that $d$ is finite provides extra information,
so the $d=∞$ analysis does \emph{not} automatically include the case where $d$ is known and finite.
While the $d=∞$ tests remain valid for $d<∞$, 
stronger tests are possible for $d<∞$.
The reason is that $M_0$ depends on $|𝓧|$.

Example: Recall that no invariant test can reject $\Hiid$ if all $x_t$ are different ($M_1=n$),
but this relied on $𝓧$ being infinite.
On the other hand, if we know/assume $|𝓧|=n$, then the probability of seeing every $x∈𝓧$ exactly once is
\begin{align*}
  P_\vt[∀t≠t':X_t≠X_{t'}] ~=~ n!⋅∏_{x∈𝓧} θ_x ~≤~ n!⋅\max_\vt ∏_{x∈𝓧} θ_x 
  ~=~ n!∏_{x∈𝓧} \frac1n ~=~ \frac{n!}{n^n} ~≤~ \sqrt{n}⋅e^{1-n}
\end{align*}
This is extremely small for large $n$, so $\Hiid$ can be rejected with very high confidence.
In particular having observed $x_{1:n}$ and knowing nothing about $𝓧$,
we cannot choose $|𝓧|=d=\#\{x:N_x≥1\}≡M_+=∑_{k=1}^n M_k$, 
even if we were somehow able to deal with the fact that such $d$ would itself be random.

\emph{Approximating infinite $𝓧$ by finite $𝓧$:}
While we made the case for countably infinite $𝓧$, 
for fixed $n$ and the approximate results we aim at, we actually do not need to consider $d=∞$,
but $\ell=d=n^3$ suffices: 
We sort $x$ in decreasing order of $P[\{x\}]$, truncate $𝓧_{pp}$ to $\fr12 n^3$ elements
and choose $\ell=\fr12 n^3$ (see above). 
So any potential complications from $d=∞$ can easily be avoided,
but it turns out that mostly $d=∞$ is more convenient.

\section{I.I.D.\ Tests}\label{sec:iidtests}

We are finally in a position to develop some tests.
We first outline the common idea behind all tests developed in this section.
In \Cref{sec:moretests} we discuss alternative approaches.
We then derive a couple of tests that feel natural.
Although they follow a common theme, they are quite diverse in the sense that 
every test highlights a new or different feature or power or technical difficulty.
The most basic test uses a single $M_k$, 
all others are linear combinations thereof, 
except the last one, which is a logarithmic combination.
The even and odd tests $E$ and $O$ are global sums of $M_k$ over all even/odd $k$.
The slope test $D_k=M_k-M_{k-1}$ is a bit more difficult to derive but also allows for a lower bound test.
The curvature test $C_k=2M_k-M_{k-1}-M_{k+1}$ and its logarithmic version $\bar U_k=\ln(M_k^2/M_{k-1}M_{k+1})$ can be very strong. 
Also, while some tests require to use the empirical variance ($E,O,\bar U_k$),
for others a theoretical upper bound is possible ($D_k,C_k,M_k$) and better ($M_k$).
$P$, $𝔼$, $\Var$, $\W$ are w.r.t.\ the mixture of Poisson distributions $P_\vl$ \eqref{eq:cn},
which approximates $P_\vt$.

\paragraph{The general idea behind the tests.}
First note that the Poisson distribution $P_λ(k)=λ^k e^{-λ}/Γ(k+1)$ is smooth 
if we take the liberty of plugging in $k∈ℝ$ 
(see e.g.\ $𝔼[M_k]$ curve in Figures~\ref{fig:MEMPV1} left for ``{\sf uniform}'').
It has a unique maximum at $k=λ$ and is log-concave, so a rather benign function.
For large $λ$, $P_λ(k)$ is also ``smooth'' in $k∈ℕ_0$ in the sense that its finite-difference approximations of slope and curvature (and higher) are small.
It still has a unique maximum at $k≈λ$ and is log-concave, and indeed approximately Gaussian with mean and variance $λ$.
$P_λ(k)$ is also differentiable in $λ$, which we will also exploit.
Now consider 
\begin{align}\label{eq:EMk}
  𝔼[M_k] ~&=~ \textstyle 𝔼[\#\{x:N_x=k\}] ~=~ 𝔼∑_x⟦N_x=k⟧ \\
  ~&=~ \textstyle ∑_x P_{\vl}[N_x=k] ~=~ ∑_x P_{λ_x}(k) ~=~ ∑_x g_k(λ_x) \nonumber
\end{align}
That is, $𝔼[M_k]$ is a sum of Poisson($λ_x$) distributions. 
Depending on the distribution of $λ_x$, $𝔼[M_k]$ as a function of $k$ may have multiple extrema,
but as a mixture of Poissons it cannot be less smooth and typically is even more smooth 
(see e.g.\ $𝔼[M_k]$ curve in Figures~\ref{fig:MEMPV1} left for ``{\sf linear}'' mixture).
Since $\bar M_k→𝔼[\bar M_k]$ for $n→∞$, $M_k$ as a function of $k$ will inherit any (lack of) structure in $𝔼[M_k]$, just with noise added.
Since invariant tests can only depend on $\v M$, they must test for some such property.
For instance, no Poisson and hence no mixture of Poissons can have $𝔼[M_k]=0$ for all odd $k$ 
(see e.g.\ Figures~\ref{fig:MEMPV1} left for ``{\sf even-n}''),
so $M_k=0$ for all odd $k$ is strong evidence against $\v X$ being iid.

\paragraph{Linear tests.} 
Most of our tests are (signed) linear combinations of a subset of the $M_k$. 
The general template for upper bounds on the mean and variance is as follows:

\begin{proposition}[\bfm Poisson upper bounds for linear tests]\label{prop:ublt} 
  Let $T=∑_k α_k M_k$ for $α_k∈ℝ$.
  Provided all involved sums and integrals are absolutely convergent, we have \ntc{\\}
  $τ:=𝔼[T]≤n⋅\sup_{λ>0} g(λ)/λ =:τ^{ub}$,
  where $g(λ):=∑_k α_k P_λ(k)=∑_k α_k λ^k e^{-λ}/k!$, and \ntc{\\}
  $\Var[T]≤∑_k α_k^2 𝔼[M_k]≤V^{ub}$, 
  where $V^{ub}:=∑_k α_k^2 μ_k^{ub}$ 
  with $μ_k^{ub}≥𝔼[M_k]$ upper bounding the expectations of $M_k$.
\end{proposition}
Instead of upper bound $V^{ub}$ as defined above, 
for some of our tests we use random $V^{ub}=∑_k α_k^2 M_k$, 
which is an upper bound in expectation justified by \Cref{lem:bubtest}.

\begin{proof}
Using \eqref{eq:EMk} and \Cref{lem:ubfe} we can upper bound 
\begin{align*}
  \textstyle 𝔼[T] ~=~ ∑_k α_k ∑_x P_{λ_x}(k) ~=~ ∑_x g(λ_x) ~≤~ n⋅\sup_{λ>0}g(λ)/λ
\end{align*}
For the variance, for $Z_k^x:=⟦N_x=k⟧$, we have $Z_k^+=M_k$ and $Z_k^x Z_{k'}^x=0$.
Furthermore, $\Cov[Z_k^x,Z_{k'}^{x'}]=0$ for $x≠x'$ since $N_x$ and $N_{x'}$, hence $Z_k^x$ and $Z_{k'}^{x'}$ are independent.
Hence \Cref{lem:Zkxcor} implies
\begin{align*}
  \textstyle \Var[T] ~&=~ \Var[∑_k α_k Z_k^+] ~≤~ ∑_k α_k^2 𝔼[Z_k^+] \ntc{~=~ ∑_k α_k^2 𝔼[M_k]} ~≤~ ∑_k α_k^2 μ_k^{ub} ~~~~~\qedeq
\end{align*}
\end{proof}\vspace{-2ex}

\begin{figure*}[tbh!] 
\begin{center}
\includegraphics[width=0.51\textwidth]{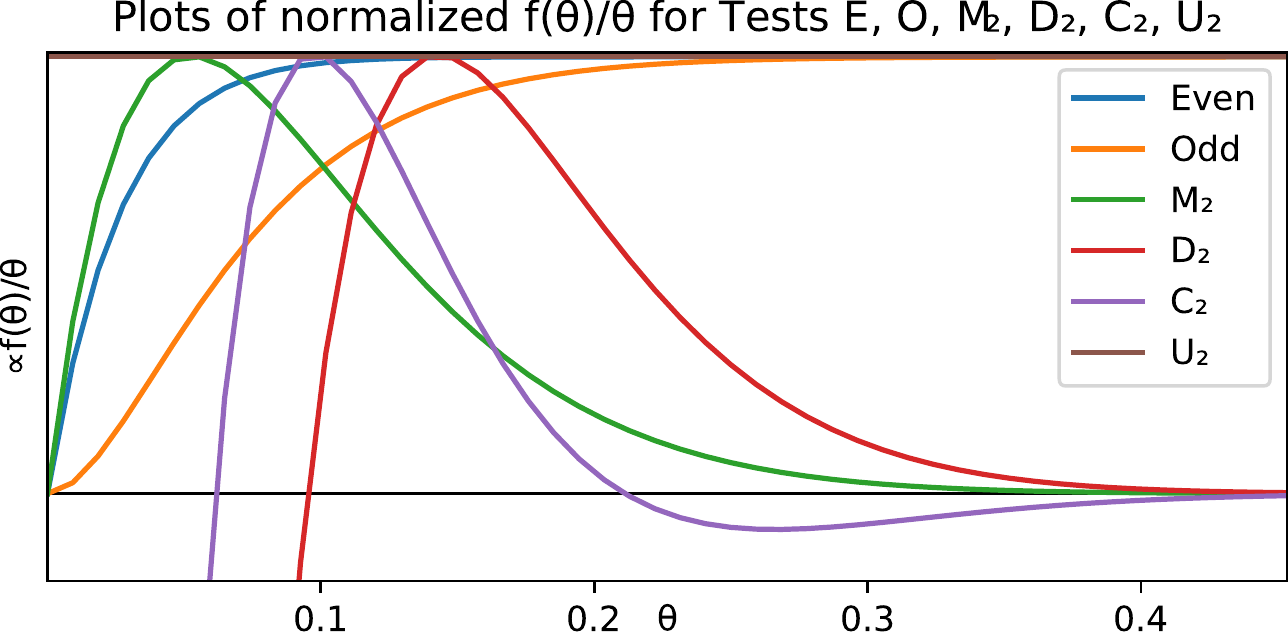}\vspace*{-2ex}
\caption{Normalized $g(λ)/λ≈f(θ)/θ$ for Tests $E,O,M_k,D_k,C_k,\bar U_k$ for $k=2$
}\label{fig:fT}\vspace*{-2ex}
\end{center}
\end{figure*}

\paragraph{\boldmath Second-order count tests $M_k$.} 
We first determine upper bounds for the second-order counts $M_k$ for each $k$ separately.
We can apply \Cref{prop:ublt} with $α_{k'}=⟦k'=k⟧$ or just directly apply \Cref{lem:ubfe}:
\begin{align}
  μ_k ~&:=~ 𝔼[M_k] ~=~ ∑_x g_k(λ_x) ~≤~ n⋅\sup_{λ>0}\frac{λ^{k-1} e^{-λ}}{k!} \nonumber\\
  ~&=~ n\frac{(k\!-\!1)^{k-1}e^{-(k-1)}}{k!} ~=~ \frac{n}{k}\frac{1-\dot{ε}_{k-1}}{\sqrt{2π(k\!-\!1)}} ~=:~ μ_k^{ub} \label{eq:Mkub}
\end{align}
The bound follows from $g_k(λ)/λ=λ^{k-1}e^{-λ}/k!$ being maximal for $λ^*=k-1$ (cf.\ Figure~\ref{fig:fT}).
The approximate expression follows from \Cref{lem:stirling} with $1/(12k+1)≤\dot{ε}_k≤1/12k$.
Note that for $k=1$ the bound is valid but vacuous.
We see that the relative frequency of $k$-multiplicities $\bar{μ}_k=μ_k/n$ is upper bounded by $O(k^{-3/2})$,
i.e.\ the expected \emph{number} of such items is $k⋅𝔼[M_k]=O(1/\sqrt{k})$.
From \Cref{lem:poisson} we know that for a single Poisson $\max_λ P_λ(k)≈1/\sqrt{2πk}$,
hence a mixture of Poisson($λ$)'s cannot be larger, which is consistent with the result above.

By \Cref{prop:ublt} 
we (also) have $\Var[M_k]≤𝔼[M_k]≤μ_k^{ub}$, 
so by \Cref{lem:bubtest} with $Z_x=⟦N_x=k⟧∈[0;1]$, the $p$-value for rejecting $\Hiid$ is
\begin{align*}
  p ~&\lesssim~ \textstyle Φ_n((μ_k^{ub}-M_k)/\sqrt{μ_k^{ub}}) \tc{\\}
  ~\tc{&}≤~ \exp(-\fr12 n(\bar M_k-\bar{μ}_k^{ub})^2/\bar{μ}_k^{ub}+O(1)) ~=~ e^{-O(n/k^{3/2})}
\end{align*}
We need to use $Φ_n(y):=c_n⋅Φ(y)$ (see \eqref{eq:cn} and \Cref{def:Gauss}) rather than $Φ$,
since in reality $M_k\sim P_\vt$ while $𝔼$ and $\Var$ were w.r.t.\ $P_\vl$,
so $p=P_\vt[T(\v M)>c]≤c_n⋅P_\vl[T(\v M)>c]\simequd c_n⋅Φ(\cdots)=:Φ_n(\cdots)$ for all of our tests $T$.
The $Φ_n$-bound only holds if $M_k>μ_k^{ub}$ and the exponential bound if furthermore $n$ is sufficiently large.
We could also have chosen $V^{ub}=M_k$, a random upper bound on $\Var[M_k]$,
but if $M_k<μ_k^{ub}$, then $p≥\fr12$, so the test has no power,
and if $M_k>μ_k^{ub}$, then using $μ_k^{ub}$ leads to a stronger test than using $M_k$.

\begin{example}
Our running example for all tests will be a data set 
where each data item is duplicated and appears exactly twice.
In this case, $M_2=n/2$ and all other $M_k=0$.
For $k=2$ we have $\bar{μ}_2^{ub}=1/2e\dot=0.184$ and $p\lesssim\exp(-\fr12 n(\fr12-\fr1{2e})^2/\frac1{2e})\dot=e^{-0.271n}$.
i.e.\ $\Hiid$ can be extremely confidently rejected for moderately large $n$.
For $k≠2$, the tests have no power ($\bar M_k=0<\bar{μ}_k^{ub}$).
\eoe\end{example}

\paragraph{\boldmath Even and odd tests $E$ and $O$.} 
The above example suggests non-trivial upper bounds on the even and odd second-order counts $M_k$,
and a test based on that, but we have to be a bit careful.
For $d=∞$, there are $d-n=∞$ many unobserved $x∈𝓧$, hence $M_0=∞$.
Similarly, for $λ_x=n/d$ and $d→∞$, every $x$ is observed exactly once, hence $M_1=n$ and all other $M_k=0$,
again not leading to a useful test. The general solution is to exclude $M_0$ and $M_1$.
First, for $α_k^\text{even}:=k⋅⟦k≠0~\text{even}⟧$ and $α_k^\text{odd}:=k⋅⟦k≠1~\text{odd}⟧$ we define
(see also Figure~\ref{fig:fT})
\begin{align*}
  g_\text{even}(λ) ~&:=~ ∑_k α_k^\text{even} P_λ(k) ~=~ ∑_{k≠0~\text{even}} k⋅P_λ(k) \tc{\\}
  ~\tc{&}=~ ∑_{k≠0~\text{even}}{λ^k e^{-λ}\over(k-1)!} ~=~ \fr{λ}2[1-e^{-2λ}] \\
  g_\text{odd}(λ) ~&=~ ∑_k α_k^\text{odd} P_λ(k) ~=~ ∑_{k≠1~\text{odd}} k⋅P_λ(k) \tc{\\}
  ~\tc{&}=~ ∑_{k≠1~\text{odd}}{λ^k e^{-λ}\over(k-1)!} ~=~ \fr{λ}2[1-e^{-λ}]^2
\end{align*}
The last expressions follow from pulling out a $λ e^{-λ}$ from the sum 
and recognizing the Taylor series expansion of $\sinh(λ)$ and $\cosh(λ)$.
The even and odd test statistics
\begin{align*}
  E ~&:=~ ∑_k α_k^\text{even} M_k ~=~ ∑_{k≠0~\text{even}} k⋅M_k \ntc{~~~\text{and}~~~}\tc{\\}
  O ~\tc{&}:=~ ∑_k α_k^\text{odd} M_k ~=~ ∑_{k≠1~\text{odd}} k⋅M_k. 
\end{align*}
Using \eqref{eq:EMk}, the expectation can be upper bounded by \Cref{prop:ublt} as 
\begin{align*}
  𝔼[E] ~&=~ ∑_{\nq k≠0~\text{even}} k ∑_x P_{λ_x}(k) ~=~ ∑_x g_\text{even}(λ_x) \tc{\\}
       ~\tc{&}≤~ n⋅\sup_{λ>0}\fr12[1-e^{-2λ}] ~=~ \frac{n}2 ~=:~ ε^{ub} ~=~ \bar{ε}^{ub}n \\ 
  𝔼[O] ~&=~ ∑_{\nq k≠0~\text{odd}} k ∑_x P_{λ_x}(k) ~=~ ∑_x g_\text{odd}(λ_x) \tc{\\}
       ~\tc{&}≤~ n⋅\sup_{λ>0}\fr12[1-e^{-λ}]^2 ~=~ \frac{n}2~=:~ ο^{ub} ~=~ \bar{ο}^{ub}n 
\end{align*}
That is, excluding singletons, we should expect at most half of the data items to appear evenly often, 
and at most half oddly often.
The even/odd upper bounds are ``attained'' for $λ^*→∞$.
Using \Cref{prop:ublt} again, we can 
also upper bound the variances of $E$ and $O$:
\begin{align*}
  \Var[E] ~&≤~ \ntc{∑_k (α_k^\text{even})^2~𝔼[M_k] ~=~} ∑_{\nq k≠0~\text{even}\nq} k^2 𝔼[M_k]
  ~~~\text{and}~~~ \Var[O] ~≤ ∑_{\nq k≠1~\text{odd}\nq} k^2 𝔼[M_k]
\end{align*}
Unfortunately no meaningful finite $\vl$-independent upper bounds on them are possible.
Trying to use the same method as for upper bounding $𝔼[E]$ leads again to $λ^*→∞$,
but this time the expression diverges. 
Note that for fixed $\vl$, the variance is finite; 
it just does not have a (non-vacuous) uniform upper bound.
So in this case we have to resort to using the empirical upper bound 
$V_\text{even}^{ub}:=∑_{k≠0~\text{even}}k^2 M_k$ for the variance,
and similarly for $V_\text{odd}^{ub}:=∑_{k≠1~\text{odd}}k^2 M_k$.
By \Cref{lem:bubtest} and \Cref{lem:ubmom}, 
the $p$-values for rejecting $\Hiid$ are
\begin{align*}
  p ~&\lesssim~ \textstyle Φ_n((ε^{ub}-E)/\sqrt{V_\text{even}^{ub}}) ~≤~ \exp(-\fr12 n(\bar E-\fr12)^2/\sqrt{\bar V_\text{even}^{ub}}+O(1)) \\
  p ~&\lesssim~ \textstyle Φ_n((ο^{ub}-O)/\sqrt{V_\text{odd}^{ub}}) ~≤~ \exp(-\fr12 n(\bar O-\fr12)^2/\sqrt{\bar V_\text{odd}^{ub}}+O(1))
\end{align*}
As before, the $Φ_n$-bound only holds if $E>ε^{ub}$ and $O>ο^{ub}$;
the exponential bound only holds if $\bar E>\fr12$ and $\bar O>\fr12$ and sufficiently large $n$.

\begin{example}
In our running example in which each data item is doubled ($M_2=n/2$),
we have $E=n\bar E=n$ and $V_\text{even}^{ub}=2n$, hence the even test has $p$-value $p\lesssim e^{-n/8\sqrt{2}}$,
i.e.\ $\Hiid$ can be extremely confidently rejected for moderately large $n$.
The odd test has no power ($O=0$). If we triple each item, then $O=n/3$ and $E=0$
and $V_{odd}^{ub}=3n$, and the odd test has $p\lesssim e^{-n/8\sqrt{3}}$.
\eoe\end{example}

\paragraph{\boldmath Slope tests $D_k:=M_k-M_{k-1}$.} 
As mentioned at the beginning of the section, $𝔼[M_k]$ as a function of $k$ is ``smooth'',
so it is natural to test for a small difference=slope $D_k:=M_k-M_{k-1}$.
Let $D_k^x=⟦N_x=k⟧-⟦N_x=k-1⟧$, hence $D_k=∑_x D_k^x$.
Then similar to before 
\begin{align}\nonumber
  δ_k ~:=~ 𝔼[D_k] ~&=~ ∑_x P_{λ_x}[N_x=k]-P_{λ_x}[N_x=k-1] \tc{\\}
  ~\tc{&}=~ ∑_x g_δ(λ_x) ~≤~ n⋅\sup_{λ>0}\frac{g_δ(λ)}{λ} =: n\bar{δ}_k^{ub} \\ \label{eq:gdelta}
  \ntc{\text{where}~~~} & g_δ(λ):=\frac{λ^k e^{-λ}}{k!}\left[1-\frac{k}{λ}\right]
\end{align}
The maximum of $P_λ(k)$ is at $λ=k$ but the bracket $[1-\frac{k}{λ}]$ kills this maximum,
moving it to $≈k+\sqrt{k}$ (Figure~\ref{fig:fT}). 
Since $g_δ(λ)≤0$ for $λ≤k$, we can assume $λ>k$, hence 
\begin{align*}
  \frac{d}{dλ}\ln\frac{g_δ(λ)}{λ} ~&=~ \frac{d}{dλ}[(k-1)\ln λ-λ-\ln k! +\ln(1-k/λ)] \tc{\\}
  ~\tc{&}=~ \frac{k-1}{λ}-1+\frac{k/λ^2}{1-k/λ} ~\stackrel!=~ 0
\end{align*}
Multiplying with $λ^2(1-k/λ)$ leads to a quadratic equation in $λ$ 
which has two solutions $λ_±^*=k-\fr12±\sqrt{k+\fr14}$, only $λ_+^*>k$ is valid, and is indeed the global maximum.
Note that $g_δ(λ)=0$ for $k≥2$ but not for $k=1$, so \Cref{lem:ubfe} and hence the following bound only applies for $k≥2$.
The reason is that $D_1$ involves $M_0$, but $M_0=∞$ for infinite $𝓧$, hence $D_1=-∞$ and as a test is vacuous.
A tedious calculation shows that 
\begin{align}\label{eq:Dkub}
  \bar{δ}_k^{ub} ~=~ \frac{g_δ(λ_+^*)}{λ_+^*} ~=~ \frac{1-\ddot{ε}_k}{k^2\sqrt{2πe}}, ~~~\text{where}~~~ \ddot{ε}_k=O(1/\sqrt{k})
\end{align}
That is, the slope of (a mixture of) Poissons is upper bounded by $\bar{δ}_k^{ub}=O(1/k^2)$.
This is smaller than $μ_k^{ub}$ by a factor of $1/\sqrt{ke}$, 
so can lead to a stronger test than test $M_k$, 
provided that indeed $M_k$ deviates from $M_{k-1}$ sufficiently.

By \Cref{prop:ublt}, the variance of $D_k$ can be upper bounded by $\Var[D_k]≤𝔼[M_k]+𝔼[M_{k-1}]=μ_k+μ_{k-1}$.
We can theoretically upper bound this by $V_k^{ub}=μ_k^{ub}+μ_{k-1}^{ub}$ or empirically estimate it by $V_k^{ub}=M_k+M_{k-1}$.
So by \Cref{lem:bubtest} with $Z_x=D_k^x∈[-1;1]$ the $p$-value for rejecting $\Hiid$ is
\begin{align*}
  p ~&\lesssim~ \textstyle Φ_n((δ_k^{ub}-D_k)/\sqrt{V_k^{ub}}) \tc{\\}
  ~\tc{&}≤~ \exp(-\fr12 n(\bar D_k-\bar{δ}_k^{ub})^2/\bar V_k^{ub}+O(1))~=~ e^{-O(n/k^{5/2})}
\end{align*}
The empirical choice for $V_k^{ub}$ can be smaller=better than the theoretical upper bound if the bound is loose,
but can also be larger=worse, 
since $M_k+M_{k-1}→𝔼[M_k+M_{k-1}]≤μ_k^{ub}+μ_{k-1}^{ub}$ is only guaranteed in the iid case,
but we precisely want to test for non-iid, in which case $M_k+M_{k-1}$ may be larger than $μ_k^{ub}+μ_{k-1}^{ub}$ even asymptotically.
In all of our experiments, the empirical choice $V_k^{ub}=M_k+M_{k-1}$ performed better.

\begin{example}
Testing our previous example where each data item is doubled ($D_2=M_2=n/2$),
for $k=2$ we have $λ_+^*=3$, hence $\bar{δ}_2^{ub}=f_2(λ_+^*)/λ_+^*=1/2e^3\dot=0.0249$ 
and $\bar{μ}_2^{ub}+\bar{μ}_1^{ub}=1/2e+1\dot=1.184 > \fr12=\bar M_2+\bar M_1$,
so in this case, using $\bar M_2+\bar M_1$ as $V_k^{ub}$ is indeed better,
and $p\lesssim\exp(-\fr12 n(\fr12-\fr1{2e})^2/\frac12)=e^{-0.2257n}$.
For $k≠2$, the $D_k$ tests have no power ($\bar M_k=0<\bar{μ}_k^{ub}$).
\eoe\end{example}
We can also lower bound $D_k$ by upper bounding $-D_k$.
The maximizing $λ^*$ is then $λ_-^*$ and $g_δ(λ_-^*)$ is the same (apart from a minus sign) to leading order in $k$,
and $p\lesssim\textstyle Φ_n((|δ_k^{ub}|+D_k)/\sqrt{V_k^{ub}})$.
In the example above, $D_3=-M_2$ would have (the same) power as $D_2$.

\paragraph{\boldmath Linear curvature tests $C_k:=2M_k-M_{k-1}-M_{k+1}$.} 
As apparent from the graphs, the curvature of a (mixture of) Poisson as a function of $k≥2$ is also bounded,
which gives us another test. Let
\begin{align*}
  g_γ(λ) ~&:=~ 2P_λ(k)-P_λ(k-1)-P_λ(k+1) \tc{\\}
  ~\tc{&}=~ \frac{λ^k e^{-λ}}{k!}\left[2-\frac{k}{λ}-\frac{λ}{k+1}\right]
\end{align*}
be the negative curvature of $P_λ$. As before, we need to maximize $g_γ(λ)/λ$ for $k≥2$.
Another tedious calculation shows 
\begin{align*}
  \frac{d}{dλ}\frac{g_γ(λ)}{λ} ~&=~...~=~ \frac{λ^{k-3}e^{-λ}}{k!}⋅\left(1-\frac{k}{λ}\right)\left[\Big(\frac{λ}{k+1}-1\Big)^2-3\right]
\end{align*}
This has 3 zeros with $λ^*=k$ corresponding to the unique maximum of $g_γ(λ)/λ$ (cf.\ Figure~\ref{fig:fT}). 
The other two are minima.
With $n⋅\bar C:=C_k:=2M_k-M_{k-1}-M_{k+1}$ being the empirical negative curvature,
we can upper bound its expectation as 
\begin{align}
  n\bar{γ} ~&:=~ 𝔼[C_k] ~=~ ∑_x g_γ(λ_x) ~≤~ n⋅\max_{λ>0}\frac{g_γ(λ)}{λ} ~=~ n\frac{g_γ(k)}{k} \nonumber\\
  ~&=~ \frac{n}{k(k+1)}\frac{k^k e^{-k}}{k!} = \frac{n(1-\dot{ε}_k)}{k(k+1)\sqrt{2πk}} ~=:~ n\bar{γ}_k^{ub} \label{eq:Ckub}
\end{align}
By \Cref{prop:ublt} we have 
\begin{align*}
  & \Var[C_k] ~≤~ 𝔼[4M_k+M_{k-1}+M_{k+1}] ~=~ 4μ_k+μ_{k-1}+μ_{k+1}\\
  &≤~ n\left[4\frac{1-\dot{ε}_k}{\sqrt{2πk}}+\frac{1-\dot{ε}_{k-1}}{\sqrt{2π(k-1)}}+\frac{1-\dot{ε}_{k+1}}{\sqrt{2π(k+1)}}\right]
  ~=:~ 6n\frac{1+\tilde{ε}_k}{\sqrt{2πk}}
\end{align*}
where $\tilde{ε}_k=O(1/\sqrt{k})$. 
Together by \Cref{lem:bubtest} with $Z_x=2⟦N_x=k⟧-⟦N_x={k-1}⟧-⟦N_x={k+1}⟧∈[-2;2]$,
\begin{align*}
  p ~&\lesssim~ \textstyle Φ_n(\sqrt{n}(\bar{γ}_k^{ub}-\bar C_k)/\sqrt{\bar V_k^{ub}}) \tc{\\}
  ~\tc{&}≤~ \exp(-\fr12 n(\bar C_k-\bar{γ}_k^{ub})^2/\bar V_k^{ub}+O(1))~=~ e^{-O(n/k^{7/2})}
\end{align*}
Similarly to the slope case, we can choose $\bar V_k^{ub}$ as $4\bar M_k+\bar M_{k-1}+\bar M_{k+1}$ or $4\bar{μ}_k^{ub}+\bar{μ}_{k-1}^{ub}+\bar{μ}_{k+1}^{ub}$.
In all of our experiments, the empirical choice performed better, 
but our running example below shows that the theoretical upper bounds can be better in certain circumstances.

\begin{example}
Continuing our previous example where each data item is doubled ($C_2=2M_2=n$),
for $k=2$ we have $\bar{γ}_2^{ub}=1/3e^2=0.0451$ 
and $4\bar{μ}_2^{ub}+\bar{μ}_1^{ub}+\bar{μ}_3^{ub}\dot=1.8260 < 2=\bar 4M_2+\bar M_1+\bar M_3$,
so in this case, using the former as $V_k^{ub}$ is slightly better,
and $p\lesssim\exp(-\fr12 n(1-0.0451)^2/1.8260)\dot=e^{-0.2497n}$.
For $k≠2$, the $C_k$ tests have no power.
\eoe\end{example}

\paragraph{\boldmath Logarithmic curvature tests $\bar U_k:=2\ln M_k-\ln M_{k-1}-\ln M_{k+1}$.} 
One weakness of the tests so far is that they rely on absolute moment bounds.
If all $x$ are equally likely, this is ok, but if $λ_x$ are diverse,
the Poisson mixture becomes wider and hence lower and hence
$𝔼[M_k]$ and its slope and curvature become smaller.
Since the upper bounds must include the worst-case when all $λ_x$ are the same,
the bounds become quite loose. We can fix this by normalizing the curvature by $𝔼[M_k]$.
The mathematics becomes somewhat tedious, 
but there is a more elegant alternative with a very similar effect.
We consider the negative curvature of $\ln M_k$:
\begin{align*}
  \bar U_k ~:=~ 2\ln\bar M_k-\ln\bar M_{k-1}-\ln\bar M_{k+1} \ntc{~=~ 2\ln M_k-\ln M_{k-1}-\ln M_{k+1}}
\end{align*}
This is scale invariant, i.e.\ if all $M_⋅$ in the vicinity of $k$ are scaled down by some factor $α$,
$\bar U_k$ stays unaffected, i.e.\ does not become smaller=weaker. 
Of course this is only useful if we can derive a good upper bound on its expectation.
Since $\bar U_k$ is non-linear we need some new approach:

Consider the function $g:ℝ^3→ℝ$ with $g(\v{\bar Z}):=\bar U_k$,
where $\v{\bar Z}:=(\bar M_k,\bar M_{k-1},\bar M_{k+1})^\trp$.
Noting that $\bar M_k$ concentrates around $\bar{μ}_k$ for large $n$,
we perform a second-order Taylor-series expansion of $g$ around $\v{\bar{ζ}}:=(\bar{μ}_k,\bar{μ}_{k-1},\bar{μ}_{k+1})$.
The CLT then implies $g(\v{\bar Z})≈\text{Gauss}(g(\v{\bar{ζ}}),\Var{[{\v{\bar Z}}^\trp ∇g(\v{\bar{ζ}})]})$.
Formally we use the multivariate delta method, \Cref{lem:mdelta}.
\begin{align*}
  \bar{υ}_k ~:=~ g(\v{\bar{ζ}}) ~=~ 2\ln\bar{μ}_k-\ln\bar{μ}_{k-1}-\ln\bar{μ}_{k+1} \ntc{~=~ 2\ln μ_k-\ln μ_{k-1}-\ln μ_{k+1}}
\end{align*}
Let us define some auxiliary probability distribution over $x$ 
solely for technical purposes without ascribing any meaning to it.
\begin{align*}
  & \tilde P_{\vl,k}[X=x] ~:=~ \frac{λ_x^k e^{-λ_x}/k!}{∑_x λ_x^k e^{-λ_x}/k!}, ~~~\text{then}~~~ \\
  & \frac{μ_{k+1}}{μ_k} ~=~ \frac{∑_x \frac{λ_x}{k+1}λ_x^k e^{-λ_x}/k!}{∑_x λ_x^k e^{-λ_x}/k!} ~=~ \frac{\tilde{𝔼}_{\vl,k}[λ_X]}{k+1} \\
  & \frac{μ_{k-1}}{μ_k} ~=~ \frac{∑_x \frac{k}{λ_x}λ_x^k e^{-λ_x}/k!}{∑_x λ_x^k e^{-λ_x}/k!} ~=~ k⋅\tilde{𝔼}_{\vl,k}\bigg[\frac1{λ_X}\bigg] ~≥~ \frac{k}{\tilde{𝔼}_{\vl,k}[λ_X]}
\end{align*}
where we applied Jensen's inequality in the last step to convex function $1/λ$.
Taking the product, the dependence on unknown $\vl$ cancels out:
\begin{align}\label{eq:Ukub}
  \bar{υ}_k ~=~ \ln\frac{μ_k}{μ_{k-1}}\frac{μ_k}{μ_{k+1}} ~≤~ \ln\frac{k+1}{k} =: \bar{υ}_k^{ub} ~≤~ \frac1k 
\end{align}
Jensen's inequality for $1/λ$ is sharp \emph{iff} $\tilde P_{\vl,k}$ is a Dirac measure.
That is, as before, the bound is attained when all $λ^*=λ_x$ are the same,
but unlike before, the maximum is flat and attained for \emph{any} $λ^*$, 
even far away from $k$.
Indeed, the log-curvature $\bar{υ}_k$ is the same, namely $\bar{υ}_k^{ub}=\ln\fr{k+1}{k}$, for all $λ^*$.
As for the variance,
\begin{align*}
  \Var[{\v{\bar Z}}^\trp ∇g(\v{\bar{ζ}})] ~&=~ \Var\left[ 
  \scriptsize\left(\!\!\begin{array}{c}\bar M_k~~ \\ \bar M_{k-1} \\ \bar M_{k+1}\end{array}\!\!\!\right)^{\!\!\trp}\!
  \left(\!\!\!\begin{array}{c}2/\bar{μ}_k \\ -1/\bar{μ}_{k-1} \\ -1/\bar{μ}_{k+1}\end{array}\!\!\!\right)\right] \tc{\\}
  ~\tc{&}=~ \Var\left[2\frac{M_k}{μ_k}-\frac{M_{k-1}}{μ_{k-1}}-\frac{M_{k+1}}{μ_{k+1}}\right] 
\end{align*}
With $Z_k^x:=⟦N_x=k⟧$ and $α_k:=2/μ_k$ and $α_{k±1}:=-μ_{k±1}$ and all other $α_{k'}:=0$,
we have $Z_k^+=M_k$, hence 
\begin{align*}
  \Var[{\v{\bar Z}}^\trp ∇g(\v{\bar{ζ}})] ~&=~ \Var[∑_x α_k Z_k^+] ~≤~ ∑_k α_k^2 𝔼[Z_k^+] \tc{\\}
  ~\tc{&}=~ \frac1{μ_{k-1}} + \frac4{μ_k} + \frac1{μ_{k+1}}
\end{align*}
where we used \Cref{lem:Zkxcor}a in the inequality and $𝔼[Z_k^+]=𝔼[M_k]=μ_k$ in the last equality.

We need an approximation or upper bound on this, but we only have lower bounds on $1/μ_k$.
We solve this problem by replacing $μ_k$ with their empirical estimates $M_k=μ_k±O_P(\sqrt{n})$.
This is the second time we are forced to use the empirical estimate to upper bound the variance,
but for a slightly different reason than for $E$ and $O$.
Using \Cref{lem:mdelta} with $\v Z_x:=(Z_k^x,Z_{k-1}^x,Z_{k+1}^x)∈\{0,1\}^3$,
the $p$-value for the logarithmic curvature test is 
\begin{align*}
  p ~&\lesssim~ Φ_n\left(\sqrt{n}\frac{\bar{υ}_k^{ub}-\bar U_k}{\textstyle\sqrt{\bar M_{k-1}^{-1}+4\bar M_k^{-1}+\bar M_{k+1}^{-1}}}\right) 
  ~=~ e^{-O(n/k^{7/2})}
\end{align*}

\paragraph{Summary.}
In the table below we summarize the most important quantities for the tests $T:𝓧^n→ℝ$ derived in this section.
For $τ:=𝔼[T]$ we derived tight upper bounds even for small $k$.
We only show the $n\gg k\gg 1$ approximations in the table and refer to the exact expressions.
For the variance $V_k^{ub}$ we only show the better upper bound (empirical except for $M_k$).
\\[2ex]{\setlength\tabcolsep{2pt}
\begin{tabular}{r|l|l|c|c|c}
  Test Name & $T:=n\bar T:=$ & $\bar{τ}:=𝔼[\bar T]≤$ & $\Var[\bar T]\lesssim\bar V^{ub}=$ & $λ^*$ & $O(\ln\fr1p$) \\ \hline
  Even$≠0$  & $E:=∑_x N_x⟦N_x≠0~\text{even}⟧$ & $\bar{ε}^{ub}=1/2$ & $\fr1n ∑_{k≠0~\text{even}}k^2 M_k$ & $∞$ & $n$ \\
  Odd$≠1$   & $O:=∑_x N_x⟦N_x≠1~\text{odd}⟧$  & $\bar{ο}^{ub}=1/2$ & $\fr1n ∑_{k≠1~\text{odd}}k^2 M_k$ & $∞$ & $n$ \\
  2nd-Count & $M_k:=∑_x⟦N_x=k⟧$             & $\bar{μ}_k^{ub}\smash{\stackrel{\eqref{eq:Mkub}}{=}}\frac{1-\dot{ε}_{k-1}}{k\sqrt{2π(k-1)}}$ & $\bar{μ}_k^{ub}$ & $k-1$ & $\frac{n}{k^{3/2}}$ \\
  Slope     & $D_k:=M_k-M_{k-1}$                  & $\bar{δ}_k^{ub}\smash{\stackrel{\eqref{eq:Dkub}}{=}}\frac{1-\ddot{ε}_k}{k^2\sqrt{2πe}}$ & $\bar M_k+\bar M_{k-1}$ & \smash{${k-1/2+\atop\sqrt{k\!+\!1/4}}$} & $\frac{n}{k^{5/2}}$ \\ 
  Lin.Curv. & $C_k:=2M_k\!-\!M_{k-1}\!-\!M_{k+1}$         & $\bar{γ}_k^{ub}\smash{\stackrel{\eqref{eq:Ckub}}{=}}\frac{1+\tilde{ε}_k}{k(k+1)\sqrt{2πk}}$ & $4\bar M_k+\bar M_{k-1}+\bar M_{k+1}$ & $k$ & $\frac{n}{k^{7/2}}$ \\
  Log.Curv. & $\bar U_k:=\ln(M_k^2/M_{k-1}M_{k+1})$    & $\bar{υ}_k^{ub}\smash{\stackrel{\eqref{eq:Ukub}}{=}}\ln\frac{k+1}{k}≤\frac1k$ & $\bar M_{k-1}^{-1}\!+\!4\bar M_k^{-1}\!+\!\bar M_{k+1}^{-1}$ & any & $\frac{n}{k^{7/2}}$
\end{tabular}}

\begin{theorem}[\bfm IID tests]\label{thm:iidtests}
Consider the test statistics $\bar T$ and associated upper bounds 
on their mean $\bar{τ}$ and variance $\Var[\bar T]$ from the above table.
Then test $\bar T(x_{1:n})≥\sqrt{n}[\bar{τ}^{ub}+z_α\sqrt{\bar V^{ub}}]$
rejects that $x_{1:n}$ is iid with confidence $\gtrsim 1-c_n α$, i.e.\ at significance level $\lesssim c_n α$,
where $z_α:=Φ^{-1}(1-α)$.
The accuracy of $\lesssim$ is $O_P(1/n\sqrt{\Var[\bar T]})$,
except for $E$ and $O$ for which it is $O((∑_x θ_x^3)/(∑_x θ_x^2)^{3/2})$.
See \Cref{lem:bubtest} (with $Z_+=T$ and $Z_x=T_x$) for $p$-values and further details.
\end{theorem}

In our experiments, $θ_x=1/d$ ({\sf uniform}) and $θ_x=2x/d/(d+1)$ ({\sf linear}),
so $(∑_x θ_x^3)/(∑_x θ_x^2)^{3/2}≤1.3/\sqrt{d}$ with $d=30...100$, 
so the Gaussian approximation is not great but ok for $α=0.05$.
For the other tests, the Gaussian approximation is quite good.

\begin{proof} 
Follows directly from the derivations in this section and \Cref{lem:bubtest} 
and the fact that all $T$ have the required decomposition $T=∑_x Z_x$ ($\v Z_x$ for $\bar U_k$ within $O_P(\frs1n)$).
$Z_x$ is bounded, so \Cref{lem:bubtest}(iii) applies, except for $E$ and $O$:

Let $E^x:=∑_{0≠k≠n~\text{even}} k⋅⟦N_x=k⟧$, then $E=∑_x E^x$.
Similar to the derivation of $𝔼[E]$ one can show for $α>0$ that 
\begin{align*}
  𝔼[|E^x-𝔼[E^x]|^α] ~&\simequd~ (\fr12 λ_x)^α \\
  \text{This implies}~~~ σ_+^2 ~&≡~ \textstyle ∑_x\Var[E^x] ~\simequd~ (\fr12 λ_x)^2 \\
  \text{and}~~~~~~~~~~~~~~ ρ_+ ~&≡~ \textstyle ∑_x 𝔼[|E^x-𝔼[E^x]|^3] ~\simequd~ (\fr12 λ_x)^3 \\
  \text{hence}~~~~~ ρ_+/σ_+^{3/2} ~&\simequd~ \textstyle (∑_x θ_x^3)/(∑_x θ_x^2)^{3/2}.
\end{align*}
The condition in \Cref{lem:bubtest}(iii) applies if $ρ_+/σ_+^{3/2}=O(1/\sqrt{n})$.
For the CLT to hold asymptotically $ρ_+/σ_+^{3/2}→0$ suffices; see \Cref{thm:esseen}.
The expressions for $O$ are the same.
\qed\end{proof}

\section{Toy/Control Experiments}\label{sec:exp}

We verify the tests developed in \Cref{sec:iidtests} on artificially generated data.
We generate iid data for the extremes of all $θ_x$ being the same,
and $\vt$ being maximally diverse. 
This is used for testing the validity of our tests (correct low Type~I error).
We then ``corrupt'' the samples in various ways to create non-iid data 
to determine the power of the tests in rejecting $\Hiid$ (low Type~II error).
For instance, some tests are able to detect data duplication and draws from finite card decks.
Every test displayed its own strengths and weaknesses.
There was no uniformly best test among them.
This is not meant to be a comprehensive evaluation of the tests,
but a sanity check that the tests work as intended.

\emph{Remark:} As discussed in \Cref{sec:XtoN} we can restrict our attention to $𝓧=ℕ$.
On the other hand, it makes no difference whether we sample from $𝓧$ or $𝓧':=\{x:θ_x>0\}$,
so experimentally we can as well assume that $d=d'=|𝓧'|$,
but given the caveat described in \Cref{sec:XtoN},
we still should imagine $𝓧=ℕ$ and the parameter $d$ in the experiments below now decoupled from $𝓧$.
Alternatively, mentally replace every $d$ in this section by $d'$.

\paragraph{Data generation.}
\emph{Iid sampling} ({\sf iid})
For the iid distributions $P_\vt$ we tested two ``extreme'' choices.
({\sf uniform}) One in which all $\vt$ are the same for $d$ categories, and $0$ for all others,
i.e.\ w.l.g.\ $P_\vt[x]=θ_x=1/d~∀x∈\{1:d\}$, and $P_\vt[𝓧\setminus\{1:d\}]=0$. 
This should be the hardest case to not accidentally reject $\Hiid$,
since $𝔼[M_k]$ \eqref{eq:EMk} is maximally peaked out (Figure~\ref{fig:MEMPV1} top left).
({\sf linear}) The other extreme is for which 
the probabilities $θ_x$ of categories $x$ are equally/uniformly ``distributed'',
i.e.\ $θ_x\propto x$ for $x∈\{1:d\}$, i.e.\ $θ_x=2x/d(d+1)$.
$𝔼[M_k]$ is maximally washed out in this case (Figure~\ref{fig:MEMPV2} top left).

\emph{Exact data duplication} ({\sf even-n})
We created non-iid distributions $Q$ out of the iid ones as follows:
To mimic the data duplication problem, we sampled $x_{1:n/2}$ iid
and then duplicated each item to $x_{1:n/2}x_{1:n/2}$,
and then shuffled (though the tests only depend on the counts $n_x$, so shuffling is not necessary).
This makes all $n_x$ even, hence all $m_k=0$ for odd $k$.

\emph{Approximate data duplication} ({\sf even-m})
We also tested duplication and then injectively corrupt the data to $x'_{1:n/2}$,
so that all $x'_{1:n/2}$ differ from all $x_{1:n/2}$.
This mimics approximate duplicates.
The effect is that for each count $n_x$ there is a deterministically corrupted $x'$ with same count $n_{x'}$,
which in effect means that all $m_k$ are even.
This is a hard signal to detect without explicitly searching for it,
and indeed our tests don't. 
See \Cref{sec:app} for more discussion.

\emph{No empty categories} ({\sf no-empty})
We also sampled, $\v x$ iid and then increased the count $n_x$ for each $x∈\{1:d\}$ by 1.
This eliminates all empty categories for $x≤d$,
but note that $𝓧$ itself is intended to be infinite,
so does not really remove \emph{all} empty categories. 
Technically we sample $x_{1:n-d}$ iid from $\{1:d\}$ and then add $x_{n-d+x}=x$ for $x∈\{1:d\}$ and then shuffle.
Such $x_{1:n}$ is \emph{not} iid w.r.t.\ any $P_\vt$,
but the signal is in general very weak, and none of our tests were able to pick it up.

\emph{No unique\&empty categories} ({\sf no-unique})
Finally we increased each count by two so that every data item appears at least twice,
but unlike {\sf even-n} can also appear an odd number of times.
Technically we sample $x_{1:n-2d}$ iid from $\{1:d\}$ 
and then add $x_{n-2d+x}=x_{n-d+x}=x$ for $x∈\{1:d\}$ and then shuffle.
As long as the original iid sample has enough items of low multiplicities,
there is a clear signal that the data is non-iid as explained in \Cref{sec:app}.

\paragraph{Choice of ``hyper''-parameters.}
We also have to choose $d$, $n$, and $k$. 
Since this is not a systematic empirical study, not even on toy data, 
but only to illustrate the tests and corroborate the theoretical arguments,
we chose some arbitrary and some interesting values without any claim of coverage.
$k$ is typically chosen where the tests are strongest,
essentially where $μ_k$ is large.

\paragraph{Test setup and graphs.}
We tested our tests ($T∈\{E,O,M_k,D_k,C_k,\bar U_k\}$) for over/under-confidence 
on the artificial iid data and their power on the artificial non-iid data above.
We sampled data $x_{1:n}$ a \smax\ times from $P$ or $Q$, and computed \smax\ $p$-values for each test.
Let $\tilde T:=Φ((τ^{ub}-T)/\sqrt{V^{ub}})$, 
where $τ^{ub}$ and $V^{ub}$ are the upper bounds for mean and variance of $T$ 
we derived for our tests (see test table before \Cref{thm:iidtests}).
We report $\hat P[\tilde T≤α]$, the fraction of $p$-values below $α$, 
as a function of $α$ (Figures~\ref{fig:MEMPV1} left).
For an ideal uniformized test (see \Cref{sec:moretests}), $P_\vt[\tilde T≤α]=α$, 
but this rarely possibly to achieve simultaneously for \emph{all} $\vt$.
If data is sampled iid from $P_\vt$, a valid test should be below this diagonal line,
or at least not much above, or at the very least 
(approximately) below $α$ for the $α$ we care about, typically $α=0.1\%...5\%$.
For non-iid data we want the power $β(α):=Q[\tilde T≤α]$ to be as large as possible,
ideally close to 1 for the $α$ we care about, 
since this is the probability our test correctly rejects $\Hiid$ if data is sampled from $Q$.

We also plot the empirical second-order counts $M_k$ as a function of $k$ for one sample $x_{1:n}$,
together with their true \emph{uncorrupted} expectation $𝔼_\vt[M_k]$,
which in some informal sense is the iid distribution closest to the corrupted non-iid distribution.
We also plot the empirical average $\langle M_k\rangle$ over the \smax\ runs,
which in case of iid data is very close to the true mean $𝔼_\vt[M_k]$ (Figures~\ref{fig:MEMPV1}\&\ref{fig:MEMPV2} left).

\paragraph{Explanation of the figures.}
The legends in Figures~\ref{fig:MEMPV1}\&\ref{fig:MEMPV2} also display 
the hyper-parameters $d$, $n$ (and $k$ for the right graphs),
together with the sampling procedure ({\sf uniform$|$linear}) and corruption model
({\sf iid$|$even-n$|$no-unique}).
Test $M$ uses the theoretical upper bound $μ^{ub}$ for its variance,
all other tests use the empirical upper bounds for their variance.
{\sf u-test} is a uniformly at random sampled ``test'' $u\sim$Uniform$[0;1]$ for control purposes, 
and should be very close to the diagonal ({\sf exact}); 
deviations are due to the finite sample approximation.
The \% behind the test is the fraction of times $\hat{β}=\hat Q[p≤0.05]$ (or $P_\vt$ if iid), 
the test rejected $\Hiid$ at 5\% significance level.
The true reject probability is $\hat{β}±O({\hat{β}(1-\hat{β})/\smax})^{1/2}$.
In each graph, we also plotted the \smax\ $p$-values 
for the (at $α=5\%$) most powerful test (with uniformly at random $y$-component).
For a perfect test on iid data, these points would be uniformly distributed in $[0;1]×[0;1]$.

\begin{figure*}[ptbh!]
\begin{center}\def\vsl{0.225}\def\vsr{0.225}
\includegraphics[width=0.47\textwidth,height=\vsl\textheight]{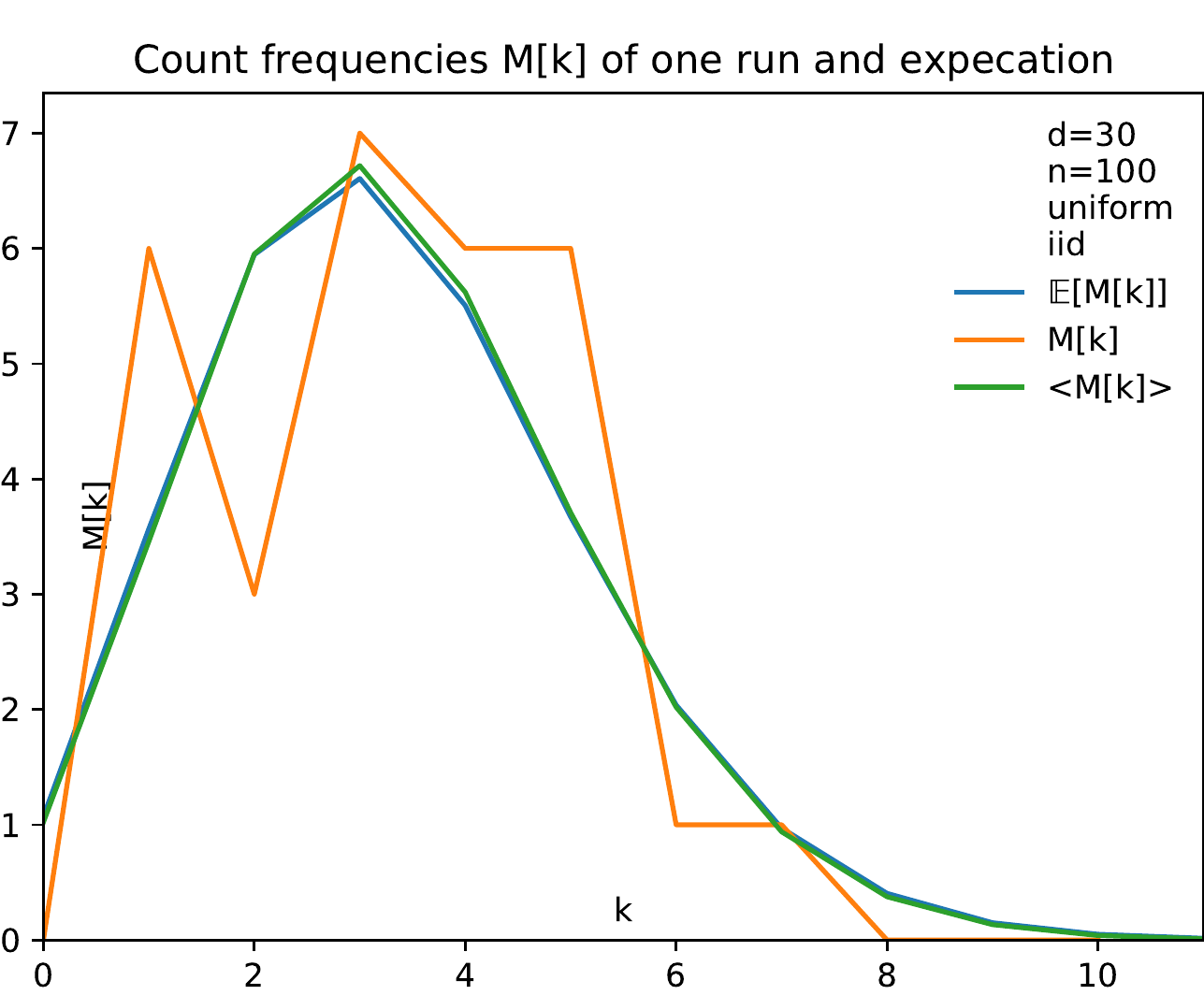}~~%
\includegraphics[width=0.51\textwidth,height=\vsr\textheight]{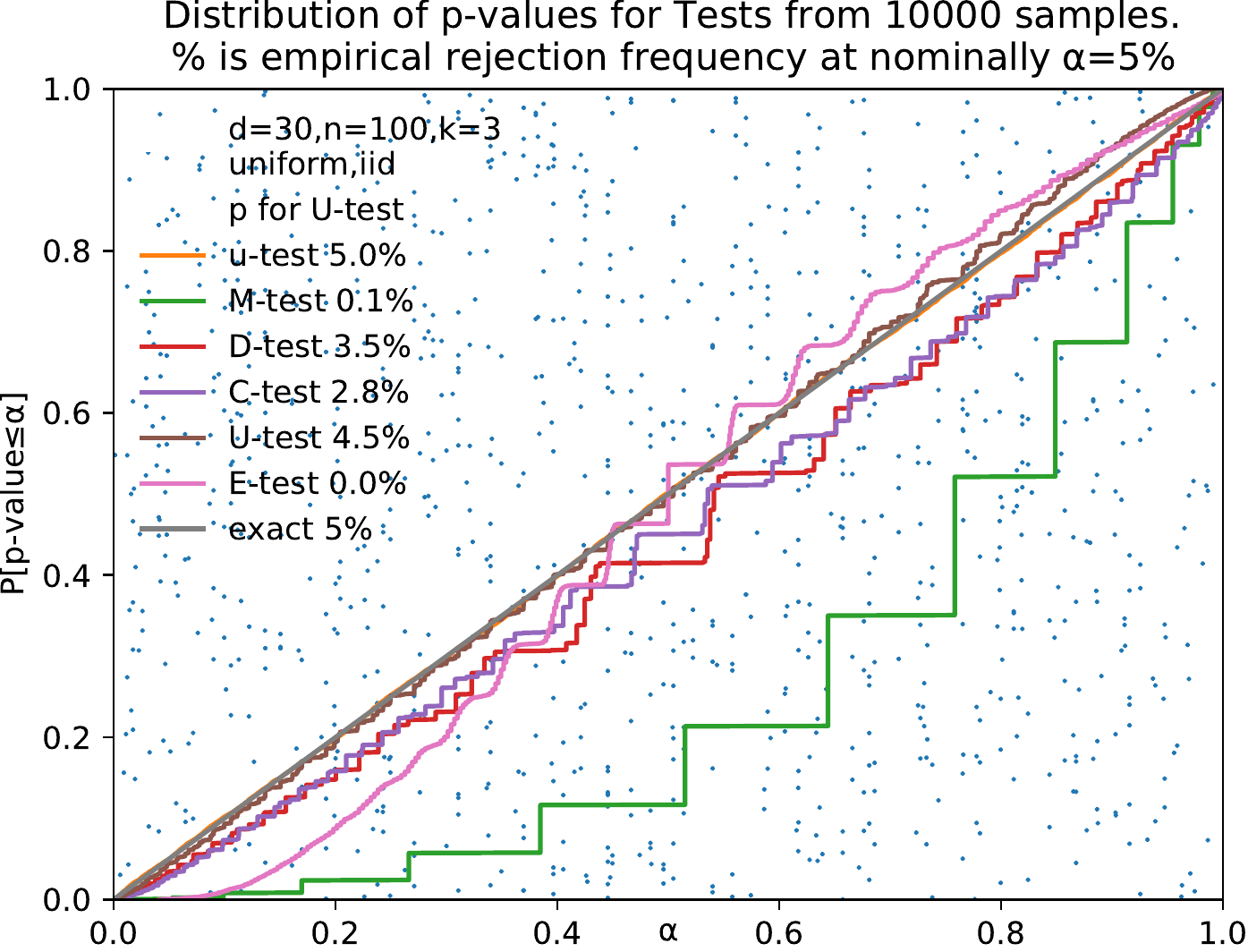}

\includegraphics[width=0.47\textwidth,height=\vsl\textheight]{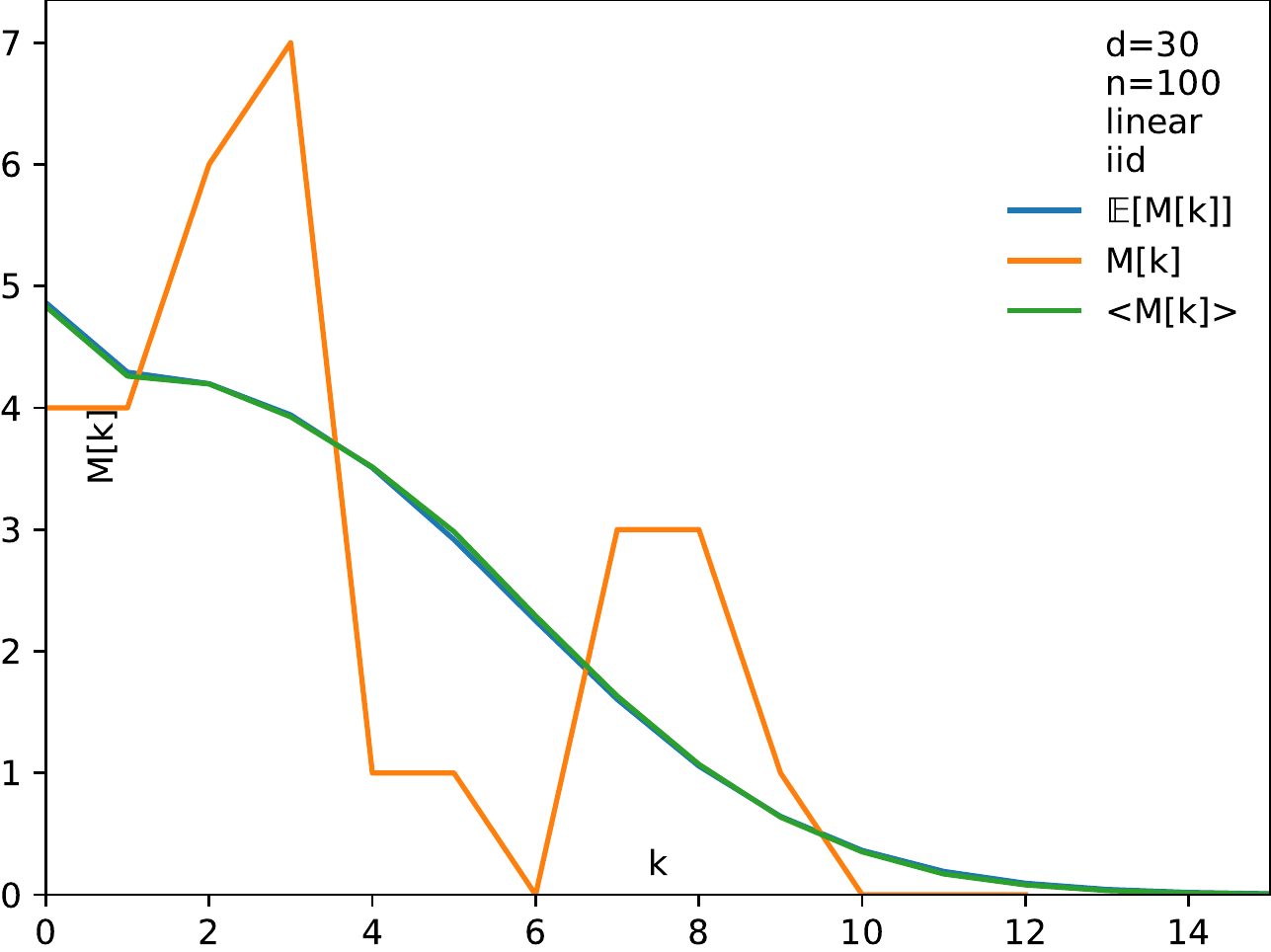}~~%
\includegraphics[width=0.51\textwidth,height=\vsr\textheight]{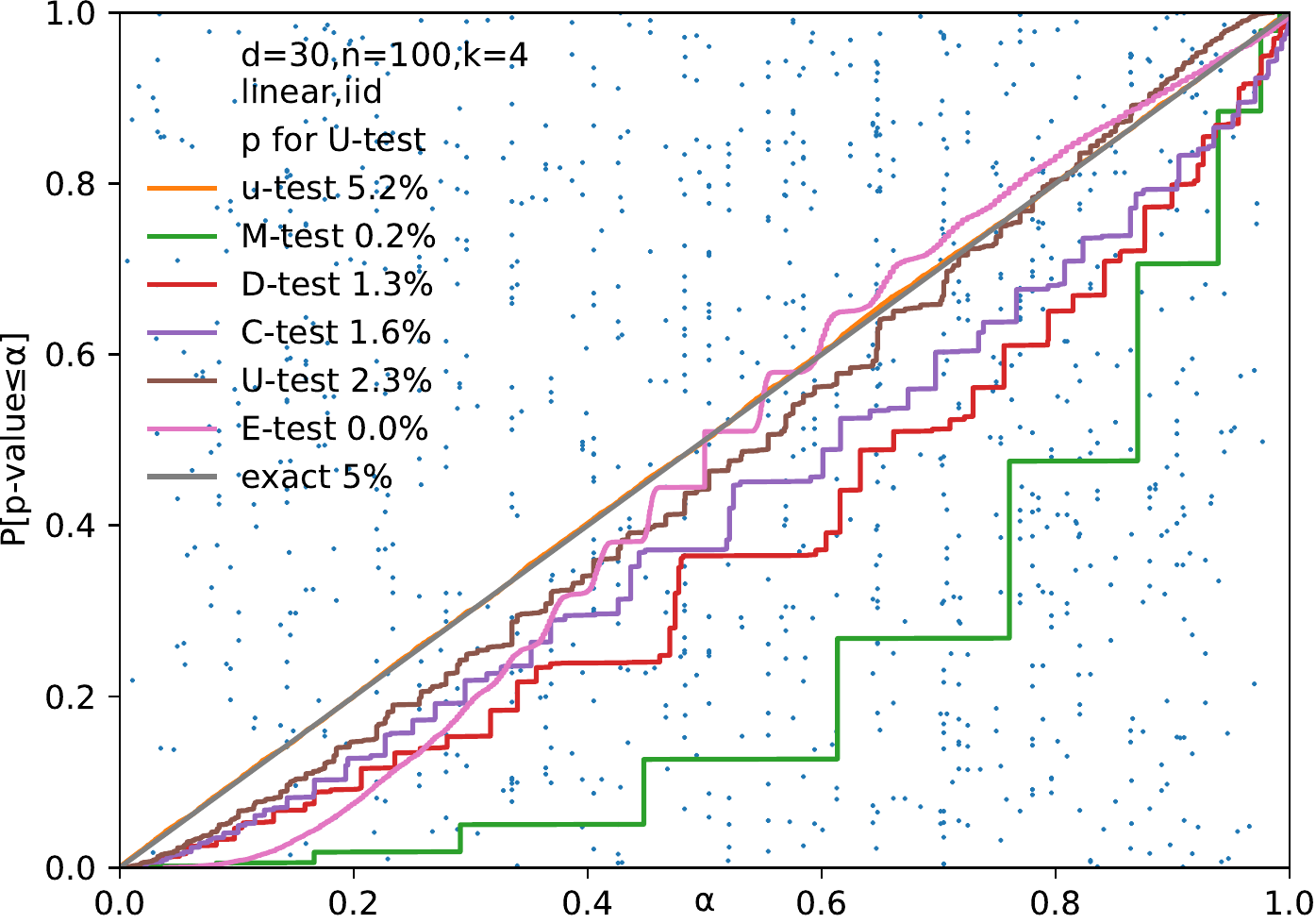}

\includegraphics[width=0.47\textwidth,height=\vsl\textheight]{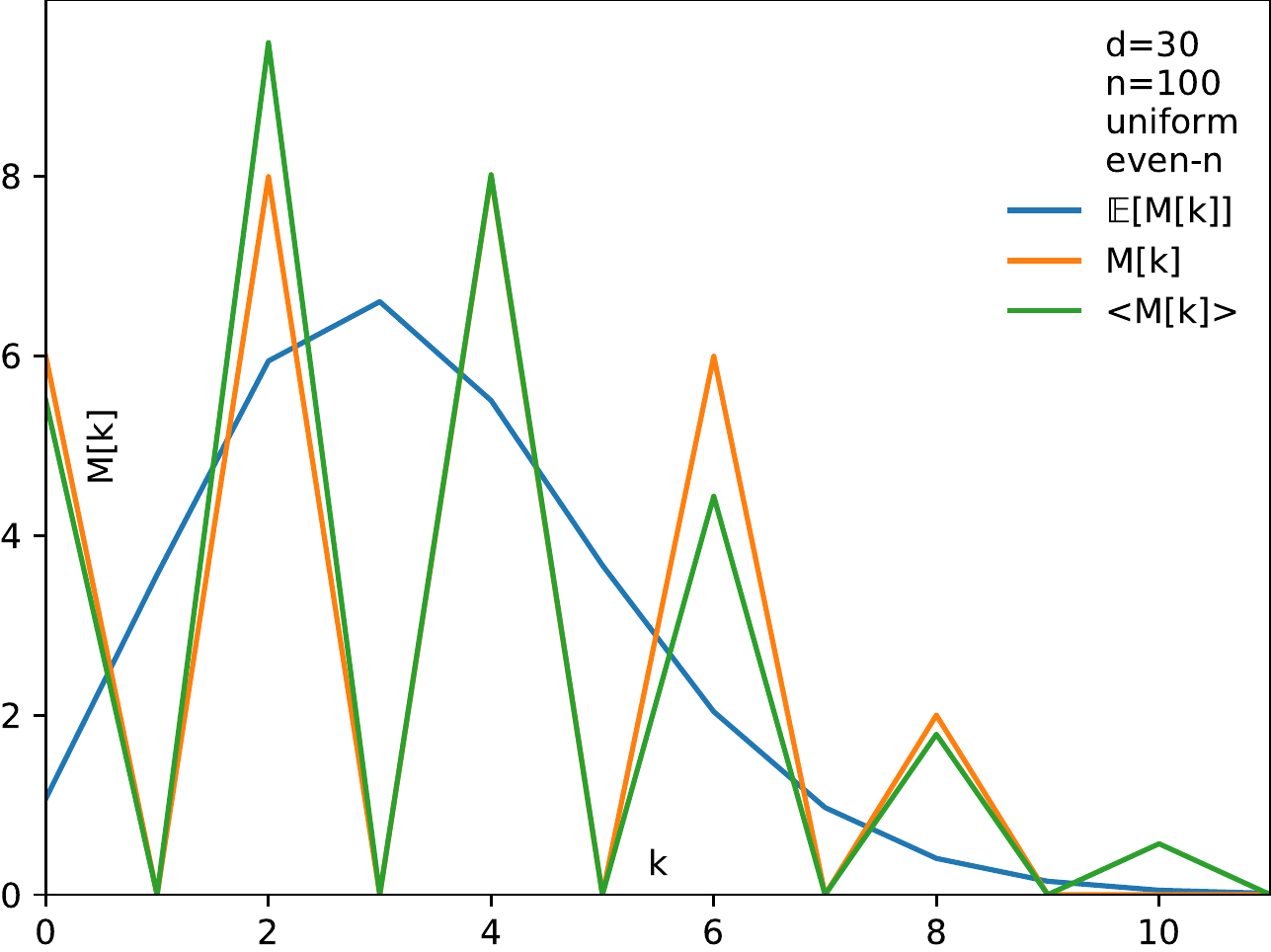}~~%
\includegraphics[width=0.51\textwidth,height=\vsr\textheight]{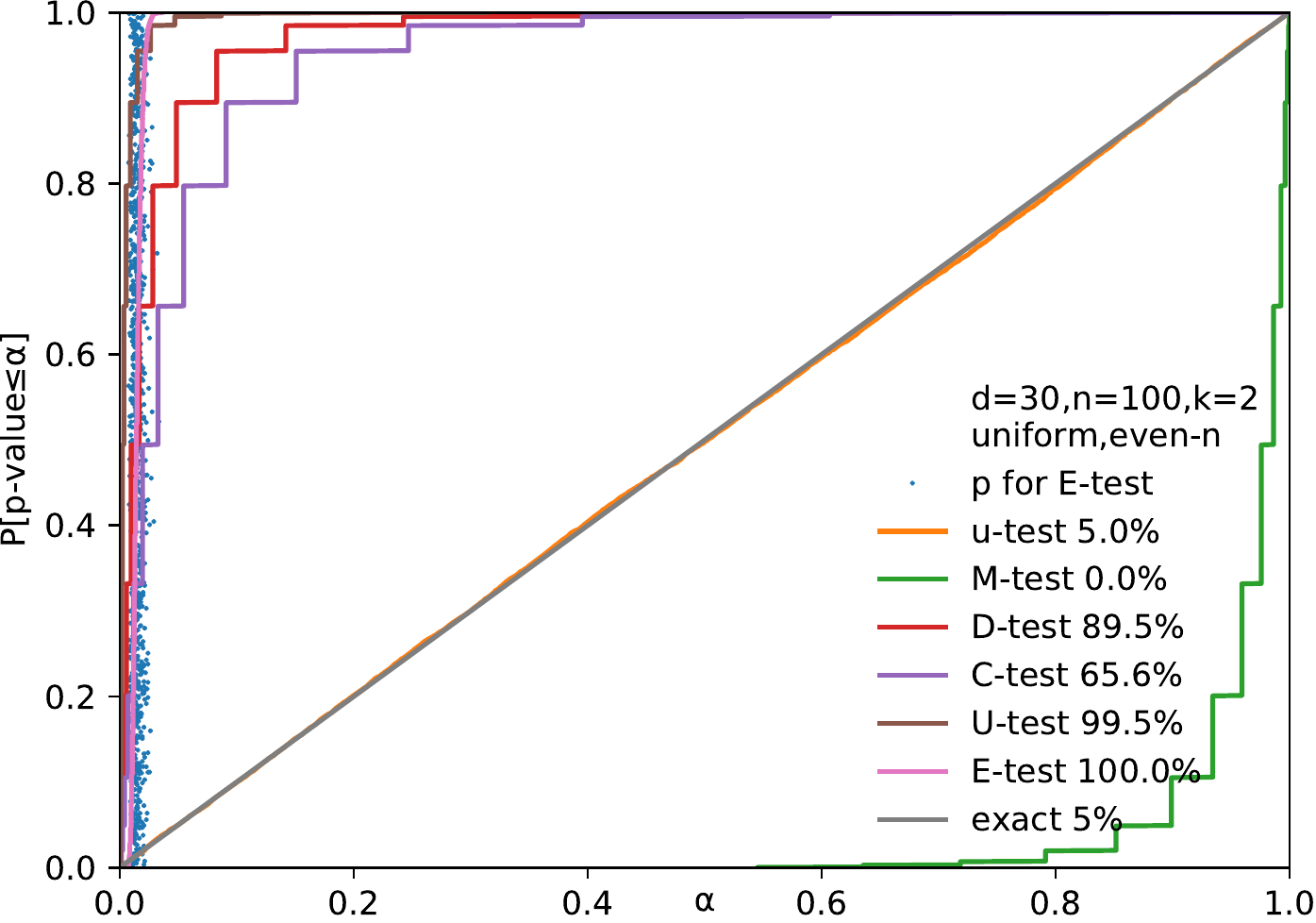}

\includegraphics[width=0.47\textwidth,height=\vsl\textheight]{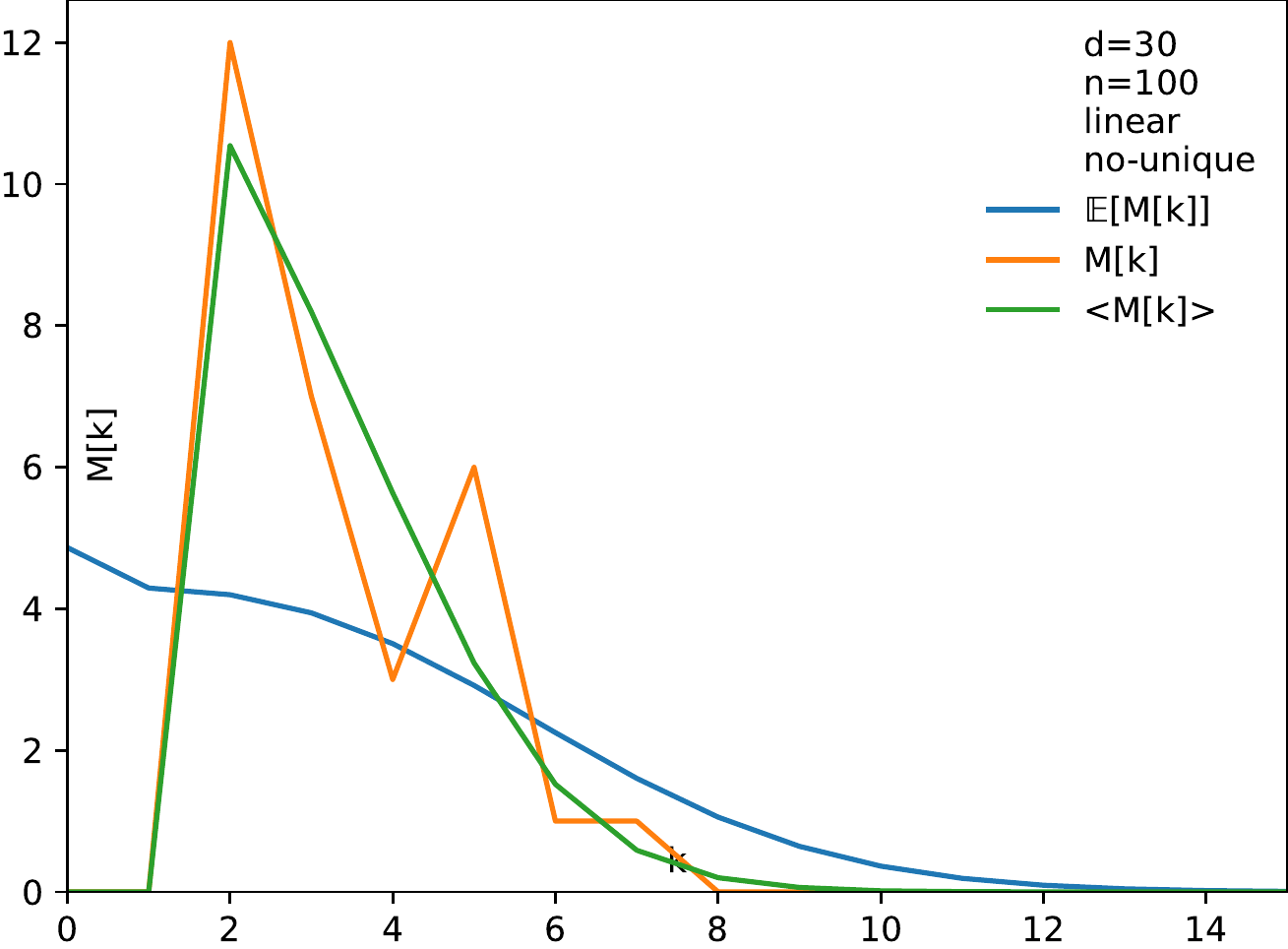}~~%
\includegraphics[width=0.51\textwidth,height=\vsr\textheight]{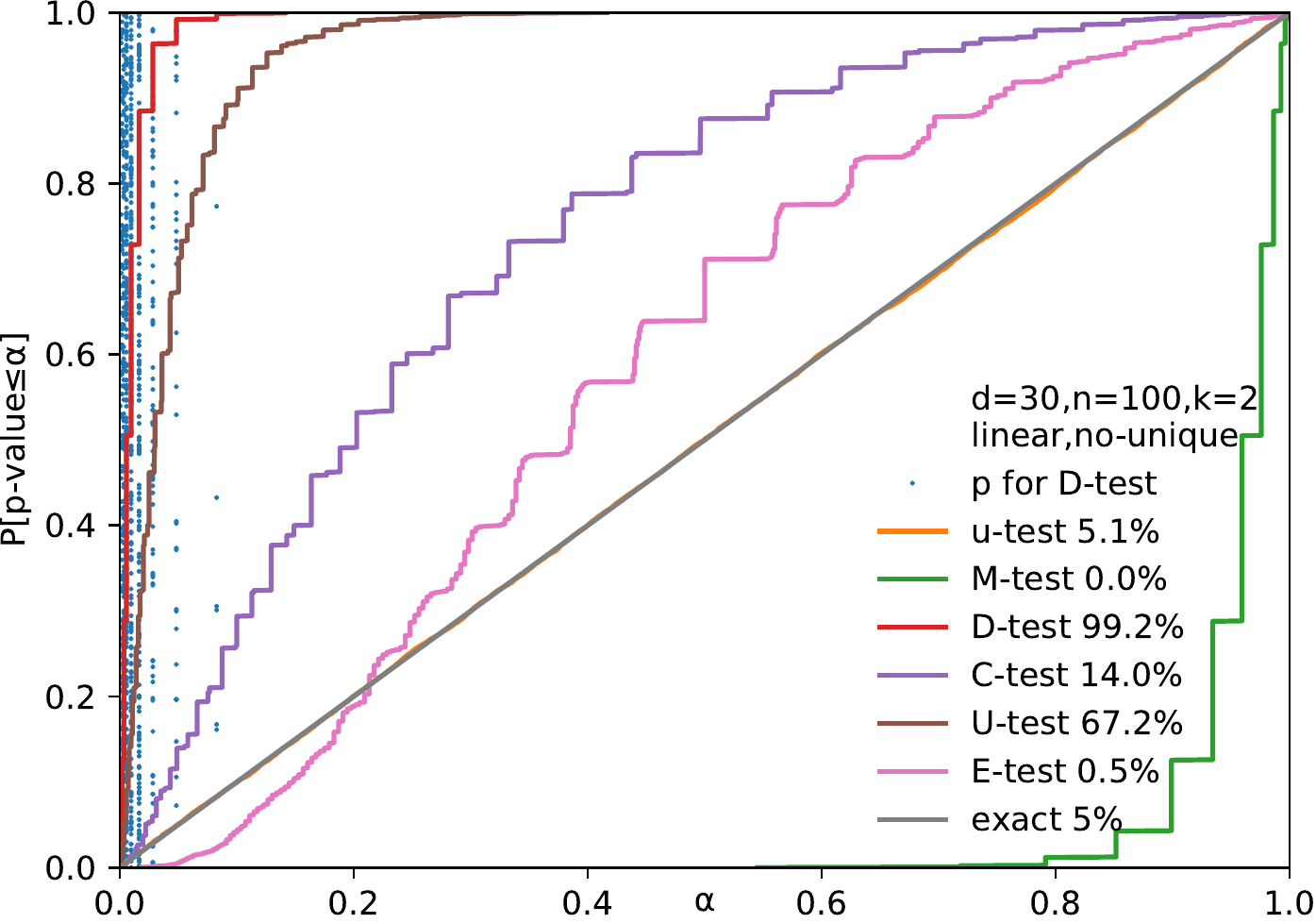}\vspace*{-2ex}
\caption{{(\bf Testing the tests)} 
\smax\ data sets $x_{1:n}$ are sampled iid from $P_\vt$ or non-iid $Q$.
\emph{(left)} One sample, average, and expected second-order counts $M_k$ as a function of $k$.
\emph{(right)} $p$-value distribution of various tests.
For iid $P_\vt$, a curve above/below the diagonal means over/under-confidence, 
i.e.\ we want on or below. For non-iid we want far above.
\% is the empirical rejection frequency at nominally $α=5\%$.
}\label{fig:MEMPV1}\vspace*{-2ex}
\end{center}
\end{figure*}

\begin{figure*}[ptbh!]
\begin{center}\def\vsl{0.225}\def\vsr{0.225}
\includegraphics[width=0.47\textwidth,height=\vsl\textheight]{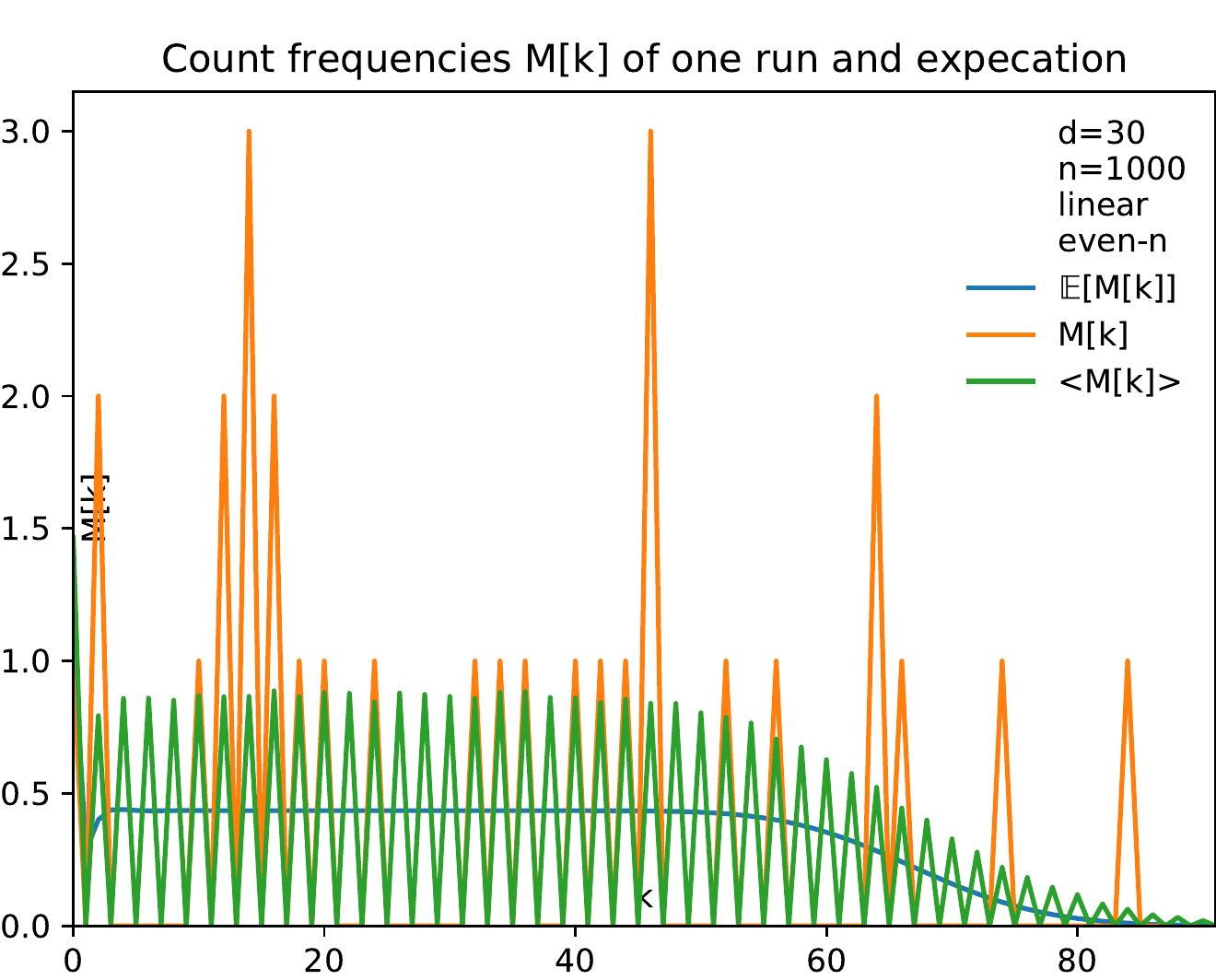}~~%
\includegraphics[width=0.51\textwidth,height=\vsr\textheight]{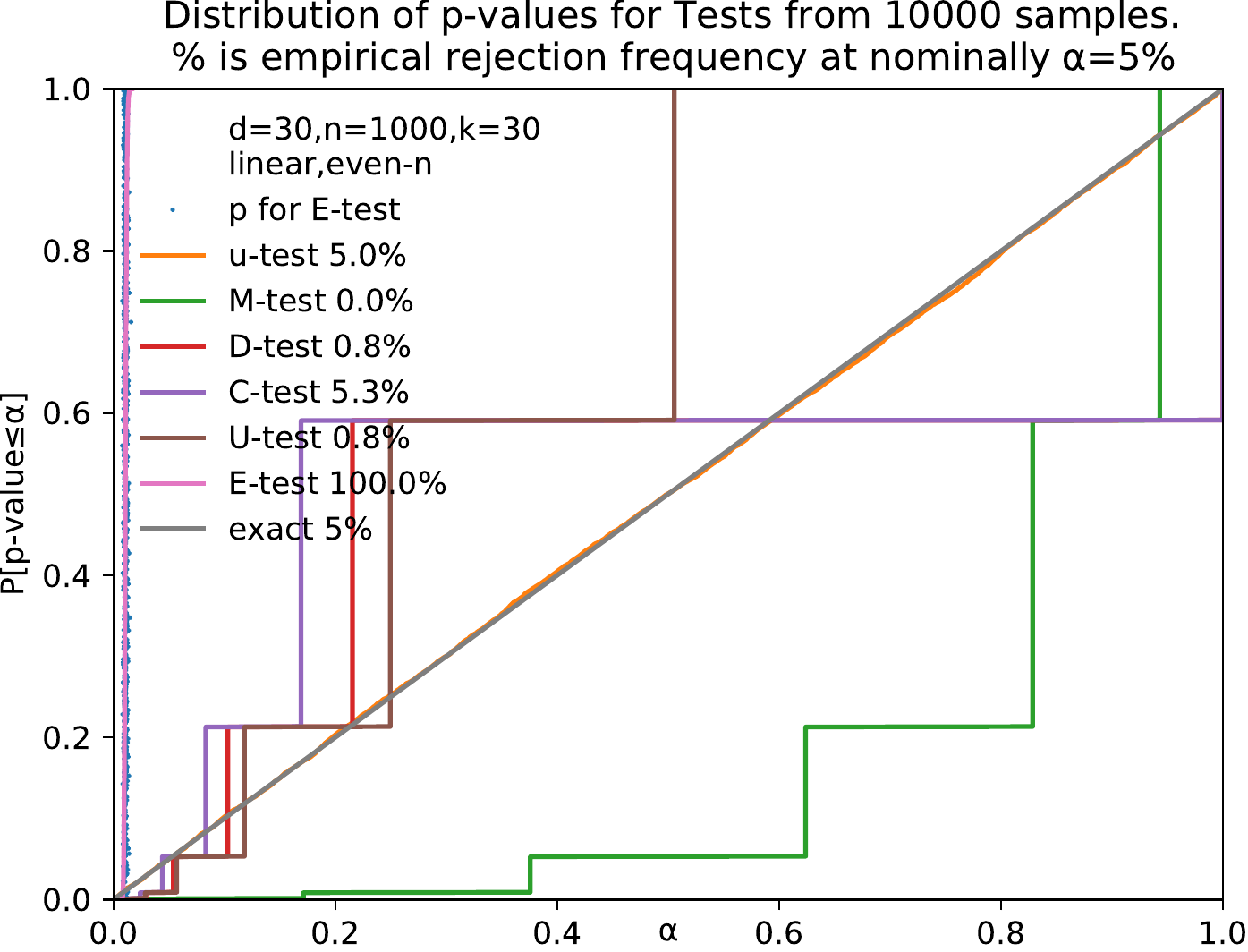}

\includegraphics[width=0.47\textwidth,height=\vsl\textheight]{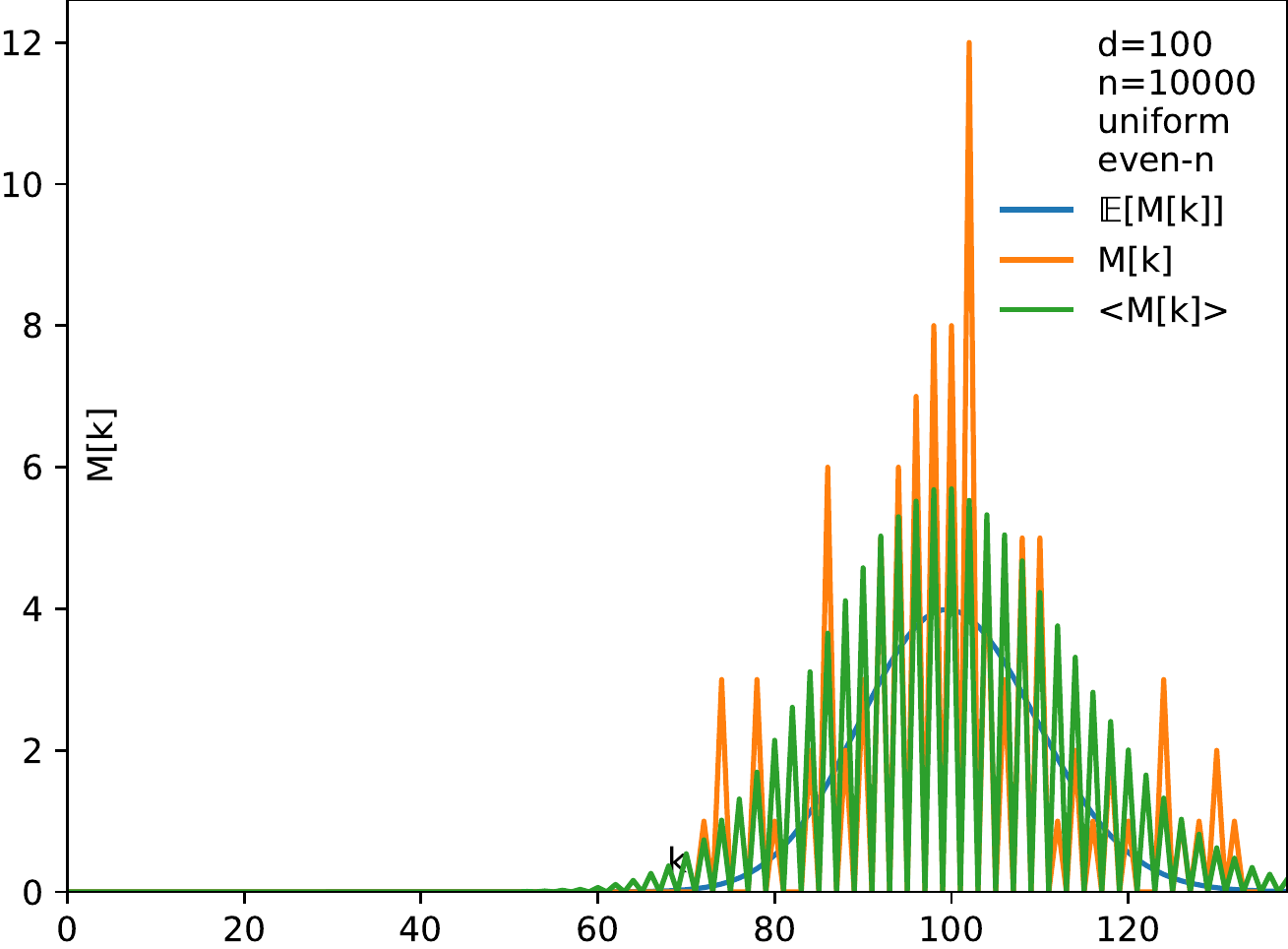}~~%
\includegraphics[width=0.51\textwidth,height=\vsr\textheight]{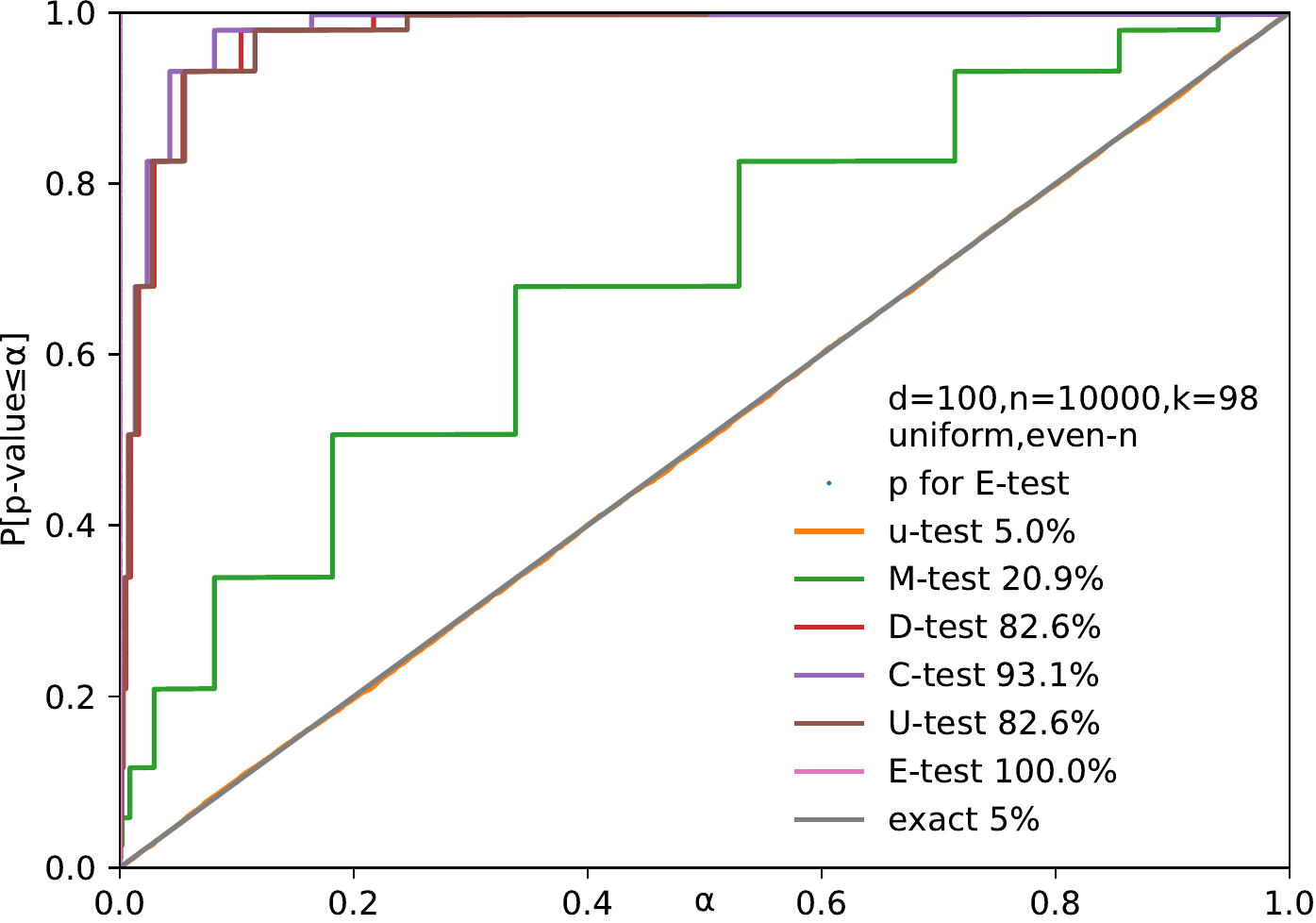}

\includegraphics[width=0.47\textwidth,height=\vsl\textheight]{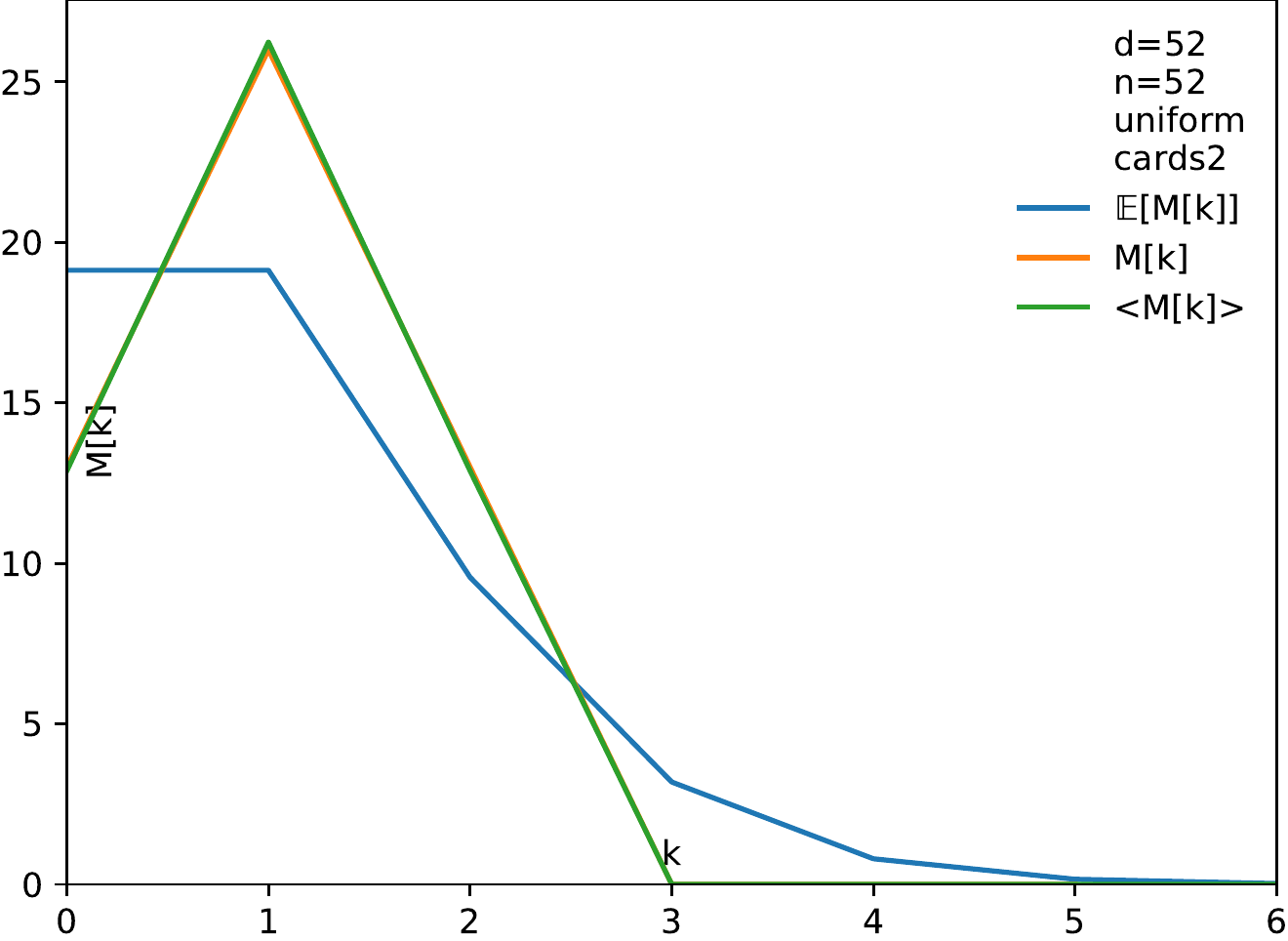}~~%
\includegraphics[width=0.51\textwidth,height=\vsr\textheight]{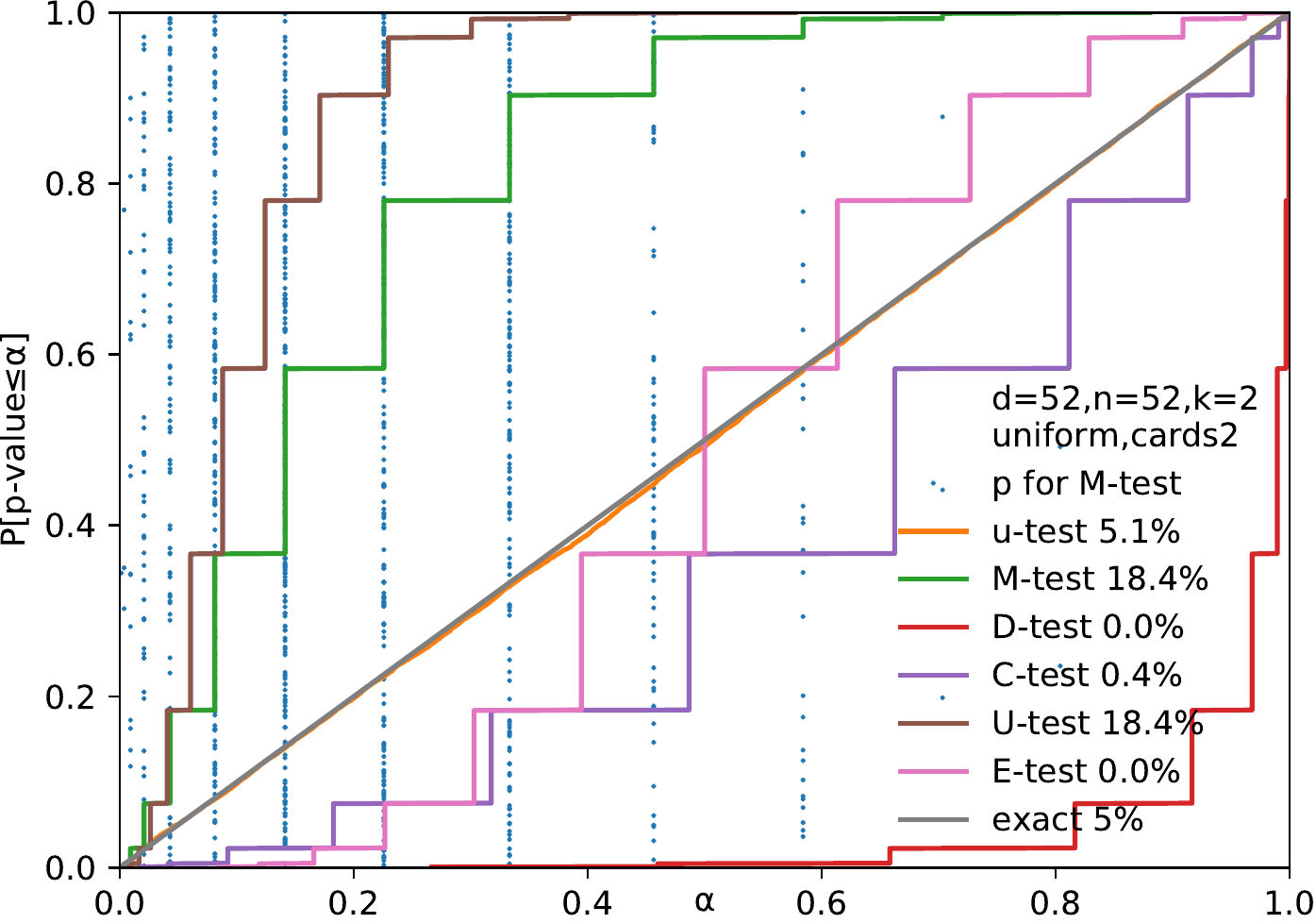}

\includegraphics[width=0.47\textwidth,height=\vsl\textheight]{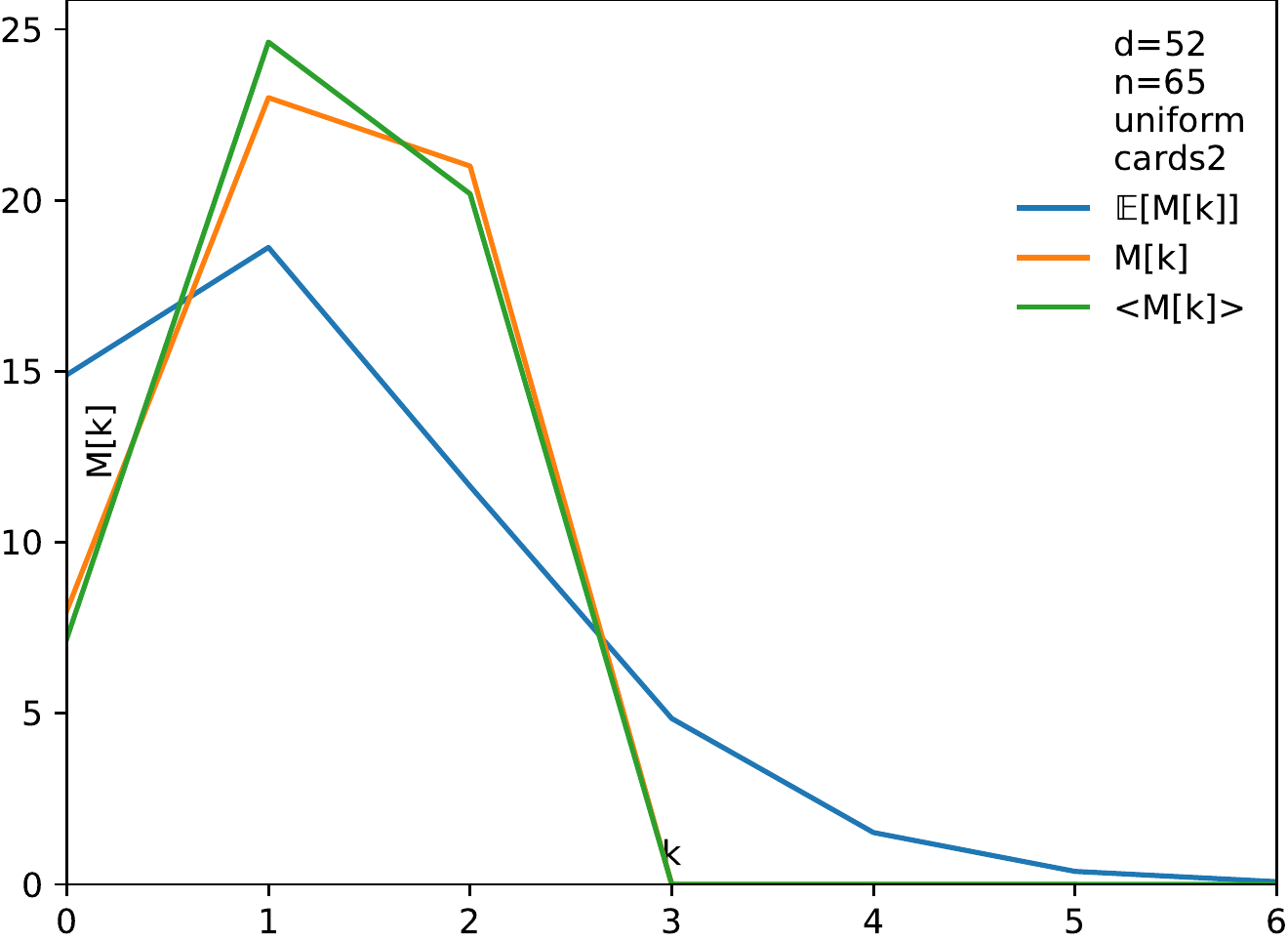}~~%
\includegraphics[width=0.51\textwidth,height=\vsr\textheight]{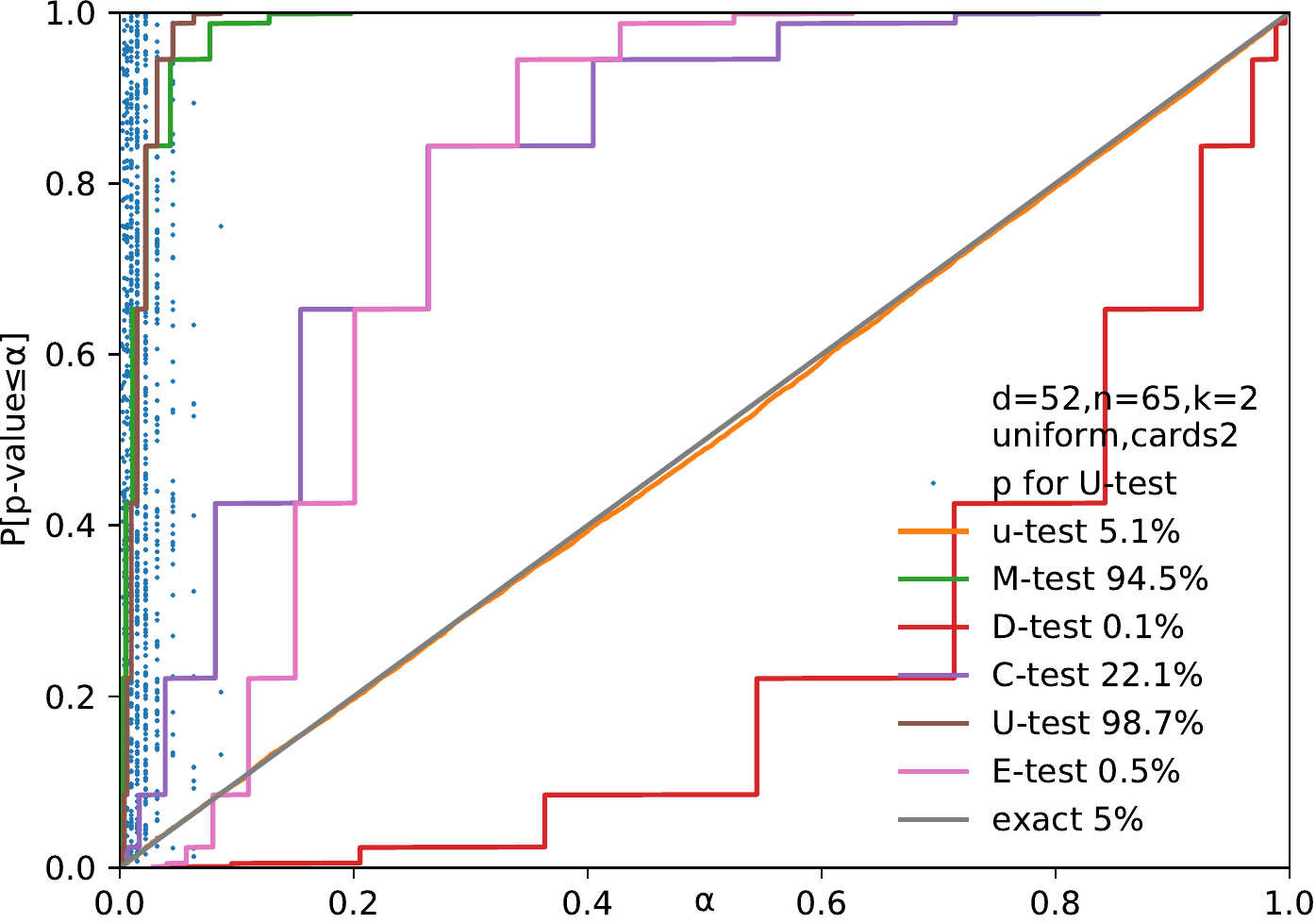}\vspace*{-2ex}
\caption{{(\bf Testing the tests)} 
\smax\ data sets $x_{1:n}$ are sampled iid from $P_\vt$ or non-iid $Q$.
\emph{(left)} One sample, average, and expected second-order counts $M_k$ as a function of $k$.
\emph{(right)} $p$-value distribution of various tests.
For iid $P_\vt$, a curve above/below the diagonal means over/under-confidence, 
i.e.\ we want on or below. For non-iid we want far above.
\% is the empirical rejection frequency at nominally $α=5\%$.
}\label{fig:MEMPV2}\vspace*{-2ex}
\end{center}
\end{figure*}

\paragraph{Experimental results per data type.}
\emph{Tests are not over-confident} ({\sf uniform-iid}):
Figure~\ref{fig:MEMPV1} top row is an example that shows that no test is over-confident (right).
For ${\sf uniform}$, all $θ$ are them same, 
hence $𝔼[\v M]$ is proportional to a Binomial($θ^*$) with maximum around $nθ^*=n/d\simequd3=k$ (left),
for which most tests are most sensitive.
We tested the tests on a variety of further combinations,
but without any surprises.
All tests are valid in the sense that they reject $\Hiid$ on iid data 
at level $α$ with probability less then $α$, or at least not much more,
so we refrained from showing the iid control plots for most of the non-iid experiments.
The occasional slight over-confidence could 
either be due to finite sample size $n$ and our asymptotic approximations 
(multinomial to Gaussian approximation) of our tests,
or variance from finite (\smax) experiment repetitions.

\emph{Tests are under-confident} ({\sf linear-iid}):
The second row shows that tests can be quite under-confident (right).
For {\sf linear}, the $θ$ values are spread out, 
hence $𝔼[\v M]$ is a broad mixture over Binomials of different $θ$ (left). 
The broader the distribution of $𝔼[\v M]$, the less confident the tests are.

\emph{Tests can be powerful} ({\sf uniform-even-n}):
The third row shows that many tests effectively detect if every data item is duplicated,
as is also obvious from the spikes in $\v M$ (left).
Unsurprisingly the most powerful one is the even-test $E$ which is tailored for this kind of data,
closely followed by the logarithmic curvature test $U_2$ with 99.6\% rejection rate.

\emph{Having no singletons can sometimes be significant} ({\sf linear-no-unique}):
Adding two copies to each observed $x$ roughly shifts $\v M$ two to the right ($M_0=M_1=0$),
so there is little non-iid signal in $M_k$ for higher $k$ (bottom left).
If most counts are very low, large $M_2$ with zero $M_1$ is suspicious.
Indeed the curvature test $M_2$ rejects $\Hiid$ 99.2\% of the time (bottom right).

\emph{Tests are often weak for larger $k$}  ({\sf uniform-even-n}):
Figure~\ref{fig:MEMPV2} top row shows a case with large (first-order) multiplicities (large $N_x$), 
i.e.\ most data items appear between 10 and 70 times.
{\sf linear} washes out any peak in $\v M$ which is indeed uniform in even $k$ from 0 to 60.
This makes all $M_k$ very small, around 1.
While some tests get a very weak signal, it is not strong enough to reject $\Hiid$.
The only effective test is {\sf even}, since combining all even $M_k$ amplifies the signal.
The second row is similar: here most $N_x$ are around 100 with $M_k$ around 4,
and besides $E$ (only) the linear curvature test is able to reject $\Hiid$ with 92.9\% confidence.

\emph{Cards from 2 decks}  ({\sf uniform-cards2}):
The bottom two rows of Figure~\ref{fig:MEMPV2} show the result of drawing $n$ cards 
from a pair of shuffled 52-card decks (104 cards in total).
Drawing half of the cards $(n=52)$ is not enough to reveal that $x_{1:n}$ is not iid (third row),
but drawing $n≥65$ suffices (last row).
240 cards from 6 decks with 312 cards in total are required to reject $\Hiid$.
$M_5$ rejects with 96\% confidence (not shown).

\paragraph{Experimental results per test.}
Comparing the various tests across the shown data examples and beyond,
they largely perform as expected. 
They are valid in the sense that the Type~I error is $\lesssim α$,
as we would expect, since they have been derived to satisfy this.
We also see that the Type~I error is often significantly smaller than $α$.
Every test has its own strengths and weaknesses,
and its power depends on the type of non-iid data they face.
All tests are sensitive to some non-iid signal,
and oblivious to (many) others.
Among the developed tests there was no uniformly best test,
but they could be combined to an approximately most powerful test as described in \Cref{sec:moretests}.
For all $k$-dependent tests (i.e.\ all but $E$ and $O$),
the more uniform $\vt$, i.e.\ the more concentrated the $μ_k$ are around some $k$,
and hence the larger $μ_k$, the more powerful the test in general.

\paragraph{Experimental results not shown.}
We confirmed that the alternative choice for the variance bounds, 
namely empirical for $M$, and theoretical for all others when available,
are worse than our (opposite) choice.
The Odd test was never effective on any of the created data,
but unsurprisingly is very effective if data is dominated by odd multiplicities.
As discussed above, our tests were not strong enough to detect 
the non-iid nature of `approximately duplicated data' ({\sf even-m}) 
and `no empty categories' ({\sf no-empty}).

\section{More/Alternative Tests}\label{sec:moretests} 

Here we present some alternative ways of deriving tests.
Either we have not worked out the details or specifics,
or we could not make them work,
or they were inferior to the ones derived in \Cref{sec:iidtests}.
They are nevertheless interesting and may be made to work with more effort.
Also disseminating failed attempts can be useful in itself.
We discuss the standard ways of combining test to broaden their power
by summation or maximization with Bonferroni correction.
Pushed to the extreme we get a universal Martin-Löf randomness test,
which in principle can detect all effective non-iid signals including  
the ``exotic'' ones from the coin flip example in \Cref{sec:app}.
We then derive a general template for invariant tests from a generalized likelihood ratio test.
This can be used to develop combinatorial tests and compression tests.
Finally we reduce the problem of testing independence of exchangeable data
to the empirical moment problem, which unfortunately is hard.

\paragraph{Summing tests.} 
Assume we have tests $T_1,T_2,...$. 
We could combine them by summing them up, $T_+=∑_k T_k$.
We could either sum their upper bounds $τ_k^{ub}$ to get an upper bound on $𝔼[T_+]$,
or derive an improved joint upper bound $τ^{ub}$  by maximizing $𝔼[T_+]$ directly.
For instance, $E=∑_{k...}k⋅M_k$ could be interpreted as a sum of tests $T_k=k⋅M_k$,
but summing up the individual bounds $𝔼[E]≤∑_{k...}k⋅μ_k^{ub}≤n⋅∑_k k/k^{3/2}=∞$ is vacuous,
while the joint bound $𝔼[T_+]≤n/2$ was non-vacuous.
Also, summing reduces the relative variance and may lead to stronger tests.
On the other hand, if some $T_k$ are highly negative,
they could ruin the sum and make the combined test weaker.
Using a weighted (e.g.\ by $1/\sqrt{\bar V_k^{ub}}$) sum may be better than a plain sum,
but makes it harder to get improved upper bounds on its expectation.

\paragraph{Bonferroni.} 
Let $c_k^α$ be such that $\sup_\vt P_\vt[T_k>c_k^α]≤α$, 
e.g.\ $c_k^α\simequdμ_k^{ub}+z_α\sqrt{μ_k^{ub}}$ for $T_k=M_k$,
and similarly for other tests, or hybrid combinations.
For finite $K⊂ℕ$, combined test $T_K:=\max\{T_k-c_k^{α/|K|}:k∈K\}$ 
with $c_K=0$ has significance $α$ ($α/|K|$ is the Bonferroni correction),
i.e.\ $P_\vt[T_K>0]≤α$.
Combining tests in this way, if some $T_k$ are highly negative, 
they do not impact the joint test $T_K$. They are just ineffective
and make $T_K$ a bit weaker, since they reduce $α/|K|$.

\paragraph{Uniformizing tests.} 
We can always uniformize tests $T$ with $\sup_\vt P_\vt[T(\v X)>c_α]≤α$
to $\sup_\vt P_\vt[\tilde T≤δ]≤δ$ as follows: 
Since $c_α:U⊆[0,1]→ℝ$ is monotone decreasing,
it has a left-continuos monotone decreasing inverse $\bar F:ℝ→[0;1]$, $\bar F(c_α)=α$,
hence $\sup_\vt P_\vt[\bar F(T(\v X))≤\bar F(c_α)]≤α$, 
hence $\tilde T(\v X):=\bar F(T(\v X))$ does as claimed.
For instance, if $T$ is standard Normal, then $c_α=z_α=Φ^{-1}(1-α)$, hence $\bar F(z)=1-Φ(z)$.
Alternatively, we can define decreasing $\bar F(t):=\sup_\vt\bar F_\vt(t)$ 
via decreasing survival function $\bar F_\vt(t):=P_\vt[T(\v X)>t]$,
then again for all $\vt$,
\begin{align*}
  P_\vt[\tilde T≤δ]  ~=~ P_\vt[\bar F(T)≤δ] ~≤~ P_\vt[\bar F_\vt(T)≤δ]
  ~≤~ P_\vt[T≥\bar F_\vt^{-1}(δ)] ~=~ \bar F_\vt(\bar F_\vt^{-1}(δ)) ~=~ δ
\end{align*}

\paragraph{Universal tests.} 
In uniform form, the Bonferroni correction is simply $\tilde T(\v X)=|K|⋅\min\{\tilde T_k:k∈K\}$.
If $\tilde T$, i.e. just one of the tests $\tilde T_k$ is small, we can reject $\Hiid$,
but we pay a price of $|K|$ in what small means. 
This generalizes to infinitely many tests, say $K=ℕ$, by defining 
\begin{align}\label{eq:unitest}
  \tilde T ~&:=~ \min\{k(k+1)\tilde T_k:k∈ℕ\}: \\
  \textstyle P_\vt[\tilde T≤δ] ~&=~ P_\vt[∃k:k(k+1)\tilde T_k≤δ] ~≤~ ∑_k P_\vt[k(k+1)\tilde T_k≤δ] ~≤~ ∑_k \frac{δ}{k(k+1)} ~=~ δ \nonumber
\end{align}
where we applied the union bound in the first inequality.
If $\tilde T_1,\tilde T_2,\tilde T_3,...$ is 
an effective enumeration of \emph{all} upper semi-computable tests \citep{Li:08},
then $U:=\tilde T$ is a so-called universal test.
No other effective test can be significantly stronger than $U$.
One can show that $U(\v X)=\sup_\vt P_\vt(\v X)/M(\v X)$, 
where $M(\v x)$ is Solomonoff's universal a-priori distribution 
\citep{Solomonoff:64,Hutter:07algprob,Hutter:07uspx,Hutter:17unilearn} 
Another interpretation is that $U$ is a Martin-Löf randomness test \citep{Li:08}
but w.r.t.\ the \emph{class} of iid distributions.
The iid randomness deficiency of $\v x$ can be defined as
$d(\v x):=\inf_\vt\{\log_2 M(\v x)-\log_2 P_\vt(\v x)\}$,
which implies $P_\vt[d(\v X)≥\log_2(1/δ)]≤δ~∀\vt$.

\paragraph{Likelihood Ratio (LR) tests.} 
Any probability distribution $Q$ on $𝓧^n$ can be converted into a uniformized iid test
\begin{align}\label{eq:logliklrt}
  \tilde T(\v X) &:=~ \sup_\vt P_\vt(\v X)/Q(\v X): \\
  P_\vt[\tilde T≤δ] ~&≤~ \textstyle P_\vt[\frac{P_\vt(\v X)}{Q(\v X)}≤δ] ~=~ P_\vt[δ\frac{Q(\v X)}{P_\vt(\v X)}≥1] 
  ~≤~ 𝔼_\vt[δ\frac{Q(\v X)}{P_\vt(\v X)}] ~=~ δ∑_{\v x}P_\vt(\v x)\frac{Q(\v x)}{P_\vt(\v x)} ~=~ δ \nonumber
\end{align}
We can even choose $Q$ to depend on $\vt$
or allow semi-probabilities $∑_{\v x∈𝓧^n}Q(\v x)≤1$.
For $Q(x):=M(x)$ we recover the universal test $U$ above.

\paragraph{Invariant LR tests for finite $𝓧$.} 
We want invariant tests, so $Q$ should only depend on $\v m$.
Let $Q(\v m):=Q[\{\v x:\v m(\v x)=\v m\}]$,
where with slight overload of notation $\v m(\v x)$ denotes the second-order counts of $\v x$.
If $𝓧$ is finite, we can symmetrize any $Q$ via
\begin{align*}
  \bar Q_d(\v x) ~&:=~ \frac{Q(\v m)}{\#\{\v x:\v m(\v x)=\v m\}}, ~~~\text{where}~~~
  \#\{\v x:\v m(\v x) =\v m\} ~=~ \Big({d\atop m_+}\Big)⋅\Big({m_+\atop \v m}\Big)⋅\Big({n\atop\v n}\Big)
\end{align*}
Note that $({n\atop\v n}):=n!/∏_{x=1}^d n_x! = n!/∏_{k=1}^n k!^{m_k}$ also depends on $\v m$ only. The same is true for 
\begin{align}
  \sup_\vt P_\vt(\v x) ~&=~ \prod_{x=1}^d \left(\frac{n_x}{n}\right)^{n_x} ~=~ \prod_{k=1}^n \left(\frac{k}{n}\right)^{k⋅m_k} \nonumber\\
  \text{Thus}~~~ \tilde T_d(\v x) ~&:=~ \frac{\sup_\vt P_\vt(\v x)}{\bar Q_d(\v x)}  \label{eq:ilrtest}
  ~=~ \frac{n!}{Q(\v m)}⋅\left({d\atop m_+}\right)⋅\left({m_+\atop \v m}\right)⋅
  ∏_{k=1}^n \left(\frac{k^k}{n^k k!}\right)^{m_k} 
\end{align}
depends on $\v m$ only. By construction, $\tilde T_d$ is a valid uniformized test \eqref{eq:logliklrt}. 
We are primarily interested in $d=∞$, 
but unfortunately the test becomes vacuous for $d→∞$.
By the argument in \Cref{sec:XtoN},
tests to leading order in $n$ remain valid if we replace infinite $𝓧$
by finite $𝓧$ of size $n^3$, i.e.\ we can use $\tilde T_{n^3}$ but this is very crude.
We can avoid the dependence on $d$ by using $\vt$-dependent $Q$:

\paragraph{\boldmath Invariant LR tests for infinite $𝓧$.} 
Let $𝓧''=\{x:n_x>0\}$ and allow $Q$ to depend on $\vt$ and decompose it as 
\begin{align*}
  Q_\vt(\v x) ~&=~ Q(\v x|\v n,𝓧'',\v m)⋅Q(\v n|𝓧'',\v m)⋅Q(𝓧''|\v m)⋅Q(\v m) \\
              ~&=~ Q(\v x|\v n)~~~~~~~~~⋅Q(\v n|\v m)~~~~⋅Q(𝓧''|m_+)⋅Q(\v m) \\
              ~&=  ~~~\left({n\atop\v n}\right)^{-1}~~~⋅~~~\left({m_+\atop\v m}\right)^{-1}⋅\left(m_+!∏_{x∈𝓧''}θ_x\right)⋅Q(\v m)
\end{align*}
Note that $\v n$ implies $𝓧''$ and $\v m$, 
the other choices in the second line are motivated by invariance.
In the last line we chose uniform probabilities for the first two factors.
For the third factor we chose the true sampling probabilities $θ_x$ of the (unique) symbols in $𝓧''$.
The $m_+!$ is because order plays no role in set $𝓧''$. Note that 
\begin{align*}
  ∑_{\nq𝓧''⊂𝓧:|𝓧''|=m_+}\nq \Big(m_+!∏_{x∈𝓧''}θ_x\Big) ~≤~ ∑_{x_{1:m_+}∈𝓧^{m_+}\nq\nq} θ_{x_1}⋅...⋅θ_{x_{m_+}} ~=~ \bigg(∑_{x∈𝓧} θ_x\bigg)^{m_+} ~=~ 1
\end{align*}
hence $Q(𝓧''|m_+)$ is indeed a valid semi-probability. Combining this with $P_\vt(\v x)=∏_{x∈𝓧''}θ_x^{n_x}$ we get 
\begin{align*}
  \tilde T(\v x) ~:=~ \sup_\vt\frac{P_\vt(\v x)}{Q_\vt(\v x)} 
  ~&=~ \sup_\vt\frac{∏_{x∈𝓧''}θ_x^{n_x-1}}{m_+!⋅Q(\v m)} \left({n\atop\v n}\right)\left({m_+\atop\v m}\right) \\
  ~&=~ \frac{n!({m_+\atop \v m})}{m_+!Q(\v m)}⋅\!∏_{k=1|2}^n \left[\frac1{k!}\left(\frac{k-1}{n-m_+}\right)^{k-1}\right]^{m_k}
\end{align*}
where we used that the maximum is attained at $θ_x=(n_x-1)/(n-m_+)$ and a similar rearranging of terms as in the previous paragraph.
This expression is similar to \eqref{eq:ilrtest} but independent from $d$ as desired.

\paragraph{Combinatorial tests.} 
We have reduced the choice of $T$ to a choice of $Q$, but what have we gained?
Note that we need to ensure $∑_{\v m}Q(\v m)≤1$, 
where the sum is over all valid second-order counts $𝓜:=\{\v m≥\v 0: ∑_{k=1}^n k⋅m_k=n\}$.
Any $Q$ satisfying this constraint results in a valid invariant uniformized test $\tilde T$.
$\text{Part}(n):=|𝓜|$ is the number of partitions of $n$ into a sum of natural numbers without regard to order \citep{Abramowitz:84}.
We could choose $Q(\v m)=\text{Part}(n)^{-1}$ uniformly.
This choice is closely related to the Good-Turing estimator \citep{Hutter:18off2onx}.
Another improved choice would be Ristad's estimator $Q(\v m)=\fr1{m_+}({m_+\atop\v m})/({n\atop m_+})$ \citep{Ristad:95}. 
The former led to essentially vacuous tests, the latter to very weak tests.

\paragraph{Compression tests.} 
Similar to Solomonoff's $M$, $Q(\v m):=2^{-K(\v m|n)}$,
where $K(\v m|n)$ is the prefix Kolmogorov complexity of $\v m$ given $n$,
is a universal distribution leading to a universal and in this case invariant test 
$U(\v m)=2^{K(\v m)}\sup_\vt P_\vt(\v x)$. 
In theory this is (at least asymptotically) the strongest test possible.
In practice, since $K$ is not computable, 
it needs to be approximated by feasible codes.
The task here is to find short prefix-free codes for $\v m$ of length $\text{CL}(\v m|n)$,
and use test $T_\text{CL}(\v m):=2^{\text{CL}(\v m)}\sup_\vt P_\vt(\v x)$. 
We tried a couple of codes such as naive $\text{CL}(m_k)\simequd\log_2(m_k+1)$,
or coding the differences $\text{CL}(m_k-m_{k-1})\simequd\log_2(2|m_k-m_{k-1}|+1)$, and variations thereof.
The resulting tests were weak to vacuous.

\paragraph{Moment method.} 
Let $ϑ:=θ/(1-θ)∈ℝ_0^+$ and $ρ[ϑ≤b]:=∑_x⟦ϑ_x≤b⟧$ be a non-negative measure on $ℝ_0^+$.
Note that $ρ[ℝ_0^+]=∑_{x=1}^d⟦ϑ_x≤∞⟧=d≠1$, i.e.\ $ρ$ is \emph{not} a probability measure.
Now let $P_ν$ be a measure on $ℝ_0^+$ that has density $dP_ν(ϑ)/dρ(ϑ):=nϑ/[(1+ϑ)^n μ_1]$ w.r.t.\ $ρ$.
Then
\begin{align*}
  μ_k  ~=~ 𝔼[M_k] ~&=~ ∑_{x=1}^d \left({n\atop k}\right)θ_x^k(1-θ_x)^{n-k}
  ~=~ ∑_{x=1}^d \left({n\atop k}\right)\frac{ϑ_x^k}{(1+ϑ_x)^n} \\ 
  ~&=~ \left({n\atop k}\right) ∫_0^∞\frac{ϑ^k}{(1+ϑ)^n} dρ(ϑ) 
  ~=~ \frac{μ_1}{n}\left({n\atop k}\right) ∫_0^∞ ϑ^{k-1}dP_ν(ϑ) 
  ~=~ \frac{μ_1}{n}\left({n\atop k}\right) ν_{k-1}
\end{align*}
where $ν_{k-1}:=𝔼_ν[ϑ^{k-1}]$ is the $k-1$st moment of $P_ν$.
Note that $𝔼_ν[1]=ν_0=nμ_1/nμ_1=1$, hence $P_ν$ is indeed a probability measure.
Now given $μ_k$ for $k≥1$ is equivalent to being given $ν_{k-1}$ for $k≥1$ and $μ_1$.
We can therefore ask whether some given sequence of real numbers 
$(ν_{k-1}:=n({n\atop k})μ_k/μ_1:k≥1)$ are moments of some probability measure, say, $P_ν$.
If so, using also $μ_1$ we can (re)construct $ρ$ from it, 
which in turn defines a mixture of Binomial distributions with $ϑ$-distribution $ρ$, 
which in turn defines an iid distribution with $θ$-distribution $ρ[θ/(1-θ)≤b]$.
This iid distribution has expected second-order counts $𝔼[M_k]$ we started with.
($μ_0=∞$ is always true for $d=|𝓧|=∞$, while $μ_0<∞$ requires $d<∞$ and poses strong extra conditions on $ρ$ hence $P_ν$ hence $(ν_k)$.)
If $(ν_k)$ are \emph{not} moments of a probability distribution,
then $(μ_k)$ are \emph{not} expectations from a mixture of Binomial distributions.

In principle we can exploit the above correspondence to develop iid tests.
Unfortunately, the moment problem, inferring a probability measure from moments, is hard.
Furthermore, we do not actually have the moments $μ_k$ but only empirical estimates $M_k$ thereof.
That is, we would need to determine whether $\hat{ν}_{k-1}:=nM_k/({n\atop k})M_1$ are approximately moments.
More precisely, we need tests which tell us whether there exist 
$ν_k∈[\hat{ν_k}±O(\sqrt{\Var[\hat{ν_k}]})]$ for all $k≥1$ 
that are moments of some probability distribution.
If not, we can reject $\Hiid$.

\section{Conclusion}\label{sec:disc}

\paragraph{Summary.}
We developed various tests for the (in)dependence of exchangeable data 
without exploiting or being given any structure in the observation space $𝓧$.
We reduced the problem to $𝓧=ℕ$ which greatly simplified the analysis.
A necessary condition for any invariant test to have power is to observe duplicate items in $x_{1:n}$.
We derived a number of tests based on the observation that the second-order counts $m_k$
are ``smooth'' in $k$ if data are iid, and demonstrated their (lack of) power empirically.
We also presented some alternative ideas for developing tests.

\paragraph{Outlook.}
While we have experimentally verified that our tests have power for some non-iid distributions,
it could be interesting to identify sub-classes of exchangeable distributions
and compute the theoretical power of our tests.
Some of the alternative approaches to developing tests from \Cref{sec:moretests} could be worked out,
esp.\ universal compression-based tests and moment-based tests.
It would be interesting to apply our tests to some real data,
but our invariant/agnostic tests only have power if $x_{1:n}$ contains exact duplicates,
and even then the non-iid signal may be too weak to detect.
In ML practice, we likely need to exploit some structure in the data.
The simplest solution would be to aggregate \emph{similar} $x$ 
into the same category ($𝓧_\text{orig}→𝓧_\text{agg}$).
The tests become more powerful to the extent that this increases the number of duplicates.
To guide the aggregation we need some metric or at least topology on $𝓧$,
and to be effective not just any but ``good'' ones.
Alternatively one could develop tests directly for $𝓧_\text{orig}$ exploiting its structure.
For instance, for $𝓧=ℝ$, 
one could check whether $x_t$ and $x_{t'}$ are correlated, 
e.g.\ for zero-mean $x_t$ whether $\fr1{n^2}∑_{t≠t'}x_t x_{t'}\not\approx 0$,
or higher-order (central) moments $\fr1{n^3}∑_{t,t',t''}x_t x_{t'}x_{t''}\not\approx 0$, etc.
In any case, as discussed in \Cref{sec:app}, this is an AI-complete problem
already for unstructured $𝓧$ and even more so for structured $𝓧$,
without a general clean solution except impractical universal tests.

We have only considered invariant tests. 
We argued that this is a natural choice,
but more convincing arguments would be good.
How limiting is it to only consider invariant tests?
Are there non-invariant tests that have more power on some invariant sub-class of $𝓠$?

\paragraph{Acknowledgements.}
I thank {\em Tor Lattimore} and {\em Bryn Elezedy} for great feedback on earlier drafts.

\fillpagebreak{0.2}
\addcontentsline{toc}{section}{\refname}
\bibliographystyle{alpha}
\begin{small}
\newcommand{\etalchar}[1]{$^{#1}$}

\end{small}

\appendix\fillpagebreak{0.2}
\section{Technical Lemmas}\label{sec:Lemmas}

To derive our tests we require a couple of technical lemmas we derive in this section.
\Cref{lem:ubfe} is heavily used to derive upper bounds on the expected value of our tests.
\Cref{lem:bubtest} is a rather standard way of creating a test of significance $α∈[0;1]$ 
from upper bounds on the mean and variance of an asymptotically Gaussian test statistic. 
Lemmas~\ref{lem:evcrv}\&\ref{lem:Zkxcor} are bounds on the variance of linear combinations of negatively correlated random variables.
\Cref{lem:mdelta} is a version of the delta-method in statistics.
In this section, $P$, $𝔼$, $\Var$, $\W$ are generic, 
and properties of random variables explicitly stated.

\paragraph{A general upper bound for expectations.} 
The following Lemma is the work horse for deriving upper bounds on the expected value of our tests.

\begin{lemma}[\bfm Upper bound for expectations]\label{lem:ubfe}
  Consider $g:ℝ^+→ℝ$ and extend its definition to 
  $g:ℝ_0^+→ℝ∪\{±∞\}$ via $g(0)=\lim\sup_{λ→0}g(λ)$ and $λ_x≥0$ and $∑_x λ_x=n$.
  Then $∑_{x=1}^d g(λ_x) ~≤~ n⋅\sup_{λ>0} g(λ)/λ$. 
  Assuming the maximum is attained and unique, set $λ^*:=\arg\max_{λ>0}g(λ)/λ$.
  The l.h.s.\ is approximately maximized by setting $λ_x≈λ^*$ for $d'≈n/λ^*$ of the $x$ and the remaining $λ_x=0$.
  The maximum is exactly attained if $d≥d'∈ℕ$. 
  For Lipschitz $g$ with $g(0)=0$ and $d'≤d$, the gap in the bound is $O(1)$ (relative gap is $O(1/n)$).
  Otherwise the gap can be unbounded.
\end{lemma}
That is, the ``only'' effect of the constraint $∑_x λ_x=n$ is that we maximize 
$g(λ)/λ$ rather than $g(λ)$ and multiply the result with $n$ rather than $d$.
For our core case $d=∞$, $d'≤d$ is satisfied.
For $d<d'<∞$, the bound can be very loose.
We will comment on that when it is due.

\begin{proof}
Consider $L(\vl)=∑_{x=1}^d g(λ_x)$ with $λ_x≥0$ and constraint $∑_x λ_x=n$
and $g:ℝ_0^+→ℝ∪\{±∞\}$ and $g(0)=g(0^+)=0$ (we lift the last assumption at the end of the proof). 
We can bound $L(\vl)$ as follows:
First, all $x$ for which $λ_x=0$ neither contribute to $L$, nor to the constraint,
so we can simply ignore such $x$, and by symmetry replace $d$ by $d':=\#\{x:λ_x>0\}≤d$,
and henceforth assume $λ_x>0$. 
\begin{align*}
  L(\vl) ~=~ ∑_{x=1}^{d'} g(λ_x) ~=~ ∑_{x=1}^{d'} \frac{g(λ_x)}{λ_x}λ_x
         ~≤~ ∑_{x=1}^{d'} \sup_{λ>0}\frac{g(λ)}{λ}λ_x
         ~=~ n⋅\sup_λ \frac{g(λ)}{λ}.
\end{align*}
If the global maximum is attained, choose any maximizer $λ^*=\arg\max_λ g(λ)/λ$. 
If $d'≤d$ and $n/λ^*$ is an integer, 
the upper bound is exact, i.e.\ $\max_\vl L(\vl)=n⋅g(λ^*)$ for $d'=n/λ^*$.
If $n/λ^*$ is not an integer, we can set $d'=\lfloor n/λ^*\rfloor$ 
or $d'=\lceil n/λ^*\rceil$ assuming it does not exceed $d$.
Since we still need to respect $∑_{x=1}^{d'}λ_x=n$,
we either have to adjust all $λ^*$ by at most $±1/n$,
or adjust one $λ_x$ by at most $±1$.
In either case, for Lipschitz $g$, 
the relative slack will be $O(1/n)$ (which we generally ignore).

We now lift the assumption that $g(0)$ is defined and $0$.
Let $c:=\lim\sup_{λ→0}g(λ)∈ℝ∪\{±∞\}$.
If $c=0$, we can simply extend $g$ to $g(0)=0$ and the above proof applies.
If $c>0$, then $\sup_λf(λ)/λ=∞$, and the bound vacuously holds.
If $c<0$, $\lim\sup_{λ→0} ∑_{x=d'+1}^d g(λ)=(d-d')c<0$, 
and hence can be replaced by $0$ for an upper bound.
\qed\end{proof}

\paragraph{On sums and averages and functions of random variables.}
Let $Y$ and $Z$ be generic real-valued random variables.
For a random variable, the corresponding lower greek denotes its expectation, 
e.g.\ $ζ=𝔼[Z]$. Variance is denoted by $σ^2=\Var[Z]$.
and third absolute central moment by $ρ=\W[Z]:=𝔼[|Z-ζ|^3]$.

Let $Z_x$ with $x∈𝓧$ be a collection of independent but typically not identically distributed random variables.
We define $Z_+:=∑_{x∈𝓧} Z_x$, sometimes over other, even uncountable, 
domains as long as only a countable number of $Z_x$ are non-zero.
Furthermore $\bar Z:=Z_+/n$, i.e.\ \emph{the bar always denotes division by $n$}, 
and is \emph{not} the average of all $Z_x$ over $x$.
Mostly $|𝓧|=∞$ anyway, but $Z_x$ and $Z_+$ and $\bar Z$ implicitly depend on the sample size $n$.
We use the $_x$, $+$, and $\bar{}$ convention also for other variables like expectations.
We sometimes drop the $+$ if it does not cause any confusion.
The reason for this convention is that typically $Z_+$ and others scale linearly with $n→∞$,
and $\bar Z$ and others are $Θ(1)$ or even converge to a finite non-zero value for $n→∞$.

\begin{definition}[Sums and ``averages'' of independent random variables]\label{def:sums} 
Let $Z_x$ with $x∈𝓧$ be a collection of independent random variables with sum $Z_+:=∑_x Z_x$ and $\bar Z:=Z_+/n$. 
Provided all involved sums and integrals are absolutely convergent,
\begin{align*}
    & ζ_x:=𝔼[Z_x],                    && σ_x^2:=\Var[Z_x],                       && ρ_x:=\W[Z_x]                    & =Θ(1) \\
    & ζ_+:=\textstyle ∑_x ζ_x=𝔼[Z_+], && \textstyle σ_+^2:=∑_x σ_x^2=\Var[Z_+],  && \textstyle ρ_+:=∑_x ρ_x & =Θ(n) \\
    & \bar{ζ}~:=ζ_+/n=𝔼[\bar Z],      && \bar{σ}^2=σ_+^2/n:=n⋅\Var[\bar Z],      && \bar{ρ}~~=ρ_+/n  & =Θ(1)
\end{align*}
\end{definition}
The last $Θ()$-column is indicative only.
Note that $\bar{σ}^2$ is \emph{not} the variance of $\bar Z$, and $σ_+:=\sqrt{σ_+^2}$ (\emph{not} $∑_x σ_x$).

We make repeated use of the Esseen version of the Berry-Esseen theorem \citep{Esseen:56} with improved constant \citep{Beek:72},
which is a strengthening of the Central Limit Theorem.
\begin{definition}[Standard Normal distribution]\label{def:Gauss}
  Let $Φ(y)=∫_{-∞}^y e^{-x^2/2}dx/\sqrt{2π}$ be the Cumulative Distribution Function (CDF) of the standard Normal.
  Define $z_α:=Φ^{-1}(1-α)$ for significance $α$ (typically $α=0.05$ and $z_{0.05}\dot=1.64$).
  The $p$-value $p:=Φ(-y)≤e^{-y^2/2}/y\sqrt{2π}$, which is sharp (in ratio) for $y→∞$.
\end{definition}

\begin{theorem}[\bfm \boldmath Berry-Esseen: $Z_+\simequd\text{Gauss}(ζ_+,σ_+^2)$ and $\bar Z\simequd\text{Gauss}(\bar{ζ},\bar{σ}^2/n)$]\label{thm:esseen}
With the notation above, let $F$ be the distribution function of 
$Y:=(Z_+-ζ_+)/σ_+=\sqrt{n}(\bar Z-\bar{ζ})/\bar{σ}$,
then $\sup_y|F(y)-Φ(y)|≤ρ_+/σ_+^3=\bar{ρ}/\sqrt{n}\bar{σ}^3$,
i.e.\ if $\bar{ρ}/\sqrt{n}\bar{σ}^3→0$, then $Y$ converges in distribution to a standard Normal.
\end{theorem}
The following lemma is a version of the multivariate delta-method in statistics \citep{Wasserman:10}.
The statements follows directly from a second-order Taylor series expansion of $g(\bar Z)$ around $\bar{ζ}$ and a multivariate version of the theorem above \citep{Raic:19}.
The result implies that Gaussian confidence intervals or $p$-values for $g(\bar Z)$ are asymptotically correct.

\begin{lemma}[\bfm \boldmath Multivariate delta method: $g(\v{\bar Z})\simequd\text{Gauss}(g(\v{\bar{ζ}}),\Var{[{\v{\bar Z}}^\trp ∇g(\v{\bar{ζ}})]})$]\label{lem:mdelta} 
With the notation above, except that $\v{\bar Z}=\fr1n∑_x\v{Z}_x∈ℝ^d$ can be vector-valued,
with $\v{\bar{ζ}}=𝔼[\v{\bar Z}]$ and 
$g:ℝ^d→ℝ$ a twice differentiable function with $∇g(\v{\bar{ζ}})≠0$ and
$σ_{\!g}^2:=∇g(\v{\bar{ζ}})^\trp\!\Var[\v{\bar Z}] ∇g(\v{\bar{ζ}})=\Var[\v{\bar Z}^\trp ∇g(\v{\bar{ζ}})]$
and $F$ the distribution function of $Y:=\sqrt{n}(g(\v{\bar Z})-g(\v{\bar{ζ}}))/σ_{\!g}$.
Then $\sup_y|F(y)-Φ(y)|=O(\bar{ρ}/\sqrt{n}\bar{σ}^3)$, 
where $\bar{σ}^2:=n⋅𝔼[||\v{\bar Z}-\v{\bar{ζ}}||_2^2]$ and $\bar{ρ}=\fr1n∑_x 𝔼[||\v Z_x-\v{ζ}_x||_2^3]$.
\end{lemma}

\paragraph{Upper-bound test.} 
We will develop various independence tests for $\v X=(X_1,...,X_n)$ based on applications of the basic test below.
It is a rather standard way of creating a test of significance $α∈[0;1]$ 
from upper bounds on the mean and variance of an approximately Gaussian test statistic,
except possibly that it is in terms of infinite sums of random variables with finite total mean and variance. 

\begin{lemma}[\bfm Basic upper-bound test for independent random variables]\label{lem:bubtest} 
Consider the hypothesis $H_\text{ind}$ that $Z_x$ for $x∈𝓧$ are independent
random variables with known upper bounds $𝔼[Z_+]≤ζ_+^{ub}$ and $σ_+^2:=\Var[Z_+]≤𝔼[V^{ub}]$, 
where $Z_+:=∑_x Z_x$, and $\Var[V^{ub}]≤O(ρ_+^2/σ_+^2)$ for some $V^{ub}≡n\bar V^{ub}$, and $ρ_+:=\W[Z_+]$.
\begin{enumerate}\parskip=0ex\parsep=0ex\itemsep=0ex
\item[(i)] At significance level $α$,
if $Z_+>ζ_+^{ub}+z_α\sqrt{V^{ub}}$ 
we can reject $H_\text{ind}$ with confidence $\simequd1-α$ (typically $α=0.05$ and $z_{0.05}\dot=1.64$), where $z_α:=Φ^{-1}(1-α)$.
\item[(ii)] The $p$-value is upper bounded by $p\lesssim\tilde T:= Φ((ζ_+^{ub}-Z_+)/\sqrt{V^{ub}})$, provided $Z_+>ζ_+^{ub}$.
That is, we can reject $H_\text{ind}$ if $p≤α$.
\item[(iii)] The above test and $p$-value are approximate with accuracy $O_P(ρ_+/σ_+^3)$ ($O_P((b-a)/σ_+)$ if $Z_x∈[a;b]$),
conservative for $σ_+^3/ρ_+→∞$, and asymptotically exact only if also $𝔼[Z_+]→ζ_+^{ub}$ and $𝔼[V^{ub}]→σ_+^2$. 
\item[(iv)] If $n\bar Z:=Z_+>ζ_+^{ub}=:n\bar{ζ}^{ub}$, then $p\lesssim Φ(\sqrt{n}(\bar{ζ}^{ub}-\bar Z)/\sqrt{\bar V^{ub}})≤\exp(-\fr12 n(\bar Z-\bar{ζ}^{ub})^2/\bar V^{ub})$ 
for sufficiently large $n$ (and $\bar Z$ and $\bar{ζ}^{ub}$ and $\bar V^{ub}$ fixed or $Θ(1)$).
\item[(v)] If $Z_x≤b$ and $\Var[V^{ub}]=0$, then $p≤\exp\{-\fr12 n(\bar Z-\bar{ζ}^{ub})^2/[\bar V^{ub}+b(\bar Z-\bar{ζ}^{ub})/3n]\}$ provided $\bar Z-\bar{ζ}^{ub}$.
\end{enumerate}
\end{lemma}
The bound in terms of $Φ$ is sharper than the exponential bounds in $(iv)$ and $(v)$. 
$Φ$ is used in our experiments;
$(iv)$ is more convenient for analytical/asymptotic analysis in our toy examples;
$(v)$ theoretically pleasing since it is non-asymptotic, 
i.e.\ exact if using theoretical (non-stochastic, non-empirical) upper bounds on $\Var[Z_+]$.

\begin{proof}
Let $ζ_x:=𝔼[Z_x]$ and $σ_x^2:=\Var[Z_x]$.
Let $Z_+:=∑_x Z_x$ be their sum with $𝔼[Z_+]=ζ_+$ and $\Var[Z_+]=σ_+^2$. 
Let $Y:=(Z_+-ζ_+)/σ_+$ with CDF $F$, where $σ_+:=\sqrt{σ_+^2}$.
\\
{\bf(i)}
$P[Z_+>ζ_+^{ub}+z_α\sqrt{V^{ub}}] \lesssim P[Z_+>ζ_+ +z_ασ_+] = P[Y>z_α] = 1-F(z_α) \simequd 1-Φ(z_α) = α.$ \\
The first $\lesssim$ follows from $ζ_+^{ub}≥ζ_+$ and is an exact inequality if $V^{ub}≥\Var[Z_+]$.
For random $V^{ub}$, we have 
\begin{align*}
  \textstyle V^{ub} ~=~ 𝔼[V^{ub}]-O_P(\sqrt{\Var[V^{ub}]}) ~≥~ σ_+^2-O_P(ρ_+/σ_+) ~=~ σ_+^2(1-O_P(ρ_+/σ_+^3))
\end{align*}
i.e.\ $\lesssim$ holds to relative accuracy $O(ρ_+/σ_+^3)$.
The approximate equality $\simequd$ holds, since $Y$ is approximately Normal:
$\sup_y |F(y)-Φ(y)|=O(ρ_+/σ_+^3)$ by the (Berry-)Esseen theorem.
\\
{\bf(ii\&iv)} For $Z_+>ζ_+$ we have \vspace{-3ex}
\begin{align*} 
  p ~&=~ 1-F(Y) ~\simequd~ Φ(-Y) ~=~ Φ\left(\frac{ζ_+-Z_+}{σ_+}\right) ~\lesssim~ Φ\left(\frac{ζ_+^{ub}-Z_+}{\sqrt{V^{ub}}}\right) \\
  ~&=~ Φ\left(\sqrt{n}\frac{\bar{ζ}^{ub}-\bar Z}{\sqrt{\bar V^{ub}}}\right) ~≤~ \exp\left(-n\frac{(\bar Z-\bar{ζ}^{ub})^2}{2\bar V^{ub}}\right) 
\end{align*}
\\
{\bf(iii)} The general accuracy has already been shown in (i).
For bounded $Z_x$ we have 
\begin{align*}
  \textstyle ρ_+ ~=~ ∑_x 𝔼[|Z_x-ζ_x|^3] ~≤~ (b-a)∑_x 𝔼[(Z_x-ζ_x)^2] ~=~ (b-a)∑_x σ_x^2 ~=~ (b-a)σ_+^2
\end{align*}
If the upper bounds are exact ($ζ_+=ζ_+^{ub}$ and $σ_+^2=𝔼[V^{ub}]$) and the relative variance of $V^{ub}$ tends to zero ($V^{ub}/σ_+^2→1$),
all approximations in (i) and (ii) become (asymptotically) exact.
\\
{\bf(v)} 
Since $Z_x≤ζ_x+b$, Bernstein's (one-sided) inequality applied to $Z_x-ζ_x$ gives
\begin{align*}
  P[Z_+-ζ_+≥t] ~=~ P\Big[∑_x(Z_x-ζ_x)≥t\Big] ~≤~ \exp\left(-\frac{t^2/2}{∑_x σ_x^2+bt/3}\right)
\end{align*}
For bounding the $p$-value, we replace $t$ by the observed $Z_+-ζ_+$:
\begin{align*}
  p ~≤~ \exp\left(-\frac{(Z_+-ζ_+)^2/2}{σ_+^2+b(Z_+-ζ_+)/3}\right)
  ~≤~ \exp\left(-\frac{(Z_+-ζ_+^{ub})^2/2}{V^{ub}+b(Z_+-ζ_+^{ub})/3}\right)
\end{align*}
where the last inequality is true since the expression 
can be shown to be monotone increasing in $ζ_+$ and $σ_+^2$ provided $Z_+≥ζ_+^{ub}$.
\qed\end{proof}

\paragraph{Negatively correlated random variables.}
We need bounds on the variance of linear combinations of negatively correlated random variables.
For positive combinations, the non-diagonal covariance terms can simply be dropped,
but we need bounds for mixed signed combinations. 
Luckily the correlations are sufficiently weak to get weaker but still useable upper bounds.

\begin{lemma}[\bfm Expectation and variance of correlated random variables.]\label{lem:evcrv} 
Let $Z_k$ be correlated random variables. 
Provided all involved sums and integrals are absolutely convergent, \\
$𝔼[(∑_k α_k Z_k)^2]~=~∑_k α_k^2 𝔼[Z_k^2]$ ~~~~~if~~~ $𝔼[Z_k Z_{k'}]=0$ $∀k≠k'$. \\
$𝔼[(∑_k α_k Z_k)^3]~=~∑_k α_k^3 𝔼[Z_k^3]$ ~~~~~if~~~ $𝔼[Z_k Z_{k'} Z_{k''}]=0$ $∀k≠k'≠k''≠k$. \\
$\Var[∑_k α_k Z_k]~~~~≤~∑_k α_k^2\Var[Z_k]$ ~~~~~if~~~ $Cov[Z_k,Z_{k'}]≤0$ and $α_k≥0$ $∀k≠k'$.
\end{lemma}
\begin{proof}\vspace{-5ex}
\begin{align*}
  \textstyle 𝔼[(∑_k α_k Z_k)^2] ~& =~ \textstyle ∑_{k,k'}α_k α_{k'} 𝔼[Z_k Z_{k'}] ~=~ ∑_k α_k^2 𝔼[Z_k^2] \\
  \textstyle 𝔼[(∑_k α_k Z_k)^3] ~& =~ \textstyle ∑_{k,k',k''}α_k α_{k'} α_{k''} 𝔼[Z_k Z_{k'} Z_{k''}] ~=~ ∑_k α_k^3 𝔼[Z_k^3] \\
  \textstyle \Var[∑_k α_k Z_k] ~&=~ \textstyle ∑_{k,k'}α_k α_{k'}\Cov[Z_k,Z_{k'}] ~≤ ∑_k α_k^2\Var[Z_k] \\[-8ex]
\end{align*}\vspace{-3ex}
\qed\end{proof}

\begin{lemma}[\bfm Double collection of (un)correlated random variables.]\label{lem:Zkxcor} 
Let $α_k∈ℝ$ and $Z_k^x$ be a double collection of random variables and $Z_k^+:=∑_x Z_k^x$.
Assume all involved sums and integrals below are absolutely convergent. 
Then \\
{\bf(a)} If $Z_k^x∈[0;1]$ and $Z_k^x$ are
uncorrelated in $x$, i.e.\ $\Cov[Z_k^x,Z_{k'}^{x'}]=0$ for all $x≠x'$ and $k,k'$,
and additionally $𝔼[Z_k^x Z_{k'}^x]=0$ $∀k≠k'$. 
Then $\Var[∑_k α_k Z_k^+]~≤~∑_k α_k^2 𝔼[Z_k^+]$.\\
{\bf(b)} If $Z_k^x∈[0;1]$ 
and $𝔼[Z_k^x Z_{k'}^x Z_{k''}^x]=0$ $∀k≠k'≠k''≠k$. 
Then $ρ_+:=∑_x\W[∑_k α_k Z_k^x]~≤~8∑_k |α_k|^3 𝔼[Z_k^+]$.\\
{\bf(c)} If $Z_k^x∈ℝ$ and $α_k≥0~∀k$ and $\Cov[Z_k^x,Z_{k'}^{x'}]≤0$ if $k≠k'$,\\
then $\Var[∑_k α_k Z_k^+]~≤~∑_k α_k^2\Var[Z_k^+]$ 
\end{lemma}

\begin{proof}(a)\vspace{-4ex}
\begin{align*}
  ~~~~\textstyle \Var[∑_k α_k Z_k^x] ~≤~ 𝔼[(∑_k α_k Z_k^x)^2]
  ~=~ ∑_k α_k^2 𝔼[(Z_k^x)^2] ~≤~ ∑_k α_k^2 𝔼[Z_k^x]
\end{align*}
where the equality follows from \Cref{lem:evcrv}. 
This and independence of $Z_k^x$ and $Z_{k'}^{x'}$ for $x≠x'$ implies 
\begin{align*}
  \textstyle \Var[∑_k α_k Z_k^+] ~&=~ 
  \textstyle \Var[∑_k α_k ∑_x Z_k^x] ~=~ 
  ∑_x\Var[∑_k α_k Z_k^x] \\
  ~&≤~ \textstyle ∑_x ∑_k α_k^2 𝔼[Z_k^x] ~=~ ∑_k α_k^2 𝔼[Z_k^+]
\end{align*}
(b) The proof follows a similar structure as (a) but with $\Var$ replaced by $\W$ and $α_k^2$ by $|α_k|^3$ and using $\W[Z]≤8𝔼[|Z^3|]$: 
\begin{align*}
  \textstyle \fr18 ρ_x ~&:=~ \fr18\W[∑_k α_k Z_k^x] ~≤~ 𝔼[|∑_k α_k Z_k^x|^3] \\
  ~&≤~ 𝔼[(∑_k|α_k||Z_k^x|)^3] ~=~ ∑_k |α_k|^3 𝔼[|Z_k^x|^3] ~≤~ ∑_k |α_k|^3 𝔼[Z_k^x]
\end{align*}
where the equality follows from \Cref{lem:evcrv}. This implies 
\begin{align*}
  \textstyle \fr18 ρ_+ ~≡~ \fr18 ∑_x ρ_x ~≤~ ∑_x ∑_k |α_k|^3 𝔼[Z_k^x] ~=~ ∑_k |α_k|^3 𝔼[Z_k^+]
\end{align*}
(c)\vspace{-4ex}
\begin{align*}
  \textstyle \Var[∑_k α_k Z_k^+] ~&=~ \textstyle ∑_{k,k'}α_k α_{k'}∑_{x,x'}\Cov[Z_k^x,Z_{k'}^{x'}] \\
  ~&≤~ \textstyle ∑_k α_k^2∑_{x,x'}\Cov[Z_k^x,Z_k^{x'}] ~=~ ∑_k α_k^2\Cov[Z_k^+,Z_k^+] \\[-8ex]
\end{align*}
\qed\end{proof}
While the following lemma is not needed to show that any of our current tests are asymptotically Normal,
we state it here for future use. 
It shows that certain unbounded $Z_x\widehat=T_x$ also satisfy the condition in \Cref{lem:bubtest}(iii).
\begin{lemma}[\bfm Upper bounds on moments]\label{lem:ubmom}
  For $Z_k^x:=M_k^x:=⟦N_x=k⟧$, hence $Z_k^+=M_k$, and $T_x:=∑_k α_k M_k^x$
  and $T_+=∑_k α_k M_k$, $τ_x:=𝔼[T_x]$, $σ_x^2:=\Var[T_x]$, $ρ_x:=\W[T_x]$,
  following the notational convention of \Cref{def:sums}, we have \\
  $τ_+:=𝔼[T_+]≤c⋅n$ and $𝔼[\bar T]=\bar{τ}≤c$ ~~if~~ $α_k≤c⋅k$ \\
  $σ_+^2:=\Var[T_+]≤c⋅n$ and $\Var[\bar T]=\bar{σ}^2/n≤c/n$ ~~if~~ $α_k^2≤c⋅k$ \\
  $ρ_+:=∑_x \W[T_x]≤8c⋅n$ and $∑_x \W[\bar T_x]=\bar{ρ}/n^2≤8c/n^2$ ~~if~~ $|α_k|^3≤c⋅k$
\end{lemma}

\begin{proof}
  $𝔼[T_+]=𝔼[∑_k α_k M_k]≤c⋅𝔼[∑_k k M_k]=c⋅𝔼[N]=c⋅n$ \\
  \Cref{lem:evcrv}a implies $\Var[T_+]=\Var[∑_k α_k Z_k^+]≤∑_k α_k^2𝔼[Z_k^+]≤c⋅∑_k k⋅𝔼[M_k]=c⋅n$ \\
  \Cref{lem:evcrv}b implies $ρ_+=∑_x\W[∑_k α_k Z_k^x]≤8∑_k |α_k|^3 𝔼[Z_k^+]≤8c⋅∑_k k⋅𝔼[M_k]=8cn$
\qed\end{proof}
Finally, we need the standard Stirling approximation:
\begin{lemma}[\bfm Stirling approximation]\label{lem:stirling} 
$\ln n!=n\ln\fr{n}{e}+\ln\sqrt{2πn}+O(\fr1n)$ or more precisely
\begin{align}\label{eq:stirling}
  \frac{1-\dot{ε}_k}{\sqrt{2πk}} ~:=~ {k^k e^{-k}\over k!} ~~~\text{with}~~~ e^{-1/12k}≤1-\dot{ε}_k≤e^{-1/(12k+1)}
\end{align}
\end{lemma}
We make frequent use of this representation/approximation. 
Rapidly $0≤\dot{ε}_k≤1/12k→0$, but for most practical purposes simply setting $\dot{ε}_k=0$ or its lower bound even for $k=1$ should be fine.
We used the exact expression in the experiments, but the asymptotic one provides more insight.

\section{I.I.D.\ Tests for Multinomial Distribution}\label{secm:iidtests}

Here we develop tests analogs to those in \Cref{sec:iidtests}
but without the Poisson approximation, i.e.\ directly for iid $\v X$ i.e.\ multinomial $\v N$.
The derivations for the upper bounds on the expectations of the test statistics $T$ are structurally very similar.
Since this section closely mirrors \Cref{sec:iidtests}, we only point out the differences, 
and refer to \Cref{sec:iidtests} for explanation of various steps and detailed explanations and discussion.
Upper bounding the variances is significantly more complicated,
and is deferred to \Cref{secm:Lemmas}.
$P$, $𝔼$, $\Var$, $\W$ are w.r.t.\ the multinomial distribution $P_\vt$ \eqref{eq:multinom}.

For large $n$, the basic functions $g(θ)≈f^n(λ)$ and Figure~\ref{fig:fT} looks virtually unchanged,
just with $λ$ replaced by $nθ$. 
The $p$-value expressions are also unchanged,
except now in terms of $Φ$ instead of $Φ_n$, i.e.\ without fudge factor $c_n$,
and the upper bounds for $𝔼[T]$ derived here are slightly different.
\Cref{secm:Lemmas} also shows that under certain conditions, $\Var_\vt[T]≈\Var_\vl[T]$.
The running example of duplicate data items is also unchanged,
even the specific constants remain the same for $n→∞$.

\paragraph{The general idea behind the tests.}
Like the Poisson, the binomial \eqref{eq:binom} 
is also ``smooth'' in $k$ and $θ$, has a unique maximum at $k≈nθ$, 
is log-concave with small slope and curvature, 
so is also a rather benign function.
Indeed it is (also) approximately Gaussian with mean $nθ$ and variance $nθ(1-θ)$.
Analogous to \eqref{eq:EMk}, consider
\begin{align}\label{eqm:EMk}
  𝔼[M_k] ~=~ ∑_x P_{\vt}[N_x=k] ~=~ ∑_x P_{θ_x}(k) ~=~ ∑_x f_k^n(θ_x) ~=~ ∑_x \left({n\atop k}\right)θ_x^k(1-θ_x)^{n-k}
\end{align}
\Cref{prop:ublt} stays nearly the same.
Since we use it repeatedly we (re)state it here in terms of $P_\vl$.
The proof is the same with the obvious substitutions of $P_\vt$ instead of $P_\vl$ 
and $∑_x θ_x=1$ instead of $∑_x λ_x=n$, which explains the ``missing'' factor $n$.

\begin{proposition}[\bfm Multinomial upper bounds for linear tests]\label{propm:ublt} 
  Let $T=∑_k α_k M_k$ for $α_k∈ℝ$.
  Provided all involved sums and integrals are absolutely convergent, we have \\
  $τ:=𝔼[T]≤\sup_{θ>0} f(θ)/θ =:τ^{ub}$,
  where $f(θ):=∑_k α_k P_θ(k)=({n\atop k})θ^k(1-θ)^{n-k}$, and \\
  $\Var[T]≤∑_k α_k^2 𝔼[M_k]≤V^{ub}$, 
  where $V^{ub}:=∑_k α_k^2 μ_k^{ub}$ 
  with $μ_k^{ub}≥𝔼[M_k]$ upper bounding the expectations of $M_k$.
\end{proposition}

\paragraph{\boldmath Second-order count tests $M_k$.} 
\begin{align}\label{eqm:fknub}
  f_k^n(θ) ~&≤~ \textstyle f_k^n\big(\frac{k}{n}\big) ~≤~ \sqrt{\frac{n}{2πk(n-k)}} 
  ~=~ \frac{1}{\sqrt{2πk}}\left[1+O(\frac{k}{n})\right], ~~~\text{hence}~~~ \\ \label{eqm:Mkub}
  μ_k ~&:=~ 𝔼[M_k] ~≤~ \fr{n}{k} f_{k-1}^{n-1}(\fr{k-1}{n-1}) =: μ_k^{ub} 
  ~≤~ \textstyle \frac{n}{k}\sqrt{\frac{n-1}{2π(k-1)(n-k)}} 
  ~=~ \fr{n}{k\sqrt{2π(k-1)}}[1+O(\fr{k}{n})]
\end{align}
These expressions very similar to \eqref{eq:Mkub} just for $f_{k-1}^{n-1}(θ)$ and 
maximizer $θ^*=\fr{k-1}{n-1}$ (cf.\ Figure~\ref{fig:fT}).
In \Cref{secm:Lemmas} we show that under certain conditions,
$\Var[M_k]≤𝔼[M_k]$ remains approximately valid also for $P_θ$, 
hence 
\begin{align*}
  p ~&\lesssim~ \textstyle Φ((μ_k^{ub}-M_k)/\sqrt{μ_k^{ub}}) \tc{\\}
  ~\tc{&}≤~ \exp(-\fr12 n(\bar M_k-\bar{μ}_k^{ub})^2/\bar{μ}_k^{ub}) ~=~ e^{-O(n/k^{3/2})}
\end{align*}
now without the fudge factor $c_n$. The same is true for the other tests.

\paragraph{\boldmath Even and odd tests $E$ and $O$.} 
In \Cref{sec:iidtests} we explained why we need to exclude $M_0$ and $M_1$ from $E$ and $O$.
In addition, for $θ_x=1/d$ and $d→∞$, every $x$ is observed exactly once, hence $M_1=n$ and all other $M_k=0$,
also not leading to a useful test, so we also need to exclude $M_n$,
i.e.\ $α_k^\text{even}:=k⋅⟦0≠k≠n~\text{even}⟧$ and $α_k^\text{odd}:=k⋅⟦1≠k≠n~\text{odd}⟧$.
Let 
\begin{align*}
  f_\text{even}(θ) ~&:=~ ∑_k α_k^\text{even} P_θ(k) ~=~ ∑_{\nq 0≠k≠n~\text{even}\nq} k⋅f_k^n(θ) \tc{\\}
  ~\tc{&}=~ \fr12 n θ[1-(1-2θ)^{n-1}-2θ^{n-1}⟦n~\text{even}⟧] ~≤~ \fr12 n θ \\
  f_\text{odd}(θ) ~&:=~ ∑_{\nq 1≠k≠n~\text{odd}\nq} k⋅f_k^n(θ) \tc{\\}
  ~\tc{&\hspace*{-9ex}}=~ \fr12 n θ[1-2(1-θ)^{n-1}+(1-2θ)^{n-1}-2θ^{n-1}⟦n~\text{odd}⟧] ~≤~ \fr12 n θ
\end{align*}
The equalities follow from binomial identities. 
The expressions in the brackets have the form $[1-h(θ)]$ for $2×2$ different functions $h(θ)$:
For $f_\text{even}$ and $f_\text{odd}$ and for even and odd $n$.
In one case $h(θ)≥0$ is trivial. 
In the other 3 cases this follows by finding the minimum $θ^*=\fr13|\fr12|\fr23$ resp.\ via $dh(θ)/dθ=0$ and showing $h(θ^*)≥0$.
This establishes the upper bounds.
The remaining definitions, derivations, and arguments are the same as in \Cref{sec:iidtests}.

\paragraph{\boldmath Slope tests $D_k:=M_k-M_{k-1}$.} 
For the slope test we have
\begin{align*}
  δ_k ~:=~ 𝔼[D_k] ~&=~ ∑_x P_{θ_x}[N_x=k]-P_{θ_x}[N_x=k-1] \tc{\\}
  ~\tc{&}=~ ∑_x f_δ(θ_x) ~≤~ \sup_{θ>0}\smash{\frac{f_δ(θ)}{θ}} ~=:~ n\bar{δ}_k^{ub} \\
  \ntc{\text{where}~~~ &} f_δ(θ):=f_k^n(θ)\tc{&}-f_{k-1}^n(θ) ~=~ ({\textstyle{n+1\atop k}})θ^{k-1}(1-θ)^{n-k}[θ-\fr{k}{n+1}]
\end{align*}
The last expression follows from inserting \eqref{eq:binom} and elementary algebra.
The maximum of $P_θ(k)$ is at $θ=k/n$ but the bracket $[θ-\frac{k}{n+1}]$ kills this maximum,
moving it to $\simequd k+\sqrt{k}$ (Figure~\ref{fig:fT}). 
As in \Cref{sec:iidtests} we can find it exactly by differentiating 
\begin{align*}
  \tc{&\ln(\fr1{θ}f_δ(θ)) = c + (k-2)\ln θ + (n-k)\ln(1-θ) + \ln(θ-\fr{k}{n+1}) \\}
  \ntc{\ln(f_δ(θ)/θ) ~&=~ \ln({\textstyle{n+1\atop k}}) + (k-2)\ln θ + (n-k)\ln(1-θ) + \ln(θ-\fr{k}{n+1}) \\}
  \tc{&}\textstyle \frac{d}{dθ}\ln(f_δ(θ)/θ) ~\ntc{&}=~ \textstyle \frac{k-2}{θ} - \frac{n-k}{1-θ} + \frac1{θ-k/(n+1)} \tc{\\}
  ~\tc{&}∝~ -[(n^2-1)θ^2+(k-2kn+n+1)θ+k^2-2k] ~\stackrel!=0
\end{align*}
The last expression follows from multiplication with $θ(1-θ)[(n+1)θ-k]$ and rearranging terms.
This is a quadratic equation in $θ$ with solution
\begin{align*}
  \tc{\nq} θ^* ~&=~ \frac{2kn-k-n-1+\sqrt{k^2(5-4n)+k(4n^2-2n-6)+(n+1)^2}}{2(n^2-1)} \tc{\\}
  ~\tc{&}=~ \frac{k-\fr12}{n}+\frac{\sqrt{k+\fr14}}{n}+O\bigg(\frac{k^{3/2}}{n^2}\Big)
\end{align*}
which is indeed the global maximum.
\begin{align}\label{eqm:Dkub}
  \bar{δ}_k^{ub} ~=~ \frac{f_δ(θ^*)}{nθ^*} ~=~ \frac{1-O(1/\sqrt{k})}{k^2\sqrt{2πe}}
\end{align}
The remainder is the same as in \Cref{sec:iidtests} with $Φ_n\leadsto Φ$ and $\Var_\vt[D_k]≈\Var_\vl[D_k]≤...$.

\paragraph{\boldmath Linear curvature tests $C_k:=2M_k-M_{k-1}-M_{k+1}$.} 
Let
\begin{align*}
  f_γ(θ) ~&:=~ 2P_θ(k)-P_θ(k-1)-P_θ(k+1) \tc{\\}
  ~\tc{&}=~ f_k^n(θ)\left[2-\frac{k}{n-k+1}\frac{1-θ}{θ}-\frac{n-k}{k+1}\frac{θ}{1-θ}\right]
\end{align*}
be the negative curvature of $P_θ$. Unlike in \Cref{sec:iidtests} we cannot maximize $f_γ(θ)/θ$
exactly anymore, but the following upper bound is quite tight (cf.\ Figure~\ref{fig:fT})
\begin{align*}
  \sup_{θ>0}[f_γ(θ)/θ] ~&≤~ \sup_{θ>0}\{f_k^n(θ)/θ\}⋅\max_{z>0}[2-α/ϑ-βϑ] \tc{\\}
  ~\tc{&}=~ μ_k^{ub}[2-2\sqrt{αβ}]
\end{align*}
where $ϑ:=\frac{θ}{1-θ}$, $α=\frac{k}{n-k+1}$, $β=\frac{n-k}{k+1}$, and $\max_z$ is attained at $ϑ^2=α/β$.
Hence 
\begin{align}
  n\bar{γ} ~&:=~ 𝔼[C_k] ~=~ ∑_x f_γ(θ_x) ~≤~ \max_{θ>0}\frac{f_γ(θ)}{θ} \nonumber \\
  ~&≤~ n\bar{γ}_k^{ub} ~:=~ μ_k^{ub}\Big[2-2\sqrt{\fr{k}{k+1}\fr{n-k}{n-k+1}}\Big] \label{eqm:Ckub}\tc{\\} 
  ~\tc{&\nonumber}≤~ \frac{μ_k^{ub}}{(k+\fr12)(1-\fr{k}{n+1/2})} ~≈~ \frac{n}{k^2\sqrt{2πk}} 
\end{align}
where $≈$ holds for $n\gg k\gg 1$.
The remainder is the same as in \Cref{sec:iidtests}.

\paragraph{\boldmath Logarithmic curvature tests $\bar U_k:=2\ln M_k-\ln M_{k-1}-\ln M_{k+1}$.} 
The derivation is the same as in \Cref{sec:iidtests} with minimal changes:
With $\tilde P_{\vt,k}[X=x] ~:=~ f_k^n(θ_x)/μ_k$ and $ϑ_x:=θ_x/(1-θ_x)$ we get
\begin{align*}
  \frac{μ_{k+1}}{μ_k} ~&=~ \frac1{μ_k}∑_x f_{k+1}^n(θ_x) 
    ~=~ \frac1{μ_k}∑_x \frac{n-k}{k+1}ϑ_x f_k^n(θ_x) \tc{\\}
    ~\tc{&}=~ \frac{n-k}{k+1}⋅\tilde{𝔼}_{\vt,k}[ϑ_X] \\
  \frac{μ_{k-1}}{μ_k} ~&=~ \frac1{μ_k}∑_x f_{k-1}^n(θ_x) 
    ~=~ \frac1{μ_k}∑_x \frac{k}{n-k+1}\frac1{ϑ_x}f_k^n(θ_x) \tc{\\}
    ~\tc{&}=~ \frac{k}{n-k+1}⋅\tilde{𝔼}_{\vt,k}\bigg[\frac1{ϑ_X}\bigg] 
    ~≥~ \frac{\frac{k}{n-k+1}}{\tilde{𝔼}_{\vt,k}[ϑ_X]}
\end{align*}
where we applied Jensen's inequality in the last step to convex function $1/ϑ$.
Taking the product, the dependence on unknown $\vt$ cancels out:
\begin{align}\label{eqm:Ukub}
  \tc{\bar{υ}_k := \ln\frac{μ_k}{μ_{k-1}}\frac{μ_k}{μ_{k+1}} ≤ \ln\frac{n-k+1}{n-k}\frac{k+1}{k} =: \bar{υ}_k^{ub} ≤ \frac1k + \frac1{n-k}}
  \ntc{\bar{υ}_k ~:=~ \ln\frac{μ_k}{μ_{k-1}}\frac{μ_k}{μ_{k+1}} ~≤~ \ln\frac{n-k+1}{n-k}\frac{k+1}{k} ~=:~ \bar{υ}_k^{ub} ~≤~ \frac1k + \frac1{n-k}}
\end{align}
The remainder is the same as in \Cref{sec:iidtests}.

\paragraph{Summary.}
In the table below we summarize the most important quantities for the tests $T:𝓧^n→ℝ$ derived in this section
for comparison to the ones derived in \Cref{sec:iidtests} based on the Poisson approximation.
For $τ:=𝔼[T]$ we derived tight upper bounds for all, even small $k$.
The $n\gg k\gg 1$ approximations in the table are the same but the referred to exact expressions differ.
\begin{center}
\setlength\tabcolsep{2pt}
\begin{tabular}{r|l|l|c|c|c}
  Test Name & $T:=n\bar T:=$ & $\bar{τ}:=𝔼[\bar T]≤$ & $\Var[\bar T]\lesssim\bar V^{ub}=$ & $θ^*$ & $O(\ln\fr1p$) \\ \hline
  Even$≠0$  & $E:=∑_x N_x⟦N_x≠0~\text{even}⟧$ & $\bar{ε}^{ub}=1/2$ & $\fr1n ∑_{k≠0~\text{even}}k^2 M_k$ & $\frs13~|~\frs12$ & $n$ \\
  Odd$≠1$   & $O:=∑_x N_x⟦N_x≠1~\text{odd}⟧$  & $\bar{ο}^{ub}=1/2$ & $\fr1n ∑_{k≠1~\text{odd}}k^2 M_k$ & $\frs23~|~\frs12$ & $n$ \\
  2nd-Count & $M_k:=∑_x⟦N_x=k⟧$             & $\bar{μ}_k^{ub}\smash{\stackrel{\eqref{eqm:Mkub}}{≈}}\frac1{k\sqrt{2πk}}$ & $\bar{μ}_k^{ub}$ & $\fr{k-1}{n-1}$ & $\frac{n}{k^{3/2}}$ \\
  Slope     & $D_k:=M_k-M_{k-1}$                  & $\bar{δ}_k^{ub}\smash{\stackrel{\eqref{eqm:Dkub}}{≈}}\frac1{k^2\sqrt{2πe}}$ & $\bar M_k+\bar M_{k-1}$ & $≈\fr1n[{k+\sqrt{k}}]$ & $\frac{n}{k^{5/2}}$ \\ 
  Lin.Curv. & $C_k:=2M_k\!-\!M_{k-1}\!-\!M_{k+1}$         & $\bar{γ}_k^{ub}\smash{\stackrel{\eqref{eqm:Ckub}}{≈}}\frac1{k^2\sqrt{2πk}}$ & $4\bar M_k+\bar M_{k-1}+\bar M_{k+1}$ & $≈k/n$ & $\frac{n}{k^{7/2}}$ \\
  Log.Curv. & $\bar U_k:=\ln(M_k^2/M_{k-1}M_{k+1})$    & $\bar{υ}_k^{ub}\smash{\stackrel{\eqref{eqm:Ukub}}{≈}}\frac1k$ & $\bar M_{k-1}^{-1}\!+\!4\bar M_k^{-1}\!+\!\bar M_{k+1}^{-1}$ & any & $\frac{n}{k^{7/2}}$
\end{tabular}
\end{center}

\begin{claim}[\bfm IID tests]\label{thmm:iidtests}
Consider the test statistics $\bar T$ and associated upper bounds 
on their mean $\bar{τ}$ and variance $\Var[\bar T]$ from the above table.
Then test $\bar T(x_{1:n})≥\sqrt{n}[\bar{τ}^{ub}+z_α\sqrt{\bar V^{ub}}]$
rejects that $x_{1:n}$ is iid with confidence $\gtrsim 1-α$, i.e.\ at significance level $\lesssim α$,
where $z_α:=Φ^{-1}(1-α)$ (typically $α=0.05$ and $z_{0.05}\dot=1.64$).
The conditions under which $\lesssim$ is reasonably accurate are discussed in \Cref{secm:Lemmas}.
See \Cref{lem:bubtest} (with $Z_+=T$ and $Z_x=T_x$) for $p$-values and further details.
\end{claim}

\section{Technical Lemmas for Multinomial vs Poisson}\label{secm:Lemmas} 

In this section we introduce the multinomial distribution $P_\vt$ and Poisson process $P_\vl$ more carefully
with the aim to find useful relations between their means and variances for our tests.
We confirm that $𝔼_\vl[T]$ derived in \Cref{sec:iidtests}
is close to $𝔼_\vt[T]$ derived directly in \Cref{secm:iidtests}, 
not just for or specific tests but more generally.
For the variance, we have only derived expressions for $\Var_\vl$ (\Cref{sec:iidtests,sec:Lemmas}).
The results in this Section show that under certain conditions $\Var_\vt\lesssim\Var_\vl$,
which in turn allows to avoid the fudge factor $c_n\simequd\sqrt{2πn}$ used in \Cref{sec:iidtests}.
The precise conditions under which this is possible have yet to be worked out.
Since we are comparing the multinomial with the Poisson distribution,
$P$, $𝔼$, $\Var$ are appropriately indexed with $\vt$ or $\vl$.

\paragraph{Poisson distribution/process.} 
The Poisson($λ$) distribution $P_λ(k):=λ^k e^{-λ}/k!=:g_k(λ)$ for $λ≥0$ and $k∈ℕ_0$ has $𝔼_λ[k]=\Var_λ[k]=λ$. 
For finite $𝓧$, let $(Ω:=ℕ_0^𝓧,2^Ω,P_\vl)$ be the probability space of independent Poisson($λ_x$) for $x∈𝓧$. 
For random variables $\v N$ and atomic events $\{\v n\}$ with $\v n∈ℕ_0^𝓧$, 
with slight overload in notation we have
\begin{align*}
  P_\vl(\v n) ~:=~ P_\vl[\v N=\v n] ~:=~ ∏_{x∈𝓧} P_{λ_x}(n_x) ~=~ ∏_{x∈𝓧} \frac{λ_x^{n_x}e^{-λ_x}}{n_x!}
\end{align*}
where $λ_x≥0$. We will assume $∑_x λ_x=n$.
For infinite $𝓧$, one defines $P_\vl$ on all finite partitions of $𝓧$,
with $λ$ of a partition being the sum (or measure in general) of $λ_x$ over the $x$ in the partition.
These probabilities so defined on the ``cylinder'' set are indeed ``self-consistent'' 
in the sense that they can uniquely be extended 
in a standard way to a measure on $ℕ_0^𝓧$ with $σ$-algebra 
generated by the cylinders (a Poisson process). The details are of no concern to us.

What is important is that partitions of $𝓧$ are also products of Poissons.
In particular, for the cylinders,
\begin{align*}
  E_k^x ~&:=~ \{\v n∈ℕ_0^𝓧:n_x=k\} \\
  E_{\bar k}^{\bar x} ~&:=~ \{\v n∈ℕ_0^𝓧:n_{\bar x}={\bar k}\},
  ~~~\text{where}~~~ n_{\bar x}:=\textstyle ∑_{x''≠x} n_{x''} \\ 
  E_{\overline{kk'}}^{\overline{xx'}} ~&:=~ \{\v n∈ℕ_0^𝓧:n_{\overline{xx'}}=\overline{kk'}\},
  ~~~\text{where}~~~ n_{\overline{xx'}}:=\textstyle{∑_{x''∈𝓧\setminus\{x,x'\}}} n_{x''} \\
  P_\vl[E_k^x] ~&=~ P_\vl[N_x=k] ~=~ P_{λ_x}(k) ~≡~ g_k(λ_x) ~≡~ λ_x^k e^{-λ_x}/k! \\
  P_\vl[E_k^x∩E_{k'}^{x'} & ∩E_{\overline{kk'}}^{\overline{xx'}}] ~=~ 
  P_\vl[E_k^x]⋅P_\vl[E_{k'}^{x'}]⋅P_\vl[E_{\overline{kk'}}^{\overline{xx'}}] ~=~
  P_{λ_x}(k)⋅P_{λ_{x'}}(k')⋅P_{λ_{\overline{xx'}}}({\overline{kk'}}) \\
  ~~~\text{where} &~~~ λ_{\overline{xx'}} ~:=~ \textstyle{∑_{x''∈𝓧\setminus\{x,x'\}}} λ_{x''} ~=~ n-λ_x-λ_{x'}
\end{align*}
and similar for other combination of events.
For the total sample size $N:=N_+=∑_{x∈𝓧}N_x$, 
and associated event $E_n^𝓧:=\{\v n:∑_x n_x=n\}$ we have
\begin{align*}
  P_\vl[E_n^𝓧] ~≡~ P_\vl[N=n] ~=~ P_n(n) ~=~ n^n e^{-n}/n! ~\simequd~ 1/\sqrt{2πn} ~=:~ 1/c_n
\end{align*}
Independence also implies the elementary identities
\begin{align*}
  𝔼_\vl[N_x] ~&=~ \Var_\vl[N_x] ~=~ λ_x ~~~\text{and}~~~
  \Cov_\vl[N_x,N_{x'}] ~=~ 0~~~\text{for}~~~x≠x',~~~\text{hence}~~~ \\
  𝔼_\vl[N] ~&=~ ∑_x 𝔼_\vl[N_x]~=~∑_x λ_x~=~n ~~~\text{and}~~~ \Var_\vl[N]~=~∑_x\Var_\vl[N_x] ~=~∑_x λ_x~=~n
\end{align*}
This means that $N=n±O_P(\sqrt{n})$ is close to $n$ for large $n$.

\paragraph{Multinomial distribution.} 
The multinomial distribution respects similar identities as the product of Poissons, except independence.
Under certain conditions it is close to Poisson and close to independent.
Traditionally, the multinomial is defined on sample space $Ω_n:=\{\v n:∑_x n_x=n\}$.
For comparison to the Poisson this is inconvenient.
We enlarge the support from $Ω_n$ to $Ω=ℕ_0^𝓧$ and define the multinomial zero on $Ω\setminus Ω_n$.
For $𝓧=\{1:d\}$,
\begin{align*}
  P_\vt(\v n) ~:=~ P_\vt[\v N=\v n] ~:=~ \Big({n\atop n_1,...,n_d}\Big)\smash{ ∏_{x=1}^d θ_x^{n_x}}
  ~~~\text{for}~~~ ∑_x n_x=n ~~~\text{and 0 else}
\end{align*}
where $θ_x≥0$ and $∑_x θ_x=1$. 
For countable $𝓧$ the above formula still applies, since only finitely many $n_x$ can be non-zero,
and $n_x=0$ gives no contribution. 
For uncountable $𝓧$ it can be extended in the same way as $P_\vl$ by partitioning $𝓧$.
A special case is the binomial distribution
\begin{align*}
  P_θ(k) ~=~ f_k^n(θ) ~:=~ ({\textstyle{n\atop k}})θ^k(1-θ)^{n-k}~~~\text{for $0≤θ≤1$ and $k∈\{0:n\}$ and 0 else}
\end{align*}
Some identities analogous to $P_\vl$ are
\begin{align*}
  P_\vt[E_k^x] ~&=~ P_\vt[N_x=k] ~=~ P_{θ_x}(k) ~≡~ f_k^n(θ_x) \\
  P_\vt[E_k^x∩E_{k'}^{x'}] ~&=~ P_\vt[N_x=k∧N_{x'}=k'] ~=:~ f_{kk'}(θ_x,θ_{x'}) ~=~ ({\textstyle{n\atop k~k'}})θ_x^kθ_{x'}^{k'}(1-θ_x-θ_{x'})^{n-k-k'}
\end{align*}
and similar for other combination of events.
By definition, $P_\vt[N=n]=P_\vt[Ω_n]=1$.
We also have the elementary identities
\begin{align*}
  𝔼_\vt[N_x] ~&=~ nθ_x,~~~\Var_\vt[N_x] ~=~ nθ_x(1-θ_x) ~~~\text{and}~~~
  \Cov_\vt[N_x,N_{x'}] ~=~ -n θ_x θ_{x'}~~~\text{for}~~~x≠x'
\end{align*}
Note that $N_x$ and $N_{x'}$ are negatively correlated, 
and $E_k^x$ and $E_k^{x'}$ are \emph{not} independent under $P_\vt$ unlike $P_\vl$.

\paragraph{Exact relations between multinomial and Poisson distribution.} 
For $\vl=n\vt$, which we henceforth assume, the multinomial and product of Poissons are closely related.
A straightforward bound for any \emph{linear} combination $S:=∑_x β_x N_x$ of $N_x$ is 
\begin{align*}
  \Var_\vt[S] ~&=~ ∑_x β_x^2\Var_\vt[N_x] + ∑_{x≠x'} β_x β_{x'}\Cov_\vt[N_x,N_{x'}]
  ~=~ ∑_x β_x^2 n θ_x(1-θ_x) - ∑_{x≠x'} β_x β_{x'}n θ_x θ_{x'} \\
  ~&=~ ∑_x β_x^2 λ_x - n(∑_x β_x θ_x)^2 ~≤~ ∑_x β_x^2 \Var_\vl[N_x] ~=~ \Var_\vl[S]
\end{align*}
Unfortunately we need bounds for \emph{non}-linear functions of $N_x$ such as $M_k=∑_x⟦N_x=k⟧$,
which are much harder to come by.
First note that $P_\vl$ conditioned on $N=n$ exactly equals $P_\vt$:
\begin{align}\label{eq:Pltnn}
  & P_\vl[\v N=\v n|N=n] ~=~ \frac{P_\vl[\v N=\v n]}{P_\vl[N=n]} 
  ~=~ \frac{∏_x λ_x^{n_x}e^{-λ_x}/n_x!}{n^n e^{-n}/n!} 
  ~=~ \left({n\atop\v n}\right)∏_x\left(\frac{λ_x}{n}\right)^{n_x}\nq
  ~=~ P_\vt[\v N=\v n] \\ \nonumber
  & \text{hence}~~ \frac{P_\vl(\v n)}{P_\vt(\v n)} 
  ~=~ \frac{P_\vl(\v n|n)P_\vl(n)}{P_\vl(\v n|n)} ~=~ P_\vl[N=n] 
  ~=:~ \frac1{c_n} ~\simequd~ \frac1{\sqrt{2πn}} ~~\text{provided}~~ ∑_x n_x=n
\end{align}
This implies that for \emph{all} events $E⊆Ω_n$, 
$P_\vl[E|n]=P_\vt[E]$ but $P_\vl[E]=c_n⋅P_\vt[E]$, 
i.e.\ they differ by a factor of $c_n=O(\sqrt{n})$. 
The intuition is that the $P_\vt$ probability mass of $E$ 
is spread out in $P_\vl$ over $∑_x n_x=n±O(\sqrt{n})$.
For events $E$ of probability exponentially small in $n$,
the slack of a sub-polynomial $c_n$ is benign in theory,
but unfortunately not in practice.
Another useful exact relation is 
\begin{align}\label{eq:PtPPPl}
  P_\vt[E_k^x] ~&=~ P_\vl[E_k^x|N=n]
  ~=~ \frac{P_\vl[E_k^x∩\{N=n\}]}{P_\vl[N=n]} 
  ~=~ \frac{P_\vl[E_k^x∩E_{\bar k}^{\bar x}]}{P_\vl[N=n]} 
  ~=~ \frac{P_\vl[E_k^x]⋅P_\vl[E_{\bar k}^{\bar x}]}{P_\vl[N=n]}
\end{align}
where $\bar k=n-k$ and $λ_{\bar x}=∑_{x''≠x}λ_{x''}$.
This can also be verified by explicit calculation similar to \eqref{eq:Pltnn}.

On the other hand we will show that under certain conditions,
$P_\vt[E_k^x]≈P_\vl[E_k^x]$ even without conditioning on $N=n$.
For instance, we already know thar $𝔼_\vt[N_x]=nθ_x=λ_x=𝔼_\vl[N_x]$ and 
$\Var_\vt[N_x]=nθ_x(1-θ_x)≈nθ_x=λ_x=\Var_\vl[N_x]$ for small $θ_x$.
There is no contradiction to $P_\vl[E]=c_n⋅P_\vt[E]$, 
since $E_k^x\not\subseteq Ω_n$. Here the intuition is that 
$E_k^x$ itself is spread out over a wide range of $n'$,
and $P_\vl[E_k^x|n']$ is approximately independent of $n'$ at least over the range $n'∈[n±O(\sqrt{n})]$.
If this is satisfied, then 
\begin{align*}
  P_\vl[E_k^x] ~=~ ∑_{n'} P_\vl[E_k^x|N=n']P_\vl(n') ~≈~ P_\vl[E_k^x|N=n] ∑_{n'}P_\vl(n') ~=~ P_\vt[E_k^x]
\end{align*}

\paragraph{Approximate relations between multinomial and Poisson.} 
We now derive our fundamental relation between a single Poisson $P_λ(k)$ and binomial $P_θ(k)$ for $λ=nθ$,
which is the basis for all other approximations.
\begin{lemma}[\bfm Expansion of log(Poisson/binomial)]\label{lem:logPoBi}
For $κ:=k/n=:1-γ$ and $θ=λ/n$ fixed and $n→∞$,
\begin{align*}
  \ln\frac{P_λ(k)}{P_θ(k)} ~≡~ \ln\frac{g_k(λ)}{f_k^n(λ)}~=~ n(κ-θ)-nγ\ln\frac{1-θ}{γ}+\ln\sqrt{γ}+O(\frac{κ}{γn})
\end{align*}
For $κ$ close to $θ$ we can further approximate this by 
\begin{align*}
  \ln\frac{P_λ(k)}{P_θ(k)} ~=~ \frac{n}{2γ}(κ-θ)^2[1+O(\fr1{γ}|κ-θ|)] + \ln\sqrt{γ} + O(\frac{κ}{γn})
\end{align*}
For $κ,θ=o(n^{-1/2})$ this implies
\begin{align*}
  P_λ(k) ~&=~ P_θ(k)⋅[1 + (\fr{n}2 (κ-θ)^2(1+O(n(κ-θ)^2+κ))-\fr12 κ-O(κ^2+κ/n))] \\
         ~&=~ P_θ(k)⋅[1±o(1)]
\end{align*}
For $κ,θ≤cn^{-1/2-δ}$ with $c,δ>0$ it implies $P_θ(k)\lessgtr P_λ(k)⋅[1±c'n^{-2δ}]$ for some $c'<∞$.
\end{lemma}

\begin{proof}\vspace{-4ex}
\begin{align*}
  \ln P_λ(k) ~&≡~ k\ln λ-λ-\ln k! ~=~ n(κ\ln θ+κ\ln n-θ)-\ln k! & \text{(by definition)}\\
  \ln P_θ(k) ~&≡~ \ln n!-\ln k!-\ln(γn)! + nκ\ln θ + nγ\ln(1-θ) & \text{(by definition)}\\
  \ln n!/(γn)! ~&=~ nκ\ln\fr{n}{e}-γn\ln γ-\ln\sqrt{γ}-O(κ/γn) & \text{(by 2×Stirling)} \\
  \ln[P_λ(k)/P_θ(k)] ~&=~ nκ\ln n - nθ - \ln[n!/(γn)!] - nγ\ln(1-θ) & \text{(by lines 1\&2)}\\
            ~&=~ n(κ-θ)-nγ\ln\fr{1-θ}{γ}+\ln\sqrt{γ}+O(κ/γn) & \text{(by line 3)} \\
            ~&= \fr{n}{2γ}(κ-θ)^2+\ln\sqrt{γ}+O(n(κ-θ)^3/γ^2)+O(κ/γn) & \text{(by next line)} \\
  \ln\fr{1-θ}{γ} ~&=~ \ln[1+\fr{κ-θ}{γ}] ~=~ \fr{κ-θ}{γ} - \smash{\fr{(κ-θ)^2}{2γ^2} + O(\fr{κ-θ}{γ})^3} &~~~\text{(by Taylor)}
\end{align*}
The last bound in the Lemma follows from exponentiating the previous bound,
Taylor expanding the exponential, 
and noting that even the largest term $n(κ-θ)^2→0$ for $κ,θ=o(n^{-1/2})$ and $1/γ=1+O(κ)$.
\qed\end{proof}

Assuming $k≤c⋅n^{1/2-δ}$ and $θ_x≤c⋅n^{-1/2-δ}~∀x$, 
the lemma implies $P_\vt[E_k^x]\lessgtr P_\vl[E_k^x]⋅[1±c'n^{-2δ}]$ uniformly for all $x$.
Taking the sum over $x∈𝓧$, noting that $M_k^x=⟦N_x=k⟧=⟦\v N∈E_k^x⟧$ and $M_k=∑_x M_k^x$, 
we also have $𝔼_\vt[M_k]\lessgtr 𝔼_\vl[M_k]⋅[1±c'n^{-2δ}]$.
This extends to (positive) linear combinations of $M_k$:

\begin{proposition}[{\bfm $𝔼_\vt[T]≈𝔼_\vl[T]$ and $∑_x\Var_\vt[M_k^x]≈\Var_\vl[M_k]$}]\label{lem:EtvsEl}
For $k_{max}≤c⋅n^{1/2-δ}$ and $θ_x≤c⋅n^{-1/2-δ}~∀x$, 
and random variable $T:=∑_{k≤k_{max}} α_k M_k$ with $α_k≥0$, we have 
$𝔼_\vt[T]\lessgtr 𝔼_\vl[T]⋅[1±c'n^{-2δ}]$. 
In particular, $𝔼_\vt[M_k]\lessgtr 𝔼_\vl[M_k]⋅[1±c'n^{-2δ}]$ and 
$∑_x\Var_\vt[M_k^x]\lessgtr\Var_\vl[M_k]⋅[1±c'n^{-2δ}]$ for $k≤k_{max}$. 
For general $α_k\lessgtr 0$, we have $|𝔼_\vt[T]-𝔼_\vl[T]|≤𝔼_\vl[S]⋅c'n^{-2δ}$, where $S:=∑_k|α_k|M_k$.
\end{proposition}

For instance, for the slope test $D_k=M_k-M_{k-1}$,
the correction is small \emph{relative} to $𝔼_\vt[D_k]$
iff $𝔼[M_k+M_{k-1}]\ll n^δ 𝔼[D_k]$.
If $M_k$ and $M_{k-1}$ and $D_k$ scale linearly in $n$, then this is the case.
We derived upper bounds for $𝔼_θ[T]$ directly in \Cref{secm:iidtests},
so we actually don't need to be concerned about approximation error for expectations.
The stated simple bound on the variance unfortunately does not generalize to $\Var_\vt[T]$,
not even $\Var_\vt[M_k]$, nor $\Var_\vt[∑_k α_k M_k^x]$.
 
\begin{proof} 
The statement for $M_k$ has been derived above.
For $T$ it follows from linearity of the expectation and 
$|P_\vt[E_k^x]-P_\vl[E_k^x]|≤P_\vl[E_k^x]ε'$ derived in \Cref{lem:logPoBi},
where $ε':=c'n^{-2δ}$:
\begin{align*}
  |𝔼_\vt[T]-𝔼_\vl[T]| ~≤~ ∑_x|α_k|⋅|P_\vt[E_k^x]-P_\vl[E_k^x]| ~≤~ ∑_x |α_k|P_\vl[E_k^x]ε' ~=~ ε'𝔼_\vl[S]
\end{align*}
If $α_k≥0~∀k$, then $S=T$. The variance bound can be derived as follows: \Cref{lem:logPoBi} implies
\begin{align*}
  |(1-P_\vt[E_k^x])-(1-P_\vl[E_k^x])| ~=~ |P_\vl[E_k^x]-P_\vt[E_k^x]| ~≤~ ε'P_\vl[E_k^x] ~≤~ ε'(1-P_\vl[E_k^x])
\end{align*}
In the last inequality we exploited $P_\vl[E_k^x]=P_{λ_x}(k)≤P_k(k)≤\fr1{\sqrt{2πk}}≤\fr12$. Hence 
\begin{align*}
  1-P_\vt[E_k^x] ~\lessgtr~ (1±ε')(1-P_\vl[E_k^x])
\end{align*}
For any function $f$ this implies
\begin{align*}
  𝔼_\vt[f(M_k^x)] ~&=~ f(1)P_\vt[E_k^x] + f(0)(1-P_\vt[E_k^x]) \\
  ~&\lessgtr~ |f(1)|P_\vl[E_k^x](1±ε') + |f(0)|(1-P_\vl[E_k^x])(1±ε') ~=~ (1±ε')𝔼_\vl[|f(M_k^x)|]
\end{align*}
Specifically for the function $f_λ(M_k^x):=(M_k^x-μ_λ)^2≥0$, where $μ_λ:=𝔼_\vl[M_k^x]$, with $μ_θ:=𝔼_\vt[M_k^x]$, we get
\begin{align*}
  \Var_\vt[M_k^x] ~&=~ 𝔼_\vt[(M_k^x-μ_θ)^2] ~≤~ 𝔼_\vt[(M_k^x-μ_λ)^2] ~=~ 𝔼_\vt[f_λ(M_k^x)] \\
  ~&≤~ (1+ε')𝔼_\vl[f_λ(M_k^x)] ~=~ (1+ε')\Var_\vl[M_k^x]
\end{align*}
Reversing the role of $λ$ and $θ$ we get $\Var_\vl[M_k^x]≤(1+ε')\Var_\vt[M_k^x]$ in the same way.
Summing both bounds over $x$ we get $∑_x\Var_\vt[M_k^x]\lessgtr(1±ε')∑_x\Var_\vl[M_k^x]$.
The variance bound in the proposition now follows from independence of $M_k^x$ w.r.t.\ $P_\vl$.
\qed\end{proof}

\paragraph{Upper bounding multinomial variances.} 
We were able to derive upper bounds for $\Var_\vl[T]$ (in \Cref{sec:iidtests}) 
by exploiting independence $\Cov_\vl[M_k^x,M_{k'}^{x'}]=0~∀x≠x'$,
but not for $\Var_\vt[T]$, since $\Cov_\vt≠0$. 
We need to show that $\Cov_\vt$ is small.
We do this by approximating $\Cov_\vt≠0$ by $\Cov_\vl=0$.
The approximation error can be determined similarly as we did for the expectation.

Indeed, it can even be reduced to an application of \Cref{lem:logPoBi}:
In addition to the Poisson notation above, 
let $E_{kk'}^{xx'}=\{\v n:n_x+n_{x'}=k+k'\}$. Then for $x≠x'$
\begin{align*}
  & P_\vt[E_k^x∩E_{k'}^{x'}] ~=~ P_\vl[E_k^x∩E_{k'}^{x'}|N=n]
  ~=~ \frac{P_\vl[E_k^x∩E_{k'}^{x'}∩\{N=n\}]}{P_\vl[N=n]}
  ~=~ \frac{P_\vl[E_k^x∩E_{k'}^{x'} ∩E_{\overline{kk'}}^{\overline{xx'}}]}{P_\vl[N=n]} \\
  ~&=~ \frac{P_\vl[E_k^x]⋅P_\vl[E_{k'}^{x'}]⋅P_\vl[E_{\overline{kk'}}^{\overline{xx'}}]}{P_\vl[N=n]}
  ~\stackrel{\eqref{eq:PtPPPl}}=~ P_\vl[E_k^x]P_\vl[E_{k'}^{x'}]\frac{P_\vt[E_{kk'}^{xx'}]}{P_\vl[E_{kk'}^{xx'}]}
  ~≡~ g_k(λ_x)g_{k'}(λ_{x'})\frac{f_{k+k'}(θ_x+θ_{x'})}{g_{k+k'}(λ_x+λ_{x'})}
\end{align*}
Again, one could verify this also by inserting the explicit expressions.
That is, even though $E_k^x$ and $E_{k'}^{x'}$ are not independent under $P_\vt$,
the probability is a ``product'' of 3 Poissons and 1 binomial. The above identity implies
\begin{align*}
  & \Cov_\vt[M_k^x,M_{k'}^{x'}] ~≡~ P_\vt[E_k^x∩E_{k'}^{x'}] - P_\vt[E_k^x]P_\vt[E_{k'}^{x'}] \\
  ~&=~ r_{kk'}(λ_x,λ_{x'})g_k(λ_x)g_{k'}(λ_{x'}) ~~~\text{with}~~~ 
       r_{kk'}(λ_x,λ_{x'})~:=~\left[\frac{f_{k+k'}^n(θ_x+θ_{x'})}{g_{k+k'}(λ_x+λ_{x'})}-
            \frac{f_k^n(θ_x)}{g_k(λ_x)}⋅\frac{f_{k'}^n(θ_{x'})}{g_{k'}(λ_{x'})}\right] \\
  ~&=~ r_{kk'}(θ_x,θ_{x'})f_k^n(θ_x)f_{k'}^n(θ_{x'}) ~~~\text{with}~~~ 
       r_{kk'}(θ_x,θ_{x'})~:=~ \left[\frac{g_k(λ_x)}{f_k^n(θ_x)}
       \frac{g_{k'}(λ_{x'})}{f_{k'}^n(θ_{x'})}
       \frac{f_{k+k'}^n(θ_x+θ_{x'})}{g_{k+k'}^n(λ_x+λ_{x'})} ~-~ 1 \right]
\end{align*}
The first/second expression is more convenient for theoretical/empirical upper bounds.
\Cref{lem:logPoBi} shows that $g_k/f_k^n→1$ for $κ,θ=O(n^{-1/2-δ})$ with $δ>0$,
hence $r_{kk'}$ tends to 0. More precisely 
\begin{align*}
  r_{kk'}(θ,θ') ~&=~ \fr{n}2(κ-θ)^2⋅(1+O(n^{-δ'}))-\fr12 κ ~+~ \fr{n}2(κ'-θ')^2⋅(1+O(n^{-δ'}))-\fr12 κ' \\ 
                ~&~~~-~ \fr{n}2(κ+κ'-θ-θ')^2⋅(1+O(n^{-δ'}))+\fr12(κ+κ') ~+~ O(n^{-δ''}) \\
               ~&=~ -n(θ-κ)(θ'-κ')⋅(1+O(n^{-δ'}))+O(n^{-δ''}) ~=...=~ r_{kk'}(λ,λ')
\end{align*}
where $δ':=\min\{\fr12+δ,2δ\}$ and $δ'':=\min\{1+2δ,\fr32+δ\}$.
If we drop all $O()$-terms we get
\begin{align*}
  \Var_\vt[M_k] ~&=~ \textstyle ∑_x \Var_\vt[M_k^x] + ∑_{x≠x'}\Cov_\vt[M_k^x,M_{k'}^{x'}] \\
  ~&≈~ \textstyle ∑_x \Var_\vt[M_k^x] - ∑_{x≠x'}n(θ_x-κ)(θ_{x'}-κ') f_k^n(θ_x)f_{k'}^n(θ_{x'}) \\
  ~&=~ \textstyle ∑_x \Var_\vt[M_k^x] - n∑_{x,x'}f_k^n(θ_x)(θ_x-κ)f_{k'}^n(θ_{x'})(θ_{x'}-κ') + n∑_x[f_k^n(θ_x)(θ-κ)]^2  \\
  ~&=~ \textstyle ∑_x \Var_\vt[M_k^x] - n[∑_x f_k^n(θ_x)(θ_x-κ)]^2 + n∑_x[f_k^n(θ_x)(θ_x-κ)]^2
\end{align*}
The middle term is negative. Adapting \Cref{lem:ubfe}, 
or a bit more convenient using $f_k^n(θ)≈g_k(λ)$ by \Cref{lem:logPoBi} and \Cref{lem:ubfe} directly,
the last term can be upper bounded as
\begin{align*}
  & n∑_x[f_k^n(θ_x)(θ_x-κ)]^2 ~≈~ \frac1n ∑_x[g_k(λ_x)(λ_x-k)]^2
  ~≤~ \sup_{λ>0} \frac{[g_k(λ)(λ-k)^2]}{λ} \\
  ~&\stackrel{\eqref{eq:gdelta}}=~ \left[\sup_{λ>0} λ^{3/2}\frac{g_δ(λ)}{λ}\right]^2
  ~≈~ \left[(λ_+^*)^{3/2} \sup_{λ>0}\frac{g_δ(λ)}{λ}\right]^2 ~=~ λ_+^*{g_δ(λ_+^*)^2} ~≈~ k g_δ(k)^2
  ~\stackrel{\eqref{eq:Dkub}}{≈}~ \frac1{2πek}
\end{align*}
which is very small compared to the typically linearly in $n$ scaling $\Var_\vt[M_k]$.
Ultimately we need an upper bound in terms of $𝔼[M_k]$, so
\begin{align*}
  \textstyle ∑_x \Var_\vt[M_k^x] ~&=~ \textstyle ∑_x 𝔼_\vt[(M_k^x)^2]-∑_x 𝔼_\vt[M_k^x]^2 ~=~ 𝔼_\vt[M_k]-∑_x f_k^n(θ_x)^2,~~~\text{hence}~~~ \\
  \Var_\vt[M_k] ~&\lesssim~  \textstyle 𝔼_\vt[M_k] + ∑_x[f_k^n(θ_x)^2[n(θ_x-κ)]^2-1]
\end{align*}
The same line of reasoning as above shows that 
\begin{align*}
  \textstyle ∑_x[f_k^n(θ_x)^2[n(θ_x-κ)]^2-1] ~&≈~ k g_δ(k)^2-\fr{n}{k} g_k(k)^2 ~≈~ \fr1{2πk}[\fr1e-\fr{n}{k}] ~≤~ 0
\end{align*}
The arguments above readily extend to linear combinations of $M_k$ (cf.\ \Cref{lem:evcrv}):

\begin{claim}[\bfm ${\Var_\vt[T]\lesssim 𝔼_\vt[T]}$]\label{claim:VtvsVl}
For $k_{max}=o(n^{1/2})$ and $\sup_x θ_x=o(n^{-1/2})$ and  $T:=∑_{k≤k_{max}} α_k M_k$ with $α_k∈ℝ$,
we have $\Var_\vt[T]\lesssim 𝔼_\vt[T]$.
The smaller $k_{max}$ and $θ_x$, the better the accuracy,
The relative error is small under suitable further conditions.
\end{claim}

\paragraph{Law of total variation.} 
Let $T$ be a random variable of interest, e.g.\ one of our test statistics $M_k$ or $U_k$, etc.
Another potential approach towards proving \Cref{claim:VtvsVl} is using the law of total variation:
\begin{align*}
  \Var_\vl[T] ~&=~ 𝔼_\vl[\Var_\vl[T|N]] + \Var_\vl[𝔼_\vl[T|N]] ~=~ \smash{∑_{n'}}\Var_\vl[T|N=n']P_\vl[N=n'] + \Var_\vl[𝔼_\vl[T|N]]
\end{align*}
Note that $\vl=n\vt$ are fixed as before ($\vl≠n'\vt$ unless $n'=n$). 
We know that $P_\vl[N=n']$ has mean and variance $n$ with light tails, 
so concentrates around $n'∈[n±(O\sqrt{n})]$.
If $\Var_\vl[T|N=n']$ does not change much in this interval, 
then the sum can be approximated by $\Var_\vl[T|N=n]≡\Var_\vt[T]$.
This implies
\begin{observation}[\bfm{$\Var_\vt[T]\lesssim\Var_\vl[T]$ or even $\Var_\vt[T]≈\Var_\vl[T]$}]
If $\Var_\vt[T|N=n']$ does not change much for $n'∈[n±(O\sqrt{n})]$, 
then $\Var_\vt[T]\lesssim\Var_\vl[T]$.
If in addition $𝔼_\vl[T|N=n']$ does not change much for $n'∈[n±O(\sqrt{n})]$,
then $\Var_\vt[T]≈\Var_\vl[T]$.
\end{observation}

\par\vspace{0pt plus \textheight}
\begin{samepage}
\section{List of Notation}\label{sec:Notation}

\begin{tabbing}
  \hspace{0.13\textwidth} \= \hspace{0.13\textwidth}\= \hspace{0.60\textwidth} \= \kill
  {\bf Symbol }      \> {\bf Type} \> {\bf Explanation} \\[0ex]
  $a/bc=a/(bc)$      \>                \> while $a/b⋅c=(a/b)⋅c$ though we actually always bracket the latter  \\[0ex]
  $⟦\text{bool}⟧$ \> $∈\{0,1\}$  \> =1 if bool=True, =0 if bool=False \\[0ex]
  $i,j$              \> $∈ℕ$           \> generic indices \\[0ex] 
  $\{i:j\}$          \> $⊂ℤ$          \> set of integers from $i$ to $j$ (empty if $j<i$) \\[0ex] 
  $ℝ,ℝ^+,ℝ^+_0$     \>                \> reals, strictly positive reals, non-negative reals \\[0ex]
  $|𝓧|$            \> $≡\#𝓧$       \> size of set $𝓧$. \\[0ex]
  $x$                \> $∈𝓧$         \> single sample \\[0ex]
  $𝓧$              \>                \> sample space of size $d=|𝓧|$, mostly $d=∞$ and $𝓧$ countable. \\[0ex]
  $𝓧'$             \> $\nq\nq=\{x:θ_x>0\}$ \> all $x$ potentially observable $d'=|𝓧'|$. \\[0ex]
  $𝓧''$            \> $\nq\nq=\{x:n_x>0\}$ \> all $x$ actually observed. $d''=|𝓧''|$ \\[0ex]
  $n$                \>                \> number of samples, sample size \\[0ex]
  $X$                \>                \> $𝓧$-valued random variable \\[0ex] 
  $\v X$             \> $≡X_{1:n}$     \> $n$ iid or exchangeable random variables \\[0ex]
  $\v x≡x_{1:n}$     \> $∈𝓧^n$       \> sample of size $n$ \\[0ex]
  $t$                \> $∈\{1:n\}$     \> sample index \\[0ex]
  $k$                \> $∈ℕ_0$        \> second-order multiplicity index \\[0ex]
  $N_x~=~\#\{X_t:X_t=x\}$ \>           \> (first-order) count=multiplicity of $x$ in $\v X$ \\[0ex]
  $M_k~=~\#\{x:N_x=k\}$ \>             \> (second-order) count=multiplicity of $k$ in $\v N$ \\[0ex]
  $\v M~=~(M_1,M_2,...)$ \>            \> vector of $M_k$ excluding $M_0$, also $M_+:=M_1+M_2+...$ \\[0ex]
  $x,n_x,m_k,\v m,...$ \>              \> realization of random variable $X,N_x,M_k,\v M,...$ \\[0ex]
  $P(x)~:=~ P[X=x]$  \>                \> probability that $X$ is $x$ \\[0ex]
  $P_θ(k)$           \> $≡f_k^n(θ)$     \> $:=({n\atop k})θ^k(1-θ)^{n-k}$ binomial distribution over $ℕ_0$ \\[0ex]
  $P_\vt$            \> $∈\Hiid$       \> iid (multinomial) distribution over $𝓧^n$ ($ℕ_0^𝓧$) \\[0ex]
  $P_λ(k)$           \> $≡g_k(λ)$     \> $:=λ^k e^{-λ}/k!$ Poisson distribution over $ℕ_0$ \\[0ex]
  $P_\vl$            \>                \> product of Poisson($λ_x$) distributions over $\v n∈ℕ_0^𝓧$ \\[0ex]
  $Q$                \> $∈𝓠$         \> exchangeable distribution \\[0ex]
  $Y,Z$              \>                \> generic random variables \\[0ex]
  $𝔼$                \>                \> expectation w.r.t.\ $P_\vt$ or $P_\vl$ unless otherwise noted \\[0ex]
  $σ^2~=~\Var[Z]$    \> $:=~𝔼[Z^2]-𝔼[Z]^2$ \> ~~~~~~~~~~~~~~variance of $Z$ and other random variables\\[0ex]
  $\Cov[Y,Z]$         \> $:=~𝔼[YZ]-𝔼[Y]𝔼[Z]$ \> ~~~~~~~~~~~~~~covariance of $Y$ and $Z$ \\[0ex]
  $ρ~=~\W[Z]$        \> $:=~𝔼[|Z-ζ|^3]$ \> ~~~~~~~~~~~~~~third absolute central moment \\[0ex]
  $O(),Θ()$          \>                \> classical $O()$ notation \\[0ex]
  $O_P()$            \>                \> stochastic $O$-notation \\[0ex]
  $≈$                \>                \> approximately equal, informal \\[0ex]  
  $f(n)\lesssim g(n)$ \>               \> means $f(n)≤g(n)⋅[1+O_P(1/\sqrt{n})]$, and similarly $\gtrsim$ and $\simequd$ \\[0ex]
  $f(n)\lessgtr g(n)⋅[1±ε]$ \>       \> means $f(n)≤g(n)⋅[1+ε]$ and $f(n)≥g(n)⋅[1-ε]$\\[0ex]
  $\dot=$ ~~~ e.g.   \> $z\dot=1.64$   \> equal to within the number of displayed digits \\[0ex]
  $:=,~≡,~\smash{\stackrel!=}$ \>      \> definition, equal by earlier definition, want it to be equal \\[0ex]
  $Z_x$              \>                \> collection of random variables with $x∈𝓧$ \\[0ex]
  $Z_+$              \> $:=∑_x Z_x$    \> sum of random variables \\[0ex]
  $\bar Z$           \> $:=Z_+/n$      \> \emph{not} an average of random variables; also $\bar Y=Y/n$ \\[0ex]
  $ζ$                \> $:=𝔼[Z]$       \> corresponding lower-case greek letters denote expectation \\[0ex]
  $ζ^{ub}$           \> $∈ℝ$           \> upper bound on expectation \\[0ex]
  $V^{ub}$           \>                \> deterministic or stochastic upper bound on variance \\[0ex]
  $\dot{ε}_k,\ddot{ε}_k,...$ \> $∈ℝ_0^+$ \> small corrections $≥0$ tending to $0$ for $k→∞$ \\[0ex]
  $\tilde{ε}_k$      \> $∈ℝ$           \> small correction tending to $0$ for $k→∞$ \\[0ex]
  $T$                \> $:𝓧^n→ℝ$    \> generic test statistic \\[0ex]
  $\tilde T$         \>                \>  uniformized test statistic ($P[\tilde T≤α]=α$) \\[0ex]
  $E,O,M_k,D_k,C_k,\bar U_k$ \>             \>specific test statistics \\[0ex]
  $α~~~~=~P_\vt[T>c_α]$ \>                \> Type~I error, prob.\ of falsely rejecting $\Hiid$, significance level \\[0ex]
  $β(α)~=~Q[T>c_α]$  \>                \> power of test $T$ at level $α$ for $Q$ \\[0ex]
  $\diamondsuit$     \>                \> end of example \& end of notation \& end of paper \\[0ex]
\end{tabbing}
\end{samepage}
 
\end{document}